\documentclass[a4paper, 10pt]{article}

\usepackage[T1]{fontenc}
\usepackage[utf8]{inputenc}

\usepackage{amssymb}
\usepackage{amsmath}
\usepackage{amsfonts}
\usepackage{amsthm}
\usepackage{mathtools}
 
\usepackage{comment}
\usepackage{url}
\usepackage{float}
\usepackage[dvipsnames]{xcolor}
\usepackage{enumerate}
\usepackage{tikz}
\usepackage{caption}
\usepackage{cases}

\newcommand{\R}{\mathbb{R}}

\DeclareMathOperator{\Cov}{Cov} 
\DeclareMathOperator{\sgn}{sgn}

\DeclareMathOperator{\support}{supp}

\theoremstyle{plain}%
\newtheorem{theorem}{Theorem}[section]
\newtheorem{lemma}[theorem]{Lemma}
\newtheorem{proposition}[theorem]{Proposition}
\newtheorem{corollary}[theorem]{Corollary}

\theoremstyle{definition}
\newtheorem{definition}[theorem]{Definition}
\newtheorem{example}[theorem]{Example}

\theoremstyle{remark}
\newtheorem{remark}[theorem]{Remark}

\usepackage{url}
\usepackage{hyperref}%
\hypersetup{
	colorlinks=false,
	linkbordercolor=gray,
	citebordercolor=gray,
	urlbordercolor=gray}%

\title{Divide and color 
representations for threshold Gaussian and stable vectors}
\author{ and Jeffrey E. Steif}

\author{Malin Palö Forsström
\thanks{Chalmers University of Technology and Gothenburg University, Gothenburg, 
Sweden and
KTH Royal Institute of Technology, Stockholm, Sweden.\ \ Email:
        \hbox{malinpf@kth.se}}
\and         Jeffrey E. Steif
\thanks{Chalmers University of Technology and Gothenburg University, Gothenburg, Sweden.\ \ Email:
        \hbox{steif@chalmers.se}}
}

\date{\today}

\begin{document}

\maketitle

\begin{abstract}
We study the question of when a \(\{0,1\}\)-valued threshold process associated to a mean zero Gaussian or a symmetric stable vector corresponds to a  {\it divide and color (DC) process}.  This means that the process corresponding to fixing a threshold level $h$ and letting a 1 correspond  to the variable being larger than $h$ arises from a random partition of the index set followed by  coloring {\it all} elements in each partition element 1 or 0 with probabilities $p$ and $1-p$,  independently for different partition elements.  

While it turns out that all discrete Gaussian free fields yield a DC process when the threshold is  zero, for general $n$-dimensional mean zero, variance one Gaussian vectors with nonnegative  covariances, this is true in general when $n=3$ but is false for $n=4$.

The behavior is quite different depending on whether the threshold level $h$ is zero or not and we show that there is no general monotonicity in $h$ in either direction. We also show that all constant variance discrete  Gaussian free fields with a finite number of variables yield DC processes for large thresholds.

In the stable case, for the simplest nontrivial symmetric stable vector with three variables, we obtain  a phase transition in the stability exponent $\alpha$ at the surprising value of $1/2$; if the index of stability is larger than $1/2$, then the process yields a DC process for large $h$ while if the index of stability is smaller than $1/2$, then this is not the case.

 \newline
 MSC 2010 \emph{subject classifications.}
 Primary 60G15, 60G52
  \medskip\noindent
\end{abstract}

\clearpage

\tableofcontents

\clearpage

\section{Introduction, notation, summary of results and background}\label{s:intro}

\subsection{Introduction}

A very simple mechanism for constructing random variables with a (positive) dependency 
structure is the so-called \emph{divide and color model} introduced in its
general form in \cite{st2017} but having already arisen in many different contexts.

\begin{definition}
A \(\{0,1\}\)-valued process \( X \coloneqq (X_i)_{i \in S} \) is a \emph{divide and color model} 
or \emph{color process} if 
\( X \) can be generated as follows. First choose a random partition $\pi$ of $S$ according to 
some arbitrary distribution, and then independently of this and
independently for different partition elements in the random partition,
assign, with probability \( p \), {\it all} the variables in a partition element
the value \( 1 \) and with probability \( 1-p \) assign 
{\it all} the variables the value \( 0 \). This final \(\{0,1\}\)-valued 
process is then called the \emph{color process} associated to $\pi$ and $p$.
We also say that $(\pi,p)$ is a \emph{color representation} of $X$.
\end{definition}

As detailed in \cite{st2017}, many processes in probability theory
are color processes; examples are the Ising model with zero external field,
the fuzzy Potts model with zero external field, the stationary distributions for the voter Model
and random walk in random scenery. 

While certainly the distribution of the color process determines $p$, it 
in fact does not determine the distribution of $\pi$. This was seen
in small cases in \cite{st2017}, and this lack of uniqueness
was completely determined in \cite{fs2019b}.

Since the dependency mechanism in a color process is so simple, it seems natural
to ask which  \(\{0,1\}\)-valued processes fall into this context. We mention that it is 
trivial to see that any color process has nonnegative pairwise correlations and so 
this is a trivial necessary condition.
In this paper, our main goal is to study the question of which \emph{threshold Gaussian}
and \emph{threshold stable} processes fall into this context. 
More precisely, in the Gaussian situation, we ask the following question. 
Given a set of random variables $(X_i)_{i\in I}$ which is jointly Gaussian with mean zero, and 
given \(h \in \mathbb{R} \), is the \(\{0,1\}\)-valued process $(X^h_i)_{i\in I}$ defined by
$$
X^h_i\coloneqq I(X_i > h)
$$
a color process? In the stable situation, we simply replace the Gaussian assumption by
$(X_i)_{i\in I}$ having a symmetric stable distribution. (We will review the necessary 
background concerning stable distributions in Subsection~\ref{ss:stablebackground}.)
For the very special case that $I$ is infinite, $h=0$ and the process is exchangeable, 
this question was answered positively, both in the Gaussian and stable cases, in \cite{st2017}.
The set of threshold stable vectors is a much richer class than the set of 
threshold Gaussian vectors. As such, it is reasonable to study both classes.

Since all the marginals in a color process are necessarily equal, if $h\neq 0$, then a 
necessary condition in the Gaussian case for $(X^h_i)_{i\in I}$ to be a color 
process is that all the $X_i$'s have the same variance.
Therefore, when considering $h\neq 0$, we will assume that all the $(X_i)$'s have variance one.
However, it will be convenient not to make this latter assumption when considering $h=0$.
For the stable case, we will simply assume that all the marginals are the same.

It has been seen in \cite{st2017} that  $p=1/2$ and $p\neq 1/2$ (corresponding to 
$h=0$ and $h\neq 0$ in the Gaussian setting) behave very differently generally speaking.
This was also seen in \cite{BMMU} and we will continue to see this here.

We finally note that the questions looked at here significantly differ from
those studied in \cite{st2017}. In the latter paper, one looked at what types of behavior
(ergodic, stochastic domination, etc.) color processes possess while in the present
paper, we analyze which random vectors (primarily among threshold Gaussian and threshold stable
vectors) are in fact color processes.

\subsection{Notation and some standard assumptions}

Given a set $S$, we let $\mathcal{B}_S$ denote the collection of partitions of the set $S$.
We denote \(\{1,2,3, \ldots, n\}\) by $[n]$ and if \(S =[n] \), we write $\mathcal{B}_n$
for $\mathcal{B}_S$. \( |\mathcal{B}_n| \) is called the \(n \)th Bell number. We denote by
$\mathcal{P}_n$ the set of partitions of the {integer} $n$.

A random partition of $[n]$ yields a probability vector 
$q=\{q_\sigma\}_{\sigma\in \mathcal{B}_n}\in \mathbb{R}^{\mathcal{B}_n}$. 
Similarly, a random \(\{0,1\}\)-valued vector $(X_1,\ldots,X_n)$
yields a probability vector 
$\nu=\{\nu_\rho\}_{\rho\in \{ 0,1 \}^n}\in \mathbb{R}^{\{ 0,1 \}^n}$. The definition of 
a color process yields immediately, for each $n$ and $p\in [0,1]$, an affine map 
$\Phi_{n,p}$ from random partitions of $[n]$,
i.e., from probability vectors $q=\{q_\sigma\}_{\sigma\in \mathcal{B}_n}$
to probability vectors $\nu=\{\nu_\rho\}_{\rho\in \{ 0,1 \}^n}$.
This map naturally extends 
to a linear mapping $A_{n,p}$ from $\mathbb{R}^{\mathcal{B}_n}$ to $\mathbb{R}^{\{ 0,1 \}^n}$.
The image of $A_{n,p}$ was determined in \cite{fs2019b}. Loosely speaking, for $p\neq 1/2$,
the image is the set of signed measures with marginal $p$, and, for $p= 1/2$,
the image is the set of signed measures which have a $\{ 0,1 \}$-symmetry. In many cases,
we will have  a signed measure mapping to our given process and the work involves showing
that this signed measure is in fact a probability measure, telling us that the process is a 
DC process. A signed measure mapping to a given process in this way is called a \emph{formal solution}, or a \emph{signed color representation}.

While perhaps not standard terminology, we  call a Gaussian vector {\sl standard} if each marginal has mean zero and variance one.

\paragraph{Standing assumption.} Whenever we consider a Gaussian or symmetric stable vector, 
we will assume it is nondegenerate in the sense that for all $i\neq j$, $P(X_i\neq X_j)=1$.  

Some further notation which we will use is the following.

\begin{description}

\item[\(\nu_{x_1,\ldots,x_n}\) \textnormal{or} \(\nu(x_1,\ldots,x_n)\)]\textcolor{white}{.}\\
will denote the probability that $\{X_1=x_1,\ldots,X_n=x_n\}$ for a \(\{0,1\}\)-valued
process $(X_1,\ldots,X_n)$.


\item[\( \nu_{(x_1,\ldots, x_n)}(h)\) \textnormal{or} \( \nu_h(x_1,\ldots, x_n) \) ]\textcolor{white}{.}\\
 will denote, given a Gaussian or stable vector $(X_1,\ldots,X_n)$,
the probability that the $h$-threshold process is equal to \( (x_1,\ldots, x_n) \); i.e., the probability that \( P(X_i^h = x_i; \, i\in [n])\). We use \( \nu_h \) to denote the corresponding probability measure on \( \{ 0,1 \}^n \).

\item[\(q_{13,2}\)]\textcolor{white}{.}\\
as an illustration, will denote, given a random partition with $n=3$, the probability that
$1$ and $3$ are in the same partition and $2$ is in its own partition. 

If we have a partition of a set of more than three elements, \(q_{13,2}\) will then mean the above but
with regard to the induced (marginal) random partition of $\{1,2,3\}$.

\item[\(N(0,A)\)]\textcolor{white}{.}\\
will denote a Gaussian vector with mean zero and covariance matrix $A$.

\end{description}

When a threshold \( h \not = 0 \) we will in general only state results for \( h > 0 \). However, since \( X^h = (X^h_i)_i = 1 - X^{-h} \), the analogous results for \( h > 0 \) follows.

\subsection{Description of results}

In Section~\ref{s: stieltjes}, we present positive results concerning the question of the existence 
of a color representation for the threshold zero case for discrete Gaussian free fields and more 
generally for Gaussian vectors whose covariance matrices are so-called {\sl inverse Stieltjes}, 
meaning that the off-diagonal elements of the inverse covariance matrix are nonpositive. 
This essentially follows from the known fact that the distribution of the signs of a discrete Gaussian free field (DGFF),
conditioned on their absolute values, is that of an Ising Model with nonnegative interaction constants 
depending on the conditioned absolute values. The latter fact has been observed 
in \cite{lw2015}. However, it turns out that a threshold zero Gaussian process can be 
a color process even if its covariance matrix is not inverse Stieltjes. We also relate the class of 
\emph{inverse Stieltjes} vectors with the set of tree-indexed Gaussian Markov chains. 

In Section~\ref{s: embedding}, we provide an alternative proof that threshold zero
tree-indexed Gaussian Markov chains are color processes using the Ornstein-Uhlenbeck process.
This proof has the advantage that the method leads to our first result for stable vectors,
namely that a threshold zero tree-indexed symmetric stable Markov chain is also a color process; in 
this case, we use subordinators.

In Section~\ref{s: geometric}, we view our Gaussian vectors from a more geometric perspective
and obtain a number of negative (and some positive) results for thresholds $h\neq 0$. 
In this section, we will obtain our first example where we have a nontrivial phase transition in $h$.
This will be elaborated on in more detail in Theorem~\ref{theorem: 4pointsoncircle} but we state 
perhaps what is the main import of that result.

\begin{theorem}\label{theorem:pt}
There exists a four-dimensional standard Gaussian vector $X$ so that $X^h$ is a color process
for small positive $h$ but is not a color process for large $h$.
\end{theorem}

\begin{remark}
Given the above it is natural to ponder over the possible monotonicity properties in $h$. 
Proposition~\ref{proposition: nonzero h on a circle} implies
that there is no three-dimensional Gaussian vector
with such a phase transition among those that are not fully supported,
while simulations indicate that there is also no fully supported
three-dimensional Gaussian vector with such a phase transition. On the other hand, 
Corollary~\ref{corollary:4examples}(iii) tells us that there are three-dimensional Gaussian vectors
which are not color processes for small $h$ but are color processes for large $h$.
This together with the previous result rules out any type of monotonicity,
in either direction. Perhaps however monotonicity holds (in one direction) for
fully supported vectors.  
\end{remark}

Returning to the threshold zero case, we recall that Proposition 2.12 in \cite{st2017}
implies that for any three-dimensional Gaussian vector with nonnegative correlations, 
the corresponding zero threshold process is a color process. Our next result says that 
this is not necessarily the case for four-dimensional Gaussian vectors.

\begin{theorem}\label{theorem:4d0threshold}
There exists a four-dimensional standard Gaussian vector $X$ 
with nonnegative correlations so that $X^0$ is not a color process.
$X$ can be taken to either be fully supported or not.
\end{theorem}

In Subsection~\ref{section: integrals}, we extend the study of the example given in the proof of
the previous theorem to the stable case.
 
In Section~\ref{section:nonzerogaussian}, we consider the large $h$ Gaussian case. We show that any 
Gaussian vector which is not fully supported does not have a color representation for large 
$h$; see Corollary~\ref{corollary: n=2 and h large}. On the other hand, we have the following.
\begin{theorem}\label{theorem: strict dgff and large h}
If \( X \coloneqq (X_1,X_2, \ldots, X_n) \) is a discrete Gaussian free field which is standard
Gaussian, then \( X^h \) is a color process for all sufficiently large \( h \).
\end{theorem}
For the definition of the discrete Gaussian free field see, for example,~\cite{dlp}.
We do not know if there is any DGFF \( X \)  with constant variance for which \( X^h \)   is   not a color process for some \( h \).

In Section~\ref{section:smalllarge}, we obtain detailed results concerning the existence of a color
representation when the threshold $h\to 0$ and when $h\to \infty$ in the general Gaussian case when $n=3$.
In the fully supported case, we have the following result which gives an exact characterization
of which Gaussian vectors have a color representation for large $h$. 
Note that if two of the covariances are zero, then we trivially have a color representation for 
all $h$. 

\begin{theorem}\label{theorem: Gaussian critical 3d}
Let \( X  \) be a fully supported three-dimensional standard Gaussian vector with 
covariance matrix \( A = (a_{ij}) \) satisfying
\( Cov(X_i,X_j) = a_{ij} \in [0,1) \) for \( 1 \leq i < j \leq 3 \).
If \( a_{ij} > 0 \) for all \( i<j \), then \( X^h \) has a color representation for sufficiently large \(h \) 
if and only if one of the following (nonoverlapping) conditions holds.
\begin{enumerate}[(i)]
\item \( \mathbf{1}^T A^{-1} > \mathbf{0} \) 
\item \( \min_i \mathbf{1}^T A^{-1}(i) = 0 \)   
\item \( \min_i \mathbf{1}^T A^{-1}(i) < 0 \) and \( \mathbf{1}^T A^{-1} \mathbf{1} < 2 \).
\end{enumerate}  
Furthermore, if exactly one of the covariances is equal to zero, then \( X^h \) does not have a color representation 
for large \( h \).
\end{theorem}

The assumption in (i) of Theorem~\ref{theorem: Gaussian critical 3d}, i.e.\  that 
\( \mathbf{1}^T A^{-1} > \mathbf{0} \), is sometimes called the \emph{Savage condition} (with respect
to the vector $\mathbf{1} = (1,1,\ldots,1)$). When \( A = (a_{ij}) \) is the covariance matrix of a 
(nontrivial) two-dimensional standard Gaussian vector, then 
\(\mathbf{1}^T A^{-1}(1) = \mathbf{1}^T A^{-1}(2) =   (1+a_{12})^{-1}>0 \), and hence the Savage 
condition always holds in this case.  If \( A = (a_{ij}) \) is the covariance matrix of a 
three-dimensional standard Gaussian vector, then one can show that
\begin{equation}\label{eq: Savage condition term}
\mathbf{1}^T A^{-1}(1) =  \frac{(1+a_{23}-a_{12}-a_{13})(1-a_{23})}{\det A}
\end{equation}
and it follows that the Savage condition holds if and only if 
\begin{equation}\label{eq: Savage condition term again}
1+ 2\min_{i<j} a_{ij} > \sum_{i<j} a_{ij}.
\end{equation}
When \( \mathbf{1}^T A^{-1} \ge \mathbf{0} \), we will refer to this as the
\emph{weak Savage condition}. This for example holds for all discrete Gaussian free fields.

The rest of the results we describe in this section concern the stable (non-Gaussian) case.  
In Section~\ref{section:nonzerostable}, we first look at the case $n=2$. While it is trivial 
that having a color representation is equivalent to having a nonnegative correlation when $n=2$,
in the stable case it is not obvious, even when $n=2$,  which spectral measures yield a threshold 
vector with a nonnegative correlation. This contrasts with the Gaussian case where nonnegative 
correlation in the threshold process is simply equivalent to the Gaussian vector having a 
nonnegative correlation. 

We first mention, in this regard, that Theorem~4.6.1 (and its proof) and
Theorem~4.4.1 in~\cite{st1994} (see also (4.4.2) on p.\ 188 there) yield the following fact
where \( S_\alpha(1,0,0) \) denotes the standard one-dimensional symmetric $\alpha$-stable 
distribution with scale one; see the next subsection for precise definitions.
For \( \alpha \in (0,2) \), if \( X \) is a symmetric 2-dimensional \( \alpha \)-stable random 
vector with marginals \( S_\alpha(1,0,0) \) spectral measure \( \Lambda \), then
(1) if \( \Lambda \) has support only in the first and third quadrants,
then \( X_1^{h} \) and \( X_2^{h} \) are nonnegatively correlated for all \( h \in \mathbb{R} \)
(and hence the threshold process is a color process) and (2)
if \( \Lambda \) has some support strictly inside the first quadrant, then 
\( X_1^{h} \) and \( X_2^{h} \) have strictly positive correlation for all sufficiently large \( h \)
(and hence the threshold process is a color process for large $h$).

The following natural example shows that one does not need to
have the spectral measure supported only in the first and third quadrants in order for 
the threshold process always to be a color process. 

\begin{proposition}\label{2dpos.corr}
Let \( S_1,S_2 \sim S_\alpha(1,0,0) \) be independent and let \(a \in (0,1) \). Set 
\[
\begin{cases}
X_1 = a S_1 + (1-a^\alpha)^{1/\alpha} S_2 \cr
X_2 = -a S_1 + (1-a^\alpha)^{1/\alpha} S_2.
\end{cases}
\]
(This ensures that \( X_1,X_2 \sim S_\alpha(1,0,0) \).) Then the following are equivalent. 
\begin{enumerate}[(i)]
\item \( a \leq 2^{-1/\alpha} \) 
\item \( X^0 \) is a color process.
\item  \( X^h \) is a color process for all \( h \).
\end{enumerate}
\end{proposition}

We now study the question of the existence of a color representation in the symmetric stable case 
when $h\to \infty$. Our first result shows that there is a fairly large class for which the answer 
is affirmative and here the method of proof comes from that used in 
Theorem~\ref{theorem: strict dgff and large h}.

\begin{theorem}\label{theorem:stablegoodsupport}
Let $X$ be a symmetric stable distribution with marginals \( S_\alpha(1,0,0) \)
whose spectral measure has some support properly inside 
each orthant. Furthermore, assume that
\begin{equation}\label{eq: stablegood}
2 \int_{\mathbb{S}^{n-1}} (\mathbf{x}^{(2)} \lor 0)^\alpha \, d\Lambda(\mathbf{x}) < 1
\end{equation}
where \(\mathbf{x}^{(2)}\) denotes the second largest coordinate of the vector $\mathbf{x}$.
Then \( X^h \) is a color process for all sufficiently large \( h \).  
\end{theorem}

The integral condition in~\eqref{eq: stablegood} will hold for example if the spectral measure is supported sufficiently 
close to the coordinate axes.

Next, we surprisingly obtain, in the simplest nontrivial stable vector with $n=3$, a certain phase 
transition in the stability exponent where the critical point is $\alpha=1/2$. We state it here 
although relevant definitions will be given later on. 

\begin{theorem}\label{theorem:ptalpha12}
Let \( \alpha \in (0,2) \) and let \( S_0 \), \( S_1 \), \( S_2 \), \( S_3 \) be i.i.d.\ each with distribution 
\( S_\alpha(1,0,0) \). Furthermore, let \( a \in (0,1) \) and for \( i = 1, 2, 3 \), define 
\[
X_i = aS_{0} + (1-a^\alpha)^{1/\alpha} S_i
\] 
and \(X_\alpha \coloneqq (X_1,X_2,X_3) \). (\( X_\alpha \) is  then a symmetric $\alpha$-stable 
vector which is invariant under permutations; it is one of the simplest such vectors other than an 
i.i.d.\ process.)
\begin{enumerate}[(i)]
\item If \( \alpha > 1/2\), then \( X^h \) is a color process for all sufficiently large \( h \).  
\item If \( \alpha < 1/2 \), then \( X^h \) is not a color process for any sufficiently large \( h \).
\end{enumerate}
\end{theorem}

The critical value of $1/2$ above was independent of the parameter $a$, as long as
\(a\in (0,1) \). If we however move to a family which has two parameters, but is still  \( \{ 0,1\}\)-symmetric and permutation transitive,
we can obtain a phase transition at any point in \( (0,2) \).

\begin{theorem}\label{theorem: alternative symmetric example}
Let \( a,b \in (0,1) \) satisfy \( 2a^2 + 2b^2 <1 \). Let $c_1=c_1(a,b)\in (0,2)$
be the unique solution to \( 2a^{c_1} + 2b^{c_1} =1 \) and $c_2=c_2(a,b):= \log{2}/|\log a -  \log b|\in (0,\infty]$. 

For \( \alpha \in (c_1,2) \), let \( S_1 \), \( S_2 \), \ldots, \( S_7 \) be i.i.d.\ 
with \( S_1 \sim S_\alpha(1,0,0) \) and define
\[
\begin{cases}
X_1 \coloneqq  aS_1 + bS_2 +bS_4 + aS_5 + (1-2a^\alpha-2b^\alpha)^{1/\alpha}S_7 \cr
X_2 \coloneqq  aS_2+bS_3 + bS_5 + aS_6 + (1-2a^\alpha-2b^\alpha)^{1/\alpha}S_7 \cr
X_3 \coloneqq  bS_1 + aS_3 + aS_4 + bS_6 + (1-2a^\alpha-2b^\alpha)^{1/\alpha}S_7 .\cr
\end{cases}
\]
Then \(X_\alpha \coloneqq (X_1,X_2,X_3) \) is  a symmetric $\alpha$-stable vector which is invariant 
under all permutations, and the following holds.
\begin{enumerate}[(i)]
\item If $c_2\le c_1$, then, for all $\alpha\in (c_1,2)$,
\( X_\alpha^h \) is a color process for all sufficiently large \( h \).  
\item If $c_2\ge 2$, then, for all $\alpha\in (c_1,2)$,
\( X_\alpha^h \) is not a color process for any sufficiently large \( h \).  
\item If $c_2\in (c_1,2)$, then, for all $\alpha\in (c_1,c_2)$,
\( X_\alpha^h \) is not a color process for any sufficiently large \( h \)
while for all $\alpha\in (c_2,2)$,
\( X_\alpha^h \) is a color process for all sufficiently large \( h \).
\end{enumerate}
In particular, for any $\alpha_c\in (0,2)$ and $\epsilon < \alpha_c$,
we can choose $a$ and $b$ so that $c_1=\epsilon$ and $c_2=\alpha_c$,
in which case $X_\alpha$ is defined for all $\alpha\in (\epsilon,2)$ and
where the question of whether the large $h$ threshold is a color process
has a phase transition at $\alpha_c$.
\end{theorem}

\begin{remark}
The case \( a > b = 0 \), which is not included in  Theorem~\ref{theorem: alternative symmetric example}, corresponds to the fully 
symmetric case studied in Theorem~\ref{theorem:ptalpha12}.
\end{remark}

\subsection{Background on symmetric stable vectors}\label{ss:stablebackground}

We refer the reader to \cite{st1994} for the theory of stable distributions
and will just present here the background needed for our results.

\begin{definition}
A random vector $X \coloneqq (X_i)_{1\le i\le d}$ in $\mathbb{R}^d$ has a \emph{stable} distribution
if for all $n$, there exist $a_n>0$ and $b_n$ so that if 
$(X^1,\ldots, X^n)$ are $n$ i.i.d.\ copies of $X$, then
$$
\sum_{1\le i\le n} X^i \overset{\mathcal D}{=} a_n X+b_n.
$$
\end{definition}

It is known that for any stable vector, there exists $\alpha\in (0,2]$ so that 
$a_n=n^{1/\alpha}$. The Gaussian case corresponds to $\alpha=2$. Ignoring constant random variables,
a stable random variable (i.e., with $d=1$ above) has four parameters, (1) $\alpha\in (0,2]$ which
is called the \emph{stability exponent}, (2) $\beta\in [-1,1]$ which is called the 
 \emph{asymmetry parameter}, (3) $\sigma$ which is a \emph{scale parameter} and
(4) $\mu$ which is a \emph{shift parameter}. When $\alpha=2$, there is no $\beta$ parameter, $\mu$
corresponds to the mean and $\sigma$ corresponds to the standard deviation divided by $\sqrt{2}$, 
an irrelevant scaling.
The distribution of this random variable is denoted by $S_\alpha(\sigma,\beta,\mu)$.
More precisely, $S_\alpha(\sigma,\beta,\mu)$ is defined by its characteristic function 
$f(\theta)$, which, for $\alpha \neq 1$ is
$$
\exp\left(-\sigma^\alpha|\theta|^\alpha(1-i\beta(\sgn \theta)\tan({\pi\alpha}/{2}))+i\mu\theta\right).
$$
See \cite{st1994} for the formula when $\alpha =1$. One should be careful and keep in
mind that different authors use different parameterizations for the family of stable
distributions. Throughout this paper, we will only consider symmetric stable 
random variables corresponding to $\beta=\mu=0$ and sometimes often assume $\sigma=1$. 
The above then simplifies to a random variable having distribution 
$S_\alpha(\sigma,0,0)$ which means its characteristic function is 
$f(\theta)=e^{-\sigma^\alpha |\theta|^\alpha}$. In the symmetric case, this formula 
is also valid for $\alpha=1$.

Finally, a random vector in $\mathbb{R}^d$ has a symmetric stable distribution with stability 
exponent $\alpha$ if and only if its characteristic function $f(\theta)$ has the form 
$$
f(\theta)=\exp(-\int_{\mathbb{S}^{d-1}} |\theta\cdot \mathbf{x}|^\alpha \, d\Lambda(\mathbf{x}))
$$
for some finite measure $\Lambda$ on $\mathbb{S}^{d-1}$ which is invariant under 
$\mathbf{x}\mapsto -\mathbf{x}$.  $\Lambda$ is called the {\it spectral measure} corresponding to
the $\alpha$-stable vector.
For $\alpha\in (0,2)$ fixed, different $\Lambda$'s yield different distributions. 
This is not true for $\alpha=2$.

In a number of cases, we will have a symmetric $\alpha$-stable vector 
$X \coloneqq (X_1,\ldots,X_d)$ which is obtained by having
$$
X =  A (Y_1,\dots,Y_m)
$$
where $A$ is a $d\times m$ matrix and $Y=(Y_1,\dots,Y_m)$ are i.i.d.\ random variables with distribution
$S_\alpha(1,0,0)$. In such a case, there is a simple formula for the spectral measure \( \Lambda \) for $X$.
Consider the columns of $A$ as elements of \(\mathbb{R}^d \), denoted by $\mathbf{x}_1,\ldots,\mathbf{x}_m$.
Then $\Lambda$ is obtained by placing, for each $i\in [m]$, a mass of weight
$ {\| \mathbf{x}_i\|_2^\alpha}/{2}$ at $\pm \mathbf{x}_i/\| \mathbf{x}_i\|_2$. See p.\ 69 in \cite{st1994}.

\section{Stieltjes matrices and discrete Gaussian free fields}
 \label{s: stieltjes}

\subsection{Inverse Stieltjes covariance matrices give rise to color processes for \texorpdfstring{$h=0$}{h=0}}

\begin{definition}
A Stieltjes matrix is a symmetric positive definite matrix with non-positive off-diagonal elements. 
\end{definition}

We will see later that the following result implies that for all 
discrete Gaussian free fields $X$, $X^0$ is a color process.

\begin{theorem}\label{theorem: Ising representation}
If \( X \sim N(0,A) \) and \( A^{-1} \) is a Stieltjes matrix, then \( X^{0}\) is a color process. 
\end{theorem}
 

In~\cite{lw2015}, it was observed that the signs of a discrete Gaussian free field
is an average of ferromagnetic Ising Models; that argument extends to the case
of a Stieltjes matrix which is given below.

\begin{proof}
Note first that as \( (b_{ij}) \coloneqq  A^{-1} \) is a Stieltjes matrix, we have that \( b_{ij} \leq 0 \) whenever 
\( i \not = j \). This implies in particular that if \( f_X \) is the probability density function 
of \( X \), then
\begin{align*}
f_X(\mathbf{x}) &\propto \exp\left(\frac{-\mathbf{x}^T A^{-1} \mathbf{x}}{2}\right) 
 = \exp\left( \sum_{\{ i,j \}} -b_{ij} \mathbf{x}_i\mathbf{x}_j - \frac{1}{2}\sum_i b_{ii} \mathbf{x}_i^2\right) .
\end{align*}

Now for each \( i\), define \( \sigma_i \coloneqq \sgn X_i \) so that \( X_i = |X_i| \sigma_i \). 
Then the conditional probability density function  of \( ( \sigma_i) \) given 
\( |X_1| = y_1 \), \( |X_2 |  =y_2\), \ldots, \( |X_n| =y_n\) satisfies 
\begin{align*}
f(\mathbf{\sigma}) &\propto 
 \exp\left( \sum_{ \{ i,j \}} -b_{ij} y_i  y_j  \sigma_i \sigma_j \right) .
\end{align*}
This is a ferromagnetic Ising model  with parameters 
\( \beta_{ij} = -b_{ij} y_i y_j \ge 0 \) and 
no external field. It is well known that the (Fortuin Kastelyn) random cluster model yields a 
color representation for the Ising model after we identify $-1$ with
$0$. Since an average of color processes is a color process, we are done. 
\end{proof}

\begin{remark}
The proof of Theorem~\ref{theorem: Ising representation} does not apply to
other threshold levels. With nonzero thresholds, this argument would
lead to Ising model with a varying external field. 
The marginals of this (conditioned) process are not in general equal, 
which precludes it from being a color process, and even if the marginals 
were equal, there is no known color representation in this case in general.

\end{remark}

We end this subsection by pointing out that there
are fully supported Gaussian vectors whose threshold zero processes
are color processes but whose inverse covariance matrix is not 
a Stieltjes matrix.

To see this, let \( a \in (0,1 ) \) and \( \varepsilon \in (0,1)\). Then the matrix

\[A=
\begin{pmatrix}
1&a&a \\
a&1&a^2-\varepsilon\\
a& a^2-\varepsilon&1
\end{pmatrix}
\]
has eigenvalues \(1-a^2+\varepsilon \) and \( \frac{2+a^2-\varepsilon \pm \sqrt{8a^2+(a^2-\varepsilon)^2}}{2}\). 
Hence \( A \) is positive definite if \( \varepsilon < 1-a^2 \).
Moreover, we have
\[
A^{-1} = \frac{1}{1-a^2-\varepsilon}
\begin{pmatrix}
1+a^2-\varepsilon & -a & -a \\
-a & \frac{1-a^2}{1-a^2+\varepsilon} & \frac{\varepsilon}{1-a^2+\varepsilon} \\
-a & \frac{\varepsilon}{1-a^2+\varepsilon} & \frac{1-a^2}{1-a^2+\varepsilon}
\end{pmatrix}
\]
Hence, \( A \) is not an inverse Stieltjes matrix for any \( \varepsilon > 0\), since for any \(\varepsilon > 0 \) we have that  
  \( A^{-1}(2,3) > 0 \). 
Consequently, if \( 0<\varepsilon<1-a^2 \), 
then \( A \) is symmetric, positive and positive definite but  not
 an inverse Stieltjes matrix.
Finally, the fact that the threshold zero process is a color process follows from 
Proposition 2.12 in \cite{st2017} which states that for $n=3$,
any \( \{ 0,1 \} \)-symmetric process with nonnegative pairwise correlations
is a color process.

A very important class of Gaussian vectors that have 
\( A^{-1} \) being a Stieltjes matrix are discrete Gaussian free fields with a finite
number of variables.
Another example are so-called tree-indexed Gaussian Markov chains. A Gaussian Markov chain with parameter 
$a\in [0,1]$ has state space $S=\R$ and is described by 
$s\mapsto as + (1-a^2)^{ {1}/{2}} W$ where $W$ is a standard normal random variable; this is 
reversible with respect the distribution of $W$. From this, one can construct tree-indexed
Gaussian Markov chains (see e.g.\ \cite{bp1994}).

We end this subsection by discussing a simple Gaussian vector 
and show that different points of view can lead to very different
color representations. To this end, consider the fully symmetric multivariate normal 
\( X \coloneqq (X_1,X_2,\ldots, X_n) \) with covariance matrix $A=(a_{ij})$ where
\( a_{ij} = a\in (0,1) \) for \( i \not = j \) and \( a_{ii} = 1\) for all \( i \). 
It is easy to check that
\[
A^{-1}(i,j) = \begin{cases} 
\frac{1+(n-2)a}{(1+(n-1)a)(1-a)} &\textnormal{ if } i = j  \cr
\frac{-a}{(1+(n-1)a)(1-a)}& \textnormal{ otherwise.}  \end{cases}
\]
Since this is a Stieltjes matrix, \( X^0 \) is a color process by 
Theorem~\ref{theorem: Ising representation} and moreover, by the proof, the 
resulting color representation has full support.
(The fact that this particular example is a color process
is also covered by Section 3.5 in \cite{st2017} using a different method.)

Now suppose we would add a variable \( X_0 \) with \( a_{00} = 1 \) 
and \( a_{i0} = \sqrt{a} \) for all \(i\in\{1,2,\ldots,n\} \). 
One can check that this defines a Gaussian vector \( (X_0,X_1,X_2, \ldots, X_n ) \)
and it is easy to check that this is a tree-indexed Gaussian Markov chain where the
tree is a vertex with $n$ edges coming out. If we let \( A_0 \) be 
the covariance matrix of \(Y\coloneqq  (X_0,X_1,X_2, \ldots, X_n ) \), then its inverse is given by
\[
A_0^{-1}(i,j) = \begin{cases} 
\frac{1+(n-1)a}{1-a} &\textnormal{ if } i = j =0 \cr
\frac{1}{1-a} &\textnormal{ if } i = j >0 \cr
\frac{-\sqrt{a}}{1-a} &\textnormal{ if } i >j = 0 \cr
\frac{-\sqrt{a}}{1-a}&\textnormal{ if } j>i = 0 \cr
0& \textnormal{ otherwise}.  
\end{cases}
\]
Being a Stieltjes matrix, \( Y^0 \) has a color 
representation by Theorem~\ref{theorem: Ising representation} and the proof yields that
if we restrict the resulting color representation of \( Y^{0} \) to
\( \{ 1,2,\ldots, n \} \), the representation is supported on partitions with at most 
one non-singleton cluster. 
In particular, this implies that when \( n = 4 \), these color representations will assign 
different probabilities to the partition \( (12,34) \), and hence the representations are distinct.

\section{An alternative embedding proof for tree-indexed Gaussian 
Markov chains which extends to the stable case}\label{s: embedding}

The purpose of this section is twofold: first to give an
alternative proof of the fact established earlier that 
tree-indexed Gaussian Markov chains are color processes
and then to use a variant of this alternative method to obtain a result
in the context of stable random variables.

\subsection{The Gaussian case} 
\begin{proof}[Alternative proof that the threshold zero of a tree-indexed Markov chain
is a color process]

We give this proof only for a path where the correlations between
successive variables are the same value $a$. The extension to the tree case 
and varying correlations is analogous.

To show that \( X \coloneqq (X_1,X_2, \ldots, X_n) \) has a 
color representation for any \( n \geq 1 \), we want to construct,
on some probability space, a random partition $\pi$
of $[n]$ and random variables \( Y=(Y_1,Y_2, \ldots, Y_n) \) so that
\begin{enumerate}[(i)]
\item \( X \) and $Y$ have the same distribution (which implies
that their corresponding sign processes have the same distribution) and
\item  $(Y^{0},\pi)$ is a color process (for $p=1/2$) with its color representation.
\end{enumerate}

To do this, let \( (Z_t) \) be the so-called Ornstein-Uhlenbeck (OU) process defined by
\[
Z_t = e^{ -t } W_{e^{2t}}
\]
where \( (W_t)_{t \geq 0} \) is a standard Brownian motion.
It is well known and immediate to check that
\( Z_t \sim N(0,1) \) for any \( t \in \mathbb{R} \) 
and that \( \Cov(Z_s,Z_t) = e^{-|s-t|} \) for any \( s,t \in \mathbb{R}\). 

Now, given $n$, consider the random vector $Y$ given by
\[
\bigl(Z_{\log(1/a)},Z_{2\log(1/a)},\ldots,Z_{n\log(1/a)}\bigr)
\]
and consider the random partition $\pi$ of $\{1,2,\ldots,n\}$ given by
$i\sim j$ if $Z_t$ does not hit zero between times
$i\log(1/a)$ and $j\log(1/a)$. 

It is immediate from the Markovian structure of both vectors and the covariances in the OU 
process that (i) holds. Next, (ii) is clear using the reflection principle 
(which uses the strong Markov property) and the fact that the hitting time of 0 is a stopping time.
\end{proof}

\begin{remark}
This argument (also) does not work for any threshold other than zero. For it to work, one would 
need that for $h>0$ and any time $t\ge 0$, the probability that an OU process started at $h$ is
larger than $h$ at time $t$ is equal to the unconditioned probability. This however does not hold.
\end{remark}

\begin{remark}
In~\cite{tl2016}, the author studies a similar construction as the construction above for discrete 
Gaussian free fields. More precisely, the author shows that one can obtain a color representation for 
a DGFF \( X \) as follows. Given \( X \), for each pair of adjacent vertices he
adds a Brownian bridge with length determined by their coupling constant. Two vertices are then put 
in the same partition element if the corresponding Brownian bridge does not hit zero. Since DGFF's have no 
stable analogue, this does not generalize to any class of stable distributions.
\end{remark}

\subsection{The stable case}
We now obtain our first result for stable vectors. Given $\alpha\in (0,2)$ and \(a\in (0,1)\),
let $U$ have distribution \( S_\alpha(1,0,0) \) and consider the Markov chain on $\R$ given by
$s\mapsto as + (1-a^\alpha)^{1/\alpha} U$. It is straightforward to check that
$U$ is a stationary distribution for this Markov chain. Hence, given a tree $T$ and a designated
root, we obtain a tree-indexed $\alpha$-stable Markov chain on $T$. Interestingly, unlike
the Gaussian case, this process depends on the chosen root as this Markov Chain is not
reversible. In particular, if $(X_0,X_1)$ are two consecutive times for this Markov chain
started in stationarity, then $(X_0,X_1)$ and $(X_1,X_0)$ have different distributions; one can see
this by looking at the two spectral measures.

\begin{proposition}\label{proposition: stable tree}
Fix $\alpha\in (0,2)$, \(a\in (0,1)\), a tree $T$ with designated root $\rho$ and consider
the corresponding tree-indexed $\alpha$-stable Markov chain $X$ on $T$. 
Then \( X^{0} \) is a color process.
\end{proposition}

\begin{proof} We give the proof only for a path and with $\rho$ being the start of the path. 
The extension to the tree case is analogous. As in the previous proof, we want to construct,
on some probability space, a random partition $\pi$
of $[n]$ and random variables \( Y=(Y_1,Y_2, \ldots, Y_n) \) so that
\begin{enumerate}[(i)]
\item \( (X_1,\ldots,X_n) \) and $Y$ have the same distribution, and
\item $(Y^{0},\pi)$ is a color process (for $p=1/2$) with its color representation.
\end{enumerate}
We first recall (see Proposition 1.3.1 in \cite{st1994}, p.20)
that  if    a standard Brownian motion \( (B_t)_{t \geq 0} \) and 
\( S \sim S_{\alpha/2}(2\cos({\pi\alpha}/{4})^{{2}/{\alpha}},1,0)\) 
are independent, then  \(S^{ {1}/{2}}B_{1}  \sim S_\alpha(1,0,0). \) The random variable
$S$ is an example of a so-called subordinator.

Now let $Y_1,S_2,\ldots,S_n,(B_t^{(2)})_{t \geq 0}, \ldots,(B_t^{(n)})_{t \geq 0}$
be independent with 
$Y_1\sim S_\alpha(1,0,0)$, each $S_i\overset{\mathcal{D}}{=} S$, where $S$ is as above and each
\((B_t^{(i)})_{t \geq 0}\) being a standard Brownian motion.

Define $Y_i$ for \(i\in \{2,\ldots,n\}\) inductively by
\[
Y_{i+1}= aY_i +  (1-a^\alpha)^{1 / \alpha} S_{i+1}^{ {1}/{2}} B_1^{(i+1)}.
\]
It is clear from the above discussion that (i) holds. 

Now we extend this process to all times \( t \in [1,n] \) as follows. Let, 
for $t\in (i,i+1)$,
$$
Y_t = aY_i +  (1-a^\alpha)^{1 / \alpha} S_{i+1}^{ {1}/{2}} B_{t-i}^{(i+1)}.
$$
Note that \( (Y_t) \) is left-continuous and has jumps exactly at the
integers. Note also that this process never jumps over the \( x \)-axis. 

Next, considering the random partition $\pi$ of $\{1,2,\ldots,n\}$ given by
$i\sim j$ if $Y_t$ does not hit zero between times $i$ and $j$. Again using the reflection principle, properties of
Brownian motion and the fact that $(Y_t)$ never jumps over the \( x \)-axis it is clear  that
(ii) holds.
\end{proof}

We apply this to a particular symmetric, fully symmetric stable \( n \)-dimensional vector. To this end, let
$S_0$, $S_1$, \ldots, $S_n$ be i.i.d.\ each having distribution \( S_\alpha(1,0,0) \) and for \( i=1,2,\ldots, n \) let 
$$
X_1\coloneqq  a S_0 + (1-a^\alpha)^{1/\alpha} S_i.
$$
We claim that $(X^0_1,X^0_2,\ldots,X^0_n)$ is a color process. To see this, consider
Proposition~\ref{proposition: stable tree} with a homogeneous $n$-ary tree and $\alpha$ and
$a$ being as above. By that proposition, the threshold zero process for the
corresponding tree-indexed Markov chain is a color process.

\section{A geometric approach to Gaussian vectors}
\label{s: geometric}

\subsection{ The geometric picture of a Gaussian vector}\label{ss: geometric}

In this section we switch to a more geometric perspective and view a mean zero 
$n$-dimensional Gaussian vector 
as the values of a certain random function at a set of $n$ points in
$\mathbb{R}^k$ for some $k$. This alternative description is completely well known. More precisely, let \( k \geq 1\),
 \( \mathbf{x}_1, \ldots, \mathbf{x}_n \in \mathbb{R}^k \), and \( W \sim N(0,I_k) \) be a standard
normal distribution in \(\mathbb{R}^k \). If we now let
\begin{equation}\label{eq: def of X given angles}
X\coloneqq (X_i)_{1 \leq i \leq n} \coloneqq  ( \mathbf{x}_i \cdot W)_{1 \leq i \leq n},
\end{equation}
then $X$ is a Gaussian vector with mean zero and covariances 
\( \Cov(X_i,X_j) = \mathbf{x}_i \cdot \mathbf{x}_j \). Note that $X_i$ having variance one
corresponds to $\mathbf{x}_i$ being on the unit sphere \( \mathbb{S}^{k-1} \) in \(\mathbb{R}^k \). 
The above representation can always be achieved with $k=n$. 
Such a representation can be achieved, up to rotations, in
\(\mathbb{R}^k \) if and only if $X$ lives on a $k$-dimensional subspace of \(\mathbb{R}^n \).
We say that $X$ has dimension $k$ if $k$ is the smallest integer where one has this representation
up to rotations.
When we have \( \mathbf{x}_1, \ldots, \mathbf{x}_n \in \mathbb{R}^k \) as above, without loss
of generality, we will always assume that \( \mathbf{x}_1, \ldots, \mathbf{x}_n \) spans \(\mathbb{R}^k \)
so that the dimension of $X$ is $k$.

Now given a standard Gaussian vector \( X \coloneqq (X_i)_{1 \leq i \leq n} \) (recall this means the marginals have mean
zero and variance one) and  $h\in \mathbb{R}$, let $(X^{h}_i)_{1\le i\le n}$ 
be, as before, the threshold process defined by ${X^{h}_i}\coloneqq I(X_i > h)$.
It will be useful to have a simple way to generate $(X^{h}_i)_{1\le i\le n}$ 
which can be done as follows. Assume that $X$ is $k$-dimensional with variances all being one.
We take $n$ points $\mathbf{x}_1,\mathbf{x}_2,\ldots, \mathbf{x}_n$ on \( \mathbb{S}^{k-1} \) 
corresponding to \( (X_i)_{1 \leq i \leq n} \) as described above. 
Let $Z\sim N(0,I_k)$. It is well known that when
$Z$ is written in polar coordinates 
$(r,\theta)$ with $r\ge 0$ and $\theta\in \mathbb{S}^{k-1}$, 
then $r$ and $\theta$ are 
independent with $\theta$ uniform on \( \mathbb{S}^{k-1} \) and
$r$ having the distribution of the square root of a $\chi$-squared
distribution with $k$ degrees of freedom. We then have that $X^{h}_i= 1$ if
and only if $\mathbf{x}_i \cdot Z > h$. Note that $\{ \mathbf{x}:  \mathbf{x}\cdot Z = h\}$ is
a random hyperplane $H_h$ in $\mathbb{R}^k$ perpendicular to $\theta(Z)$ and so
$X^{h}$ is equal to one for points on \( \mathbb{S}^{k-1} \) which lie on one
side of $H_h$ and zero for points lying on the other side. Note that when
$h=0$, the hyperplane goes through the origin and it is the points on the same side 
as $\theta(Z)$ that get value one; in particular, when $h=0$, the
value of $X^{h}_i$ only depends on $\theta(Z)$ and not on $r(Z)$.
However, when $h >0$, the hyperplane $H_h$ can go through any point of the 
one-sided infinite line from the origin going through $\theta(Z)$. 
In particular, $H_h$ might not intersect \( \mathbb{S}^{k-1} \) 
at all; this would correspond exactly to $r(Z) <h$.

\subsection{Gaussian vectors canonically indexed by the circle}

\begin{proposition}\label{proposition: k = 2, h = 0} \label{prop: first}
Consider $n$ points $\mathbf{x}_1,\ldots,\mathbf{x}_n$ on \( \mathbb{S}^1 \) satisfying $\mathbf{x}_i\cdot \mathbf{x}_j\ge 0$ for all
$i,j$; this is equivalent to the correlations $a_{ij}$ of the corresponding Gaussian process $X$ 
being nonnegative. Then \( X^{0} \) is a color process.
\end{proposition}

\begin{proof}
Using the nonnegative correlations, it is easy to check that the $n$ points
${\{\mathbf{x}_1, \mathbf{x}_2, \ldots, \mathbf{x}_n\} \subseteq \mathbb{S}^1}$ must lie on an arc of 
length at most $\pi/2$. Since the distribution of a Gaussian process is invariant under rotations, 
we may assume that the $n$ points lie on the arc $0\le \theta \le \pi/2$. Hence we can assume 
that $\mathbf{x}_j=e^{i\theta_j}$ with 	$0\le \theta_1 < \theta_2 < \ldots < \theta_n\le \pi/2$.
	
We will couple \( X^{0} \) with a color process together with its color representation 
in such a way that \( X^{0} \) and the color process match exactly. We first show how one 
uniform point $U$ on $\mathbb{S}^1$ generates a color process together with its color representation. 
Let
$$
I_1=[0,\theta_1],I_2=[\theta_1,\theta_2], \ldots,
I_k=[\theta_{k-1},\theta_k],\ldots, I_{n+1}=[\theta_{n},{\pi}/{2}]
$$
noting that  the first and last arcs might be trivial. Letting $I_k^\theta$ be $I_k$ rotated 
counterclockwise by $\theta$, we note that
$$
\left\{I^\theta_k: k\in \{1,\ldots,n+1\},\theta 
\in \{0,{\pi}/{2},\pi,{3\pi}/{2}\}\right\}
$$
partitions $\mathbb{S}^1$. Now for $k=1,\ldots,n+1$, if $U$ falls in
$I^0_k \cup I^{\frac{\pi}{2}}_k \cup I^\pi_k \cup I^{\frac{3\pi}{2}}_k$,
we partition $\{\mathbf{x}_1, \mathbf{x}_2, \ldots, \mathbf{x}_n\}$ into the two sets
$J_1\coloneqq \{\mathbf{x}_1, \ldots, \mathbf{x}_{k-1}\}$ and $J_2\coloneqq \{\mathbf{x}_k, \ldots, \mathbf{x}_n\}$ with
the obvious caveat when $k \in \{ 1, n+1 \}$. Next we color $J_1$ and $J_2$ as follows. 
If $U$ is in $[0,\pi/2]$, we color each cluster 1, if $U$ is in $[\pi/2,\pi]$, we 
color $J_1$ 0 and $J_2$ 1, if $U$ is in $[\pi,3\pi/2]$, we color each cluster 0 and
if $U$ is in $[3\pi/2,2\pi]$, we color $J_1$ 1 and $J_2$ 0. This clearly yields a 
color process (with $p=1/2$) together with its color representation. Finally 
observe that this color process is exactly $X^{0}$ if we use \( U \) for $\theta(Z)$.
\end{proof}

\begin{remark}
It is easy to see that the threshold zero-process here is such that it is constant with probability at least $1/2$.
Hence the proof of Theorem 1.2 in~\cite{fs2019b} also yields it is a color process.
Moreover, the color representation obtained there can be checked to be the same as the one given above. The description of the color representation given in the present section will however be useful when dealing with the case $h\neq 0$ as in
Proposition~\ref{proposition: nonzero h on a circle}.
\end{remark}

\begin{remark} For any color process \( (Y_i) \) with \( p = 1/2 \), for any
\( i \) and \( j \) it is clear that 
\begin{equation}\label{eq: colorcorrelation}
\Cov (Y_i,Y_j) = \frac{q_{ij}}{4}
\end{equation}
and that
\[
P(Y_i = Y_j = 1) = \frac{1}{4} + \frac{q_{ij}}{4}.
\]
In the case of Proposition~\ref{prop: first}, it is clear that
\begin{equation}\label{eq: simplifying equation}
1-q_{ij} = \frac{|\theta_j - \theta_i|}{\pi/2}
\end{equation}
and hence that
\begin{equation}\label{eq: p00}
P(X_i^{0} = X_j^{0} = 1) = \frac{1}{2} - \frac{|\theta_j - \theta_i|}{2\pi}.
\end{equation}
Since \( |\theta_j - \theta_i| = \arccos a_{ij} \)
it follows that
\begin{equation}\label{eq: p00again}
P(X_i^{0} = X_j^{0} = 1) = \frac{1}{2} - \frac{\arccos a_{ij}}{2\pi}.
\end{equation}
This is of course one of many ways to derive this last expression which is known as
Sheppard's formula (see \cite{s1899}).

This discussion also leads to the formula
\begin{equation}\label{eq: q.corr}
q_{ij} = 1 - \frac{2\arccos a_{ij}}{\pi}.
\end{equation}
\end{remark}

The proof of the following elementary lemma, based on inclusion-exclusion, is left to the reader.

\begin{lemma}\label{lemma: Fourier}
If $X \coloneqq(X_1,X_2,X_3)$ is \( \{ 0,1 \} \)-symmetric, then 
\begin{equation}\label{eq: p000}
\nu_{000} = \frac{ \nu_{00.} +  \nu_{0.0} + \nu_{.00}}{2}  - \frac{1}{4}.
\end{equation}

In particular, using~\eqref{eq: p00}, if $X$ corresponds to threshold zero for a mean zero
Gaussian vector, the above is equal to
\begin{equation}\label{eq: p000 Gaussian}
\frac{1}{2} - \frac{ \theta_{12}+  \theta_{13} +  \theta_{23} }{4\pi}.
\end{equation}
\end{lemma}

\medskip

\begin{proposition}\label{proposition: nonzero h on a circle}
Consider $n$ points $\mathbf{x}_1,\ldots,\mathbf{x}_n$ on \( \mathbb{S}^1 \) satisfying $\mathbf{x}_i\cdot \mathbf{x}_j\ge 0$ for all
$i,j$. Then \( X^{h} \) does not have a color representation for any \( h \not = 0 \), \( n\geq 3 \).
\end{proposition}

\begin{proof}
It suffices to prove this for $h>0$ and $n=3$. Since $h >0$, it is clear 
from the construction of $X^h$ described in~\eqref{eq: def of X given angles} that $(0,1,0)$ has 
positive probability but that $(1,0,1)$ has probability zero. However, it 
is immediate that no color process can have this property.
\end{proof}

\begin{figure}[H]
\begin{minipage}[b]{0.45\linewidth}
\centering
\begin{tikzpicture}
\draw (0,0) circle (1.5cm);
\draw[->] (0,-2) -- (0,2);
\draw[->] (-2,0) -- (2,0);
\fill (1.44889,0.388229) circle (2pt);
\fill (1.06066,1.06066) circle (2pt);
\fill (0.292635,1.47118) circle (2pt);
\draw (-1.5,1.5) -- (1.5,-1.5);
\end{tikzpicture}
\caption{The image above illustrates the situation when   \( h = 0 \).}
\label{figure: k = 2 and h=0}
\end{minipage}
\hspace{0.5cm}
\begin{minipage}[b]{0.45\linewidth}
\centering
\begin{tikzpicture}
\draw (0,0) circle (1.5cm);
\draw[->] (0,-2) -- (0,2);
\draw[->] (-2,0) -- (2,0);
\fill (1.44889,0.388229) circle (2pt);
\fill (1.06066,1.06066) circle (2pt);
\fill (0.292635,1.47118) circle (2pt);
\draw (0,2) -- (2,0);
\end{tikzpicture}
\caption{The image above illustrates the situation when   \( h > 0\).}
\label{figure: k = 2 and h>0}
\end{minipage}

\end{figure}

\subsection{A general obstruction for having a color representation for \texorpdfstring{$h \neq 0$}{h>0} } 

By symmetry, we can assume \( h > 0 \).

The following is precisely a higher dimensional analogue of 
Proposition~\ref{proposition: nonzero h on a circle}. The latter is the 
special case $n=3$ together with the fact that any three points on the circle 
are in general position.

\begin{theorem}\label{theorem: geometric restrictions}
The standard Gaussian process $X$ associated to $n$ points \( \mathbf{x}_1,  \ldots, \mathbf{x}_n \in \mathbb{S}^{n-2}\) in general position (equivalently not contained in
an $(n-2)$-dimensional hyperplane) is such that $X^h$ is not a color process for any $h>0$.

More generally, if \(X \coloneqq (X_1,X_2, \ldots, X_n) \) is a random vector such that
\begin{itemize}
\item \( (X_1,X_2, \ldots, X_{n-1}) \) is fully supported on \( \mathbb{R}^{n-1} \)
\item there is \( (a_1,a_2,\ldots, a_n) \in \mathbb{R}^n \backslash \{ 0 \} \) such that a.s.\
\begin{equation}\label{eq: non-equal sumsbefore}
\sum_{i=1}^n a_i X_i = 0
\end{equation} 
and
\begin{equation}\label{eq: non-equal sums}
\sum_{i=1}^n a_i \not = 0,
\end{equation} 
\end{itemize}
then \(X^{h} \) is not a color process for any \( h > 0 \).
\end{theorem}

\begin{remark}
Any $n$-dimensional standard Gaussian vector which is not fully dimensional can be 
represented by points on $\mathbb{S}^{n-2}$. When the $n$ points are {\it not} in general position,
which can only happen if $n\ge 4$, in which case the above result is not applicable,
we will see in Corollary~\ref{corollary: n=2 and h large} that nonetheless
$X^h$ is not a color process for large $h$. 
Perhaps the simplest example of a four-dimensional
Gaussian vector which is not fully dimensional but does not correspond to points
on $\mathbb{S}^{2}$ in general position appears in Figure~\ref{figure: ball and plane}.
In the next subsection, we will see in 
Theorem~\ref{theorem: 4pointsoncircle} that this case will lead us to an important example
for which we will have a phase transition. 
\end{remark}

\begin{proof}[Proof of Theorem~\ref{theorem: geometric restrictions}]
We will first observe that the second statement implies the first.
One can order the $n$ points $\mathbf{x}_1,\ldots,\mathbf{x}_n \in \mathbb{S}^{n-2}$ in general position 
such that the first $n-1$ points   are linearly independent.
This implies that the corresponding Gaussian vector $X=(X_1,\ldots,X_n)$
satisfies the first condition. Next, since $\mathbf{x}_1,\ldots,\mathbf{x}_n$ are linearly dependent
(as they sit inside $\mathbb{R}^{n-1}$)
there exists
\( (a_1,a_2,\ldots, a_n) \in \mathbb{R}^n \backslash \{ 0 \} \) such that 
$$
\sum_{i=1}^n a_i \mathbf{x}_i=0
$$
which implies~\eqref{eq: non-equal sumsbefore}. 
Finally~\eqref{eq: non-equal sums} must hold since
$\mathbf{x}_1,\ldots,\mathbf{x}_n$ are in general position.

For the second statement, 
note first that we can assume that \( |a_i| > 0 \) for \( i = 1,2, \ldots, n \) since we can
remove the $X_i$'s for which $a_i=0$. If \( a_j > 0 \) for all \( j \)
(with a similar argument if \( a_j < 0 \) for all \( j \)), then for all 
$h>0$, \( \nu_{1^n}(h) = 0 \) in which case there clearly cannot be any color representation. 
We hence assume that there are both positive and negative values among the $a_j$'s.
Furthermore since \( \sum_{i = 1 }^n a_i X_i = 0 \) and \( (X_1,X_2, \ldots, X_{n-1} )\) is fully 
supported, for any \( i \), if we define \( I_i = \{ 1,2, \ldots, n \}\backslash \{ i \} \), 
then the vector \( (X_j)_{j \in I_i} \) is fully supported. This implies in particular 
that we, possibly after reordering the random variables and changing all the signs, 
can assume that
\[
\sum_{i \colon a_i > 0} |a_i| < \sum_{j \colon a_j<0} |a_j|
\]
and that \( a_n > 0 \).

Fix now $h>0$. Now define the binary string \( \rho \) by \( \rho(i) = I(a_i <0) \) and 
let \( \mathcal{E} \) be the event that
\[
\forall j<n \colon X_j > h \text{ if } a_j < 0 \textnormal{ and } X_j \le h \text{ if } a_j > 0 .
\]
Since \( (X_1,X_2,\ldots, X_{n-1}) \) is fully supported, the probability of the event \( \mathcal{E} \)
is strictly positive. Since 
\[
\sum_{i \colon a_i > 0} |a_i| X_i = \sum_{j \colon a_j<0} |a_j| X_j,
\]
this implies that on \(\mathcal{E} \),
\begin{align*}
X_n &= \frac{-\sum_{j<n} a_jX_j}{a_n}
= \frac{\sum_{j<n \colon a_j < 0} |a_j|X_j}{a_n} - \frac{\sum_{j<n \colon a_j > 0} |a_j|X_j}{a_n}
\\& 
\geq h \cdot \left( \frac{\sum_{j<n \colon a_j < 0} |a_j|}{a_n} - \frac{\sum_{j<n \colon a_j > 0} |a_j|}{a_n} \right) > h
\end{align*} 
which in particular implies that \( \nu_\rho = 0 \).

On the other hand, since \( (X_1,X_2,\ldots, X_{n-1}) \) is fully supported, the event
\[
\forall j < n \colon X_j \in [\alpha h, h] \textnormal{ if } a_j < 0 \textnormal{ and } X_j \in (h, \beta h] \textnormal{ if } a_j>0
\]
has strictly positive probability for any \( \alpha \in (0,1) \) and \( \beta \in (1,\infty) \). On this event we have that
\begin{align*}
X_n &= \frac{\sum_{j <n \colon a_j<0} |a_j| X_j - \sum_{i <n\colon a_i > 0} |a_i| X_i }{a_n}
\\&\geq h \cdot \frac{\alpha \sum_{j <n \colon a_j<0} |a_j| - \beta \sum_{i <n\colon a_i > 0} |a_i| }{a_n}.
\end{align*}
Since 
\[
\sum_{j <n \colon a_j<0} |a_j| - \sum_{i <n\colon a_i > 0} |a_i| > a_n
\]
it follows that \( X_n > h \) if \( \alpha \) and \( \beta \) are both sufficiently close to one. 
In particular, this implies that \( \nu_{1-\rho} > 0 \). Since \( \nu_\rho = 0 \) but \( \nu_{1-\rho} > 0 \), 
it follows that \( X^h \) cannot have a color representation.
\end{proof}                     

\begin{figure}[ht]
\centering
\includegraphics[width=0.8\textwidth]{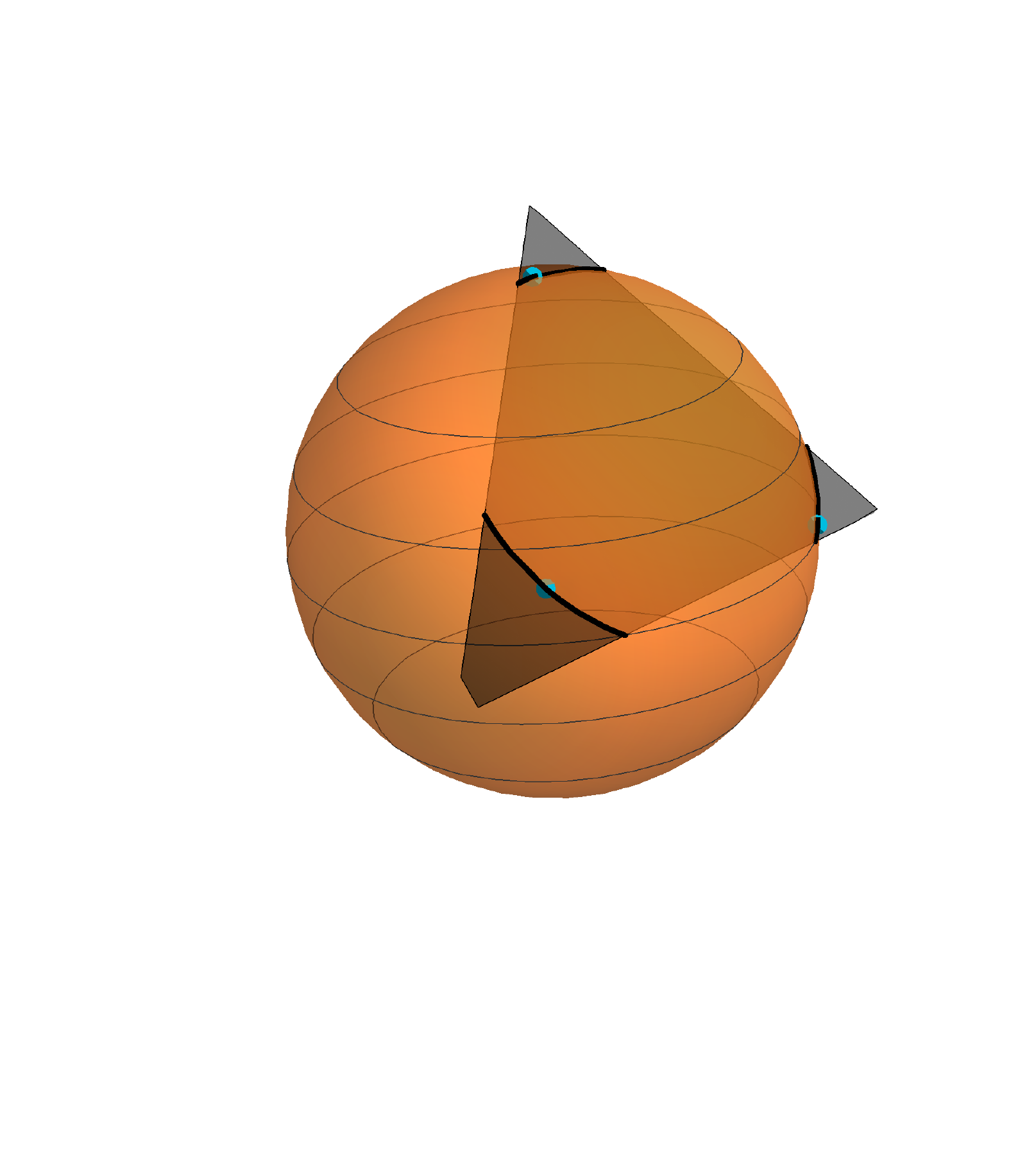}
\vspace{-15ex}
\caption{The picture above shows the three points \( \mathbf{x}_1 \), \( \mathbf{x}_2 \) and \( \mathbf{x}_3 \) corresponding to a mean zero variance one Gaussian vector with \( a_{12} = a_{23} = 0.2 \) and \( a_{13} = 0.2^2 \). The bold black lines are the positions where we could add a fourth point \( \mathbf{x}_4 \) without the 
existence of a color representation for some \( h > 0 \) being ruled out by 
Theorem~\ref{theorem: geometric restrictions}.}
\label{figure: ball and plane}
\end{figure}

\subsection{A four-dimensional Gaussian exhibiting a non-trivial phase transition}
 
In this subsection we will study an example, corresponding to four points on \( \mathbb{S}^2 \), for which the existence of a color representation for positive \( h \) is not ruled out by 
Theorem~\ref{theorem: geometric restrictions}.
To this end, let \( \theta \in (0,\pi/2] \) and define \( \mathbf{x}_1, \mathbf{x}_2, \mathbf{x}_3, \mathbf{x}_4 \in \mathbb{S}^2 \) by
\[
\begin{cases}
\mathbf{x}_1 = (\sin \theta, 0, \cos \theta) \cr
\mathbf{x}_2 = (0, \sin \theta, \cos \theta) \cr
\mathbf{x}_3 = (-\sin \theta, 0, \cos \theta) \cr
\mathbf{x}_4 = (0, -\sin \theta, \cos \theta) \cr
\end{cases}
\]
and for $i=1,2,3,4$, let \( X_i = \mathbf{x}_i\cdot W\), where \( W \sim N(0,I_3) \). Then \( X \sim N(0,A) \) for 
\begin{equation}\label{eq: A def}
A = \begin{pmatrix}
1 & \cos^2 \theta & \cos^2 \theta - \sin^2 \theta & \cos^2 \theta \\
\cos^2 \theta & 1 & \cos^2 \theta & \cos^2 \theta - \sin^2 \theta  \\
 \cos^2 \theta - \sin^2 \theta & \cos^2 \theta & 1 & \cos^2 \theta  \\
 \cos^2 \theta & \cos^2 \theta - \sin^2 \theta & \cos^2 \theta & 1
\end{pmatrix}.
\end{equation}
Geometrically, this corresponds to having four points in a square on a 2-sphere at the same latitude, and it follows easily that
\begin{equation}\label{equation: square}
X_1 + X_3 = X_2 + X_4.
\end{equation}
Note that \( A \) has nonnegative entries if and only if \(  \theta \leq \pi/4 \).

The following theorem implies Theorem~\ref{theorem:pt}.

\begin{theorem}\label{theorem: 4pointsoncircle}
Let \( X^\theta \) be a Gaussian vector with covariance matrix given by~\eqref{eq: A def}. Then
\begin{enumerate}[(i)]
\item \( X^{\theta,0} \) is a color process for all \( \theta \in (0,\pi/4] \), 
\item there is \( \theta_0 > 0 \) such that for all \( \theta<\theta_0 \), 
there exists \( h_\theta>0 \) such that \( X^{\theta,h} \) is a color process 
for all \( h \in (0, h_{\theta}) \).
\item for all \( \theta \in (0,\pi/4 ) \), there is \( h_\theta > 0 \) such that  
\( X^{\theta,h} \) has no color representation for any \( h > h_\theta \).
\end{enumerate}
\end{theorem}

\begin{lemma}\label{lemma: 4pointsoncirclelemma}
Let \( X^\theta \) be a Gaussian vector with covariance matrix given by~\eqref{eq: A def}. 
Then for all $h$, \( X^{\theta,h} \) has a color representation if and only if there is a 
color representation of \( (X^{\theta,h}_1,X^{\theta,h}_2,X^{\theta,h}_3) \) which satisfies
\begin{equation} \label{eq: necessary and sufficient inequalities}
\begin{cases}
 q_{123} \geq q_{13,2} \geq 0\cr 
 2q_{12,3} - 2q_{13,2}\geq   q_{1,2,3} \geq 0 .
\end{cases}
\end{equation}
\end{lemma}

\begin{proof}
Fix \( h \ge 0 \) and assume first that there is a color representation \( (q_\sigma) \) 
(the dependence on $h$ will be suppressed)
of \( (\nu^h_\rho)\). Since the distribution of  \( X^\theta \) is invariant under the action of the 
dihedral group, we can assume that \( (q_\sigma) \) also is.
Note that it follows from~\eqref{equation: square} that \( \nu_{0101} = 0 \) , and 
hence \( \nu_{010\cdot} = \nu_{0100}\). In particular, this implies that
\begin{equation}\label{eq: reduced solution zeros}
q_{1,2,3,4} = q_{13,2,4} = q_{1,24,3} = q_{13,24} = 0,
\end{equation}
and using this, we obtain (using the assumed symmetry)
\begin{equation}\label{eq: reduced solution}
\begin{cases}
q_{1234} = q_{123} - q_{123,4} = q_{123}-q_{13,4}= q_{123}-q_{13,2} \cr 
q_{123,4} = q_{13,4}  = q_{13,2} \cr
q_{12,3,4} = q_{2,3,4} - q_{14,2,3} = q_{1,2,3}-q_{12,3,4} = q_{1,2,3}/2\cr
q_{12,34} = q_{12,3} - q_{124,3} - q_{12,3,4} = q_{12,3} -  q_{13,2} - q_{1,2,3}/2.
\end{cases}
\end{equation}
Since this is a color representation by assumption, \( q_\sigma \geq 0 \) for all \( \sigma \), 
which is equivalent to~\eqref{eq: necessary and sufficient inequalities}.
This proves the necessity in the first part of the lemma.

To see that we also have sufficiency, let \( q=( q_{123}, q_{12,3},q_{13,2}, q_{1,23},  q_{1,2,3}) \) 
be a color representation of \( (X_1^{\theta,h},X_2^{\theta,h},X_3^{\theta,h}) \) 
which satisfies the inequalities in~\eqref{eq: necessary and sufficient inequalities}.
Define \( q_{\sigma} \) for \( \sigma \in \mathcal{B}_4 \)
by~\eqref{eq: reduced solution zeros}~and~\eqref{eq: reduced solution}
and extend to all partitions by making it invariant under the dihedral group.
Since~\eqref{eq: necessary and sufficient inequalities}
holds, \( q_{\sigma} \geq 0 \) for all \( \sigma \in \mathcal{P}_4 \). Also, one checks that they sum to
one and the projection onto $\{1,2,3\}$ is $q$ above.
Using the fact that \( \nu_{010\cdot} = \nu_{0100} \), one can check that the probability of any 
configuration is determined by the three--dimensional marginals. From here,
one verifies that this yields a color representation of \( X^{\theta,h} \), as desired.
\end{proof}

\begin{proof}[Proof of Theorem~\ref{theorem: 4pointsoncircle}]
To see that (i) holds, let \( h = 0 \). We will apply Lemma~\ref{lemma: 4pointsoncirclelemma}.
%
Then one easily verifies that
the process \( (X^{\theta,0}_1, X^{\theta,0}_2, X^{\theta,0}_3 ) \)    
has a \textit{signed} color representation given by
\[
\begin{cases}
q_{123} &= 1 - 4  (\nu_{001} + \nu_{010} + \nu_{100})  + t \cr
q_{12,3} &= 4\nu_{001} - t \cr
q_{1,23} &= 4\nu_{100} - t \cr
q_{13,2} &=  4\nu_{010} - t \cr
q_{1,2,3} &= 2t
\end{cases}
\]
for some free variable \( t \in \mathbb{R}\). This will give a color representation for all \( t \) 
which is such that \( q_\sigma \geq 0 \) for all \( \sigma \in \mathcal{B}_3 \).
Using~\eqref{eq: p00}~and~\eqref{eq: p000 Gaussian}, one easily verifies that in a Gaussian setting, 
the set of equations above can equivalently be written as
\[
\begin{cases}
q_{123} &
= 1 - \frac{\theta_{12} + \theta_{13} + \theta_{23}}{\pi} + t\cr
q_{12,3} &
= \frac{(\theta_{12} + \theta_{13} + \theta_{23})-2\theta_{12}}{\pi} - t\cr
q_{1,23} &
= \frac{(\theta_{12} + \theta_{13} + \theta_{23})-2\theta_{23}}{\pi} - t\cr
q_{13,2} &
= \frac{(\theta_{12} + \theta_{13} + \theta_{23})-2\theta_{13}}{\pi} - t\cr
q_{1,2,3} &= 2t.
\end{cases}
\]
Rearranging, we see that these are all nonnegative if and only if
\begin{equation}\label{eq: inequality in square example}
0 \lor \left( \frac{\sum_{i\neq j} \theta_{ij}}{\pi} -1 \right)
\leq t 
\leq  \frac{\sum_{i\neq j} \theta_{ij}-2(\theta_{12} \lor \theta_{13} \lor \theta_{23})}{\pi}.
\end{equation}
In our specific example, we have that 
\[
\begin{cases}
\theta_{12} = \theta_{23} = \arccos \cos^2 \theta \cr
\theta_{13} =  2\theta \geq \theta_{12} 
\end{cases}
\]
and hence \eqref{eq: inequality in square example} simplifies to
\[
0 \lor \left(\frac{2 \arccos \cos^2 \theta + 2\theta}{\pi} -1\right)  \leq t \leq  \frac{2\arccos \cos^2 \theta - 2 \theta}{\pi}.
\]
Similarly, we can rewrite~\eqref{eq: necessary and sufficient inequalities} as

\[\begin{cases}
 t \geq \frac{\sum_{i\neq j} \theta_{ij} - \theta_{13} - \pi/2}{\pi} = \frac{2 \arccos \cos^2 \theta  - \pi/2}{\pi}  \cr
t \leq \frac{2(\theta_{13} - \theta_{12})}{\pi} = \frac{2(2\theta - \arccos \cos^2\theta )}{\pi} .
\end{cases}
\]
If we put these sets of inequalities together, and use  that \( \theta \in (0,\pi/4] \),  
we obtain the following necessary and sufficient condition for the existence of such a \( t \):
\[
0 \lor \frac{2 \arccos \cos^2 \theta  - \pi/2}{\pi}  \leq   \frac{2\arccos \cos^2 \theta - 2 \theta}{\pi} \land \frac{2(2\theta - \arccos \cos^2\theta )}{\pi} .
\]
Here it is easy to verify that
\[
0 \lor \frac{2 \arccos \cos^2 \theta  - \pi/2}{\pi}   \leq   \frac{2\arccos \cos^2 \theta - 2 \theta}{\pi}
\]
and that
\[
0 \leq \frac{2(2\theta - \arccos \cos^2\theta )}{\pi}
\]
and hence to see that we can always pick \( t \) so that the above inequalities hold it suffices to show that
\[
 \frac{2 \arccos \cos^2 \theta  - \pi/2}{\pi}  \leq   \frac{2(2\theta - \arccos \cos^2\theta )}{\pi} 
\]
for all \( \theta \in (0,\pi/4] \). To this end, note first that we can rewrite the inequality above as
\[
 \arccos \cos^2 \theta - \theta   \leq    \pi/8 .
\]
This can be verified to hold for all \( \theta \in (0, \pi/4] \)
by verifying that the left hand side is increasing in \( \theta \) 
for \( \theta \in (0,\pi/4] \) and noting that
\[
 \arccos \cos^2 (\pi/4) - \pi/4    = \arccos({1}/{2}) - \pi/4 = \pi/3 - \pi/4 = \pi/12 \leq \pi/8.
\]
The desired conclusion now follows.

To see that (ii) holds, note first that by~Theorem~\ref{theorem: small h} and a computation, 
the value of the free parameter \( t \) corresponding to the limit of \(h \to 0 \) is given by
\[
t = 1  -\frac{\arccos \left( \frac{ 2\sin^4\theta}{(1+\cos^2 \theta)^2 }-1\right)}{\pi}   .
\]
Using the proof of (i), it follows that it suffices to show that
\[
 \frac{2 \arccos \cos^2 \theta  - \pi/2}{\pi}  <  1  -\frac{\arccos \left( \frac{ 2\sin^4\theta}{(1+\cos^2 \theta)^2 }-1\right)}{\pi} <  \frac{2(2\theta - \arccos \cos^2\theta )}{\pi} 
\]
for all sufficiently small \( \theta \). To this end, note first that at \( \theta = 0 \) the first 
expression is equal to \( -1/2 \) while the second and third expression are both  equal to zero, and 
hence the first inequality is strict for all sufficiently small \( \theta \). To compare the last 
two expressions, one verifies that the derivatives of these two expressions at \( \theta = 0 \) 
are given by \( 0 \) and \( 4-2\sqrt{2}\) respectvely, and hence (ii) is established.

Finally, (iii) follows from Corollary~\ref{corollary: n=2 and h large}.
\end{proof}

\subsection{A four-dimensional Gaussian with nonnegative correlations
whose zero threshold has no color representation}

In this subsection, we study a particular example which will in particular yield a proof
of Theorem~\ref{theorem:4d0threshold}; see (ii) and (iii) below.
                             
\begin{theorem}\label{theorem: symmetric plus mean}
Let \( (X_1, X_2, \ldots, X_{n-1}) \) be a fully symmetric multivariate mean zero variance one
Gaussian random vector with pairwise correlation \( a \in [0,1) \), and let 
\[ X_{n} = (X_1+X_2+\ldots + X_{n-1})/\sqrt{a(n-1)^2+(1-a)(n-1)}.
\] ensuring that $X_{n}$ has mean zero and variance one. In addition, nonnegative pairwise
correlations is immediate to check.
If \( X^a\coloneqq  (X_1, X_2, \ldots, X_{n}) \), then the following hold.
\begin{enumerate}[(i)]
\item When \( n = 3 \), \( X^{a,0} \) is a color process for any  \( a \in [0,1) \).
\item When \( n \geq 4 \) and \( a \) is sufficiently close to zero (or zero), 
\( X^{a,0} \) is not a color process.
\item For $n\ge 4$, there exists a fully supported multivariate mean zero variance one
Gaussian random variable $X$ with nonnegative correlations for which
$X^0$ is not a color process.
\item When \( n \geq 4 \) and \( a \) is sufficiently close to one, \( X^{a,0} \) is a color process.
\item For any \( n \geq 3 \), \( a \in [0,1) \) and \( h > 0 \), \( X^{a,h} \) is not a color process.
\end{enumerate}
\end{theorem}

\begin{proof}
(i). The claim for \( n = 3 \) follows immediately from Proposition~\ref{proposition: k = 2, h = 0}
or Proposition 2.12 in \cite{st2017}. 

(ii). We first consider \( n \geq 4 \) and $a=0$ and obtain the result in this case.
If $X^0$ is a color process, then it must be the case that the color representation
gives weight $1/(n-1)$ to each of the $n-1$ partitions which consist of all singletons
except $n$ is in a block of size 2. This is because (1) since $X_1,X_2,\dots,X_{n-1}$ 
are independent, none of $1,2,\ldots,n-1$ can ever be in the same cluster, 
(2) if $n$ is in its own cluster
with positive probability, then $\nu_{0^{n-1}1} >0$ which contradicts the fact that
$X_1,X_2,\ldots,X_{n-1}$ all negative and $X_n$ positive is impossible and (3) symmetry. 
On the other hand, by~\eqref{eq: q.corr}, each of the above partition elements must have value
$1 - \frac{2\arccos \frac{1}{\sqrt{n-1}}}{\pi}$. The conclusion is that if it is a color process, then
$$
\frac{1}{n-1}=1 - \frac{2\arccos \frac{1}{\sqrt{n-1}}}{\pi}.
$$
This is true for $n=3$ (as it must be) but we show this is false for all $n\ge 4$.
Rearranging, this is equivalent to
\begin{equation}\label{eq:  necc.cond}
  \frac{\pi}{2} \cdot \frac{n-2}{n-1}  =  \arcsin\sqrt{  \frac{n-2 }{n-1} } .
\end{equation}
Now consider the two functions \( f(x) = \pi x^2/2 \) and \( g(x) = \arcsin x \) for 
\( x \in [0,1] \). Then we clearly have \( f(0) = g(0) \) and \( f(1)=g(1) \). Moreover, 
one can easily check that both functions are continuously differentiable, that their first 
derivatives agree only at 
\( x = \sqrt{\frac{1}{2} \pm \frac{\sqrt{\pi^2-4}}{2\pi}} \) (i.e. at \( x \approx 0.338247\) 
and \( x \approx 0.941057 \)) and that \( f'(0) = 0 < 1 = g'(0) \)
and \( f'(1) = \pi < \infty = g'(1) \). This easily implies that \(\{x: f(x)> g(x)\} \) is of the 
form $(b,1)$. Hence we need only check that~\eqref{eq:  necc.cond} fails for $n=4$ with the left
side being larger. However, this is immediate to check. Finally, to obtain the result for small $a$
depending on $n$, one just uses the fact that the set of color processes is closed.

(iii). Fix $n\ge 4$, take $a=0$ and replace $X_n$ by $X'_n\coloneqq  \epsilon Z + (1-\epsilon^2)^{ {1}/{2}}X_n$ 
where $Z$ is another standard Gaussian independent of everything else. Then for every 
$\epsilon>0$, the resulting vector $X$ is fully supported with
nonnegative correlations. However, for small $\epsilon$, $X^0$ cannot be a color
process since the color processes are closed and the limit as $\epsilon\rightarrow 0$
is not a color process by (ii).

For (iv), note that by the proof of Theorem 1.2 in~\cite{fs2019b}, a sufficient condition for a \( \{ 0,1 \} \)-symmetric process to be
a color process is that \( \nu_{0^n} \geq 1/4 \). In our case, we clearly have that for 
any \( n \), \( \nu_{0^n} \to {1}/{2} \) as \( a \to 1 \), and hence the desired 
conclusion follows.

Finally for (v), with \( n \geq 3 \), \( a \in [0,1) \) and \( h > 0 \),
this follows immediately from Theorem~\ref{theorem: geometric restrictions}.
\end{proof}

\subsection{ An extension to the stable case}\label{section: integrals}

In this subsection, we explain to which extent the results in the previous 
subsection  can be carried out for the stable case. We assume now that
\( X_1\), \(X_2\), \ldots, \(X_{n-1} \) are i.i.d.\ each with distribution $S_\alpha(1,0,0)$
for some $\alpha\in (0,2)$ and we let 
\[ X_{n} = (X_1+X_2+\ldots + X_{n-1})/(n-1)^{1/\alpha} \]
and \( X\coloneqq  (X_1, X_2, \ldots, X_{n}) \). 

Proposition 2.12 in \cite{st2017} implies, as before, that when \( n = 3 \), \( X^{0} \) is 
a color process (the \( \{ 0,1 \} \)-symmetry is obvious and the nonnegative correlations being an 
easy consequence of Harris' inequality).
Concerning whether \( X^{h} \) can be a color process for some \( n \ge  3 \) and
\( h > 0 \), Theorem~\ref{theorem: geometric restrictions} 
implies that it cannot be {\it except} perhaps when $\alpha=1$.
For \( n \geq 4 \) it seems, by using similar arguments and Mathematica, that \( X^0 \) is a color process for at most one value of \( \alpha \).

\section{Results for large thresholds 
and the discrete Gaussian free field}\label{section:nonzerogaussian}

In the first subsection of this section, we show that non-fully supported Gaussian vectors do not have 
color representations for large $h$. On the other hand, in the second subsection, we give the proof of 
Theorem~\ref{theorem: strict dgff and large h} that discrete Gaussian free fields have color representations
for large $h$.

\subsection{An obstruction for large \texorpdfstring{$h$}{h}}
We first deal with the case \( n = 2 \), where we have the following easy result.

 \begin{proposition}\label{theorem: h to infinity}
Let \( X \coloneqq  (X_1,X_2)   \) be a standard Gaussian vector with \\
\( \Cov(X_1,X_2) \in [0,1) \). Then \( X^{h} \) 
has  a (unique) color representation \((q_\sigma)_{\sigma \in \mathcal{B}_2}\) for all \( h \in \mathbb{R} \) and \( \lim_{h \to \infty} q_{12}(h) =0 \).
\end{proposition}

This result essentially follows from  Theorem~2.1 in~\cite{dm2001} (see also Lemma~\ref{lemma: Gaussian cube 
tails II} here) but we include a proof sketch here.
\begin{proof}[Proof of Proposition~\ref{theorem: h to infinity}]
Note first that since \( n = 2 \), the nonnegative correlation immediately
implies that  \( X^h \) has a color representation for all \( h \in \mathbb{R} \), and hence we need only show that  \( \lim_{h \to \infty} q_{12}(h) =0 \).
Since it can be easily checked that
\[
q_{12}(h)= \frac{\nu_{11}(h)-\nu_1(h)^2}{\nu_0(h)\nu_1(h)}  
\]
we need to show that 
\begin{equation}\label{eq: q12 limit}
  \lim_{h \to \infty} {\nu_{11}(h)}/{\nu_1(h)}  = 0;
  \end{equation}
this however is straighforward.
\end{proof}

The previous result immediately implies the following.
 
\begin{corollary}\label{corollary: h large limit}
If \( X \coloneqq  (X_1, X_2, \ldots, X_n) \)  is a standard Gaussian vector with \( \Cov(X_i,X_j) \in [0,1) \) 
for all \( i <j \) and   \( X^{h} \) has a color representation \( (q_\sigma)_{\sigma \in \mathcal{B}_n} \) for arbitrarily large \( h \), then 
$$
\lim_{h \to \infty}  q_{1,2,3,\ldots, n} (h) =1.
$$
\end{corollary}

Interestingly, this gives the following negative result when \( X \) is not fully dimensional.
\begin{corollary}\label{corollary: n=2 and h large}
Let \( (X_1, X_2, \ldots, X_n) \) be a standard Gaussian vector with \\
\( \Cov(X_i,X_j) \in [0,1) \) for all \( i <j \). If \( X \) is not fully supported, then for 
all sufficiently large \( h \), \( X^h \) is not a color process.
\end{corollary}

\begin{proof} 
Since \( X \) is not fully dimensional, there must exist a linear relationship between the variables. As a 
result, there must exist $\rho\in \{0,1\}^n$ so that for all $h>0$, $\nu_\rho(h)=0$. Hence, if there is a 
color representation \( (q_\sigma(h)) \) for some \( h \), it must satisfy  $q_{1,2,\ldots,n }(h)=0$. The 
desired conclusion now follows from Corollary~\ref{corollary: h large limit}.
\end{proof}

\subsection{Discrete Gaussian free fields and large thresholds}
\label{subsection:DGFFlargeh}

In this section, our main goal will be to prove Theorem~\ref{theorem: strict dgff and large h}.
Note that all our random  vectors in this section will be fully supported which we know is anyway necessary 
in view of Corollary~\ref{corollary: n=2 and h large}.

Before we continue, we remind the reader that a  Gaussian vector \( X \sim N(0,A ) \) is a discrete Gaussian free field if and only if
\begin{enumerate}[(i)]
	\item \( A \) is a block matrix with strictly positive blocks,
	\item \( A \) is an inverse Stieltjes matrix,
	\item \( A \) satisfies the weak Savage condition, i.e.\ \( \mathbf{1}^TA^{-1} \geq \mathbf{0} \), and
	\item for at least one row \( i \) in each block of \( A \), \( \mathbf{1}^TA^{-1} (i)  > 0 \).
\end{enumerate}
This correspondence will be used throughout this whole section.

We first note the following corollaries of   Theorem~\ref{theorem: strict dgff and large h}.

\begin{corollary}
Let \( a \in (0,1) \) and let \( X \coloneqq (X_1,X_2, \ldots, X_n) \) be a standard Gaussian vector with \( \Cov(X_i,X_j) = a^{}   \) for  all \( i<j \). Then \( X^h \) is a color process for all sufficiently large \( h\).
\end{corollary}

\begin{proof}
Let \( A \) be the covariance matrix of \( X \). Then one verifies that for \( i,j \in [n] \) we have 
\[
A^{-1}(i,j) = 
\begin{cases}
\frac{1+(n-2)a}{(1-a)(1+(n-1)a)} &\text{if } i  = j  \cr
\frac{-a}{(1-a)(1+(n-1)a)} &\text{if } i \not = j .
\end{cases}
\]
Consequently, \( A \) is an inverse Stieltjes matrix. Moreover, for all \( j \in [n] \) we have that
\[
\mathbf{1}^T A^{-1}(j) = \frac{1}{1+(n-1)a}.
\]
and hence \( \mathbf{1}^T A^{-1} > \mathbf{0} \). 
Applying Theorem~\ref{theorem: strict dgff and large h}, the desired conclusion follows.
\end{proof}

\begin{corollary}
Let \( a \in (0,1) \) and let  \( X \coloneqq (X_1,X_2, \ldots, X_n) \) be a standard Gaussian vector with \( \Cov(X_i,X_j) = a^{|i-j|}  \) for  all   \( i,j \in [n] \), yielding a Markov chain. Then \( X^h \) is a color process for all sufficiently large \( h\).
\end{corollary}

\begin{proof}
Let \( A \) be the covariance matrix of \( X \). Then one verifies that for \( i,j \in [n] \) we have 
\[
A^{-1}(i,j) = 
\begin{cases}
\frac{1}{1-a^2} &\text{if } i  = j \in \{ 1,n \} \cr
\frac{1+a^2}{1-a^2} &\text{if } i  = j  \not \in \{ 1,n \} \cr
\frac{-a}{1-a^2} &\text{if } |i-j|=1 \cr
0&\text{otherwise.} 
\end{cases}
\]
Consequently, \( A \) is an inverse Stieltjes matrix. Moreover, for all \( j \in [n] \) we have that
\[
\mathbf{1}^T A^{-1}(j) = \begin{cases} \frac{1}{1+a} &\text{if } j \in \{ 1,n \} \cr
 \frac{1-a}{1+a} &\text{if } j \not \in \{ 1,n \} 
\end{cases}
\]
and hence \( \mathbf{1}^T A^{-1} > \mathbf{0} \).  Applying Theorem~\ref{theorem: strict dgff and large h}, the desired conclusion follows.
\end{proof}

We now state and prove a few lemmas that will be needed in the proof of 
Theorem~\ref{theorem: strict dgff and large h}. The first of these will give sufficient conditions 
for \( X^h \) to be a color process for large \( h \) in terms of the decay of the tails of  
\( \nu(1^S) \) for   sets \( S \).  As usual, \(  \ll \) means the relevant ratio goes to zero and \( \asymp \) means things are "equal up to constants".

\begin{lemma}[Theorem 1.6 in~\cite{fs2019b}]\label{lemma: solution sim lemma}
Let \( (\nu_p)_{p \in (0,1)} \) be a family of probability measures on \( \{ 0,1\}^n \). 
Assume that \( \nu_p \) has 
marginals   \(  p\delta_1 + (1-p) \delta_0 \) and that for all \( S \subseteq [n] \) with \( |S| \ge 2 \) 
and all \( k \in S \), as \( p \to 0 \), we have that  
\begin{equation}\label{eq: the main assumption}
p  \nu_p(1^{S \backslash \{ k \}}) \ll \nu_p(1^S) \asymp \nu_p(1^S 0 ^{S^c})
\end{equation}
and
\begin{equation}\label{eq: the main assumption ii}
\lim_{p \to 0}  \sum_{S \subseteq [n] \colon |S| \geq 2} \frac{ \nu_p(1^S0^{S^c}) }{p}< 1 .
 \end{equation}
Then \( X_p \sim \nu_p \) is a color process for all sufficiently small \( p > 0 \).
\end{lemma}

\begin{lemma} \label{lemma: subset of DGFF is DGFF}
Let \( X \coloneqq (X_1,X_2, \ldots, X_n) \) be a standard Gaussian vector with strictly positive, 
positive definite covariance matrix \( A \). Assume further that \( A \) is an inverse Stieltjes matrix 
and that \( \mathbf{1}^T A^{-1} \geq \mathbf{0 }\).  Then for each \( S \subseteq [n] \), the 
covariance matrix \( A_S \) of \( X_S \coloneqq  (X_i)_{i \in S} \) is a strictly positive, 
positive definite inverse Stieltjes matrix with \( \mathbf{1}^T A_S^{-1} \geq \mathbf{0} \).
\end{lemma}

\begin{remark}
The main part of the proof of this lemma consists of showing that if the weak Savage condition holds 
for a matrix \( A \) which is an inverse Stieltjes matrix, then the weak Savage condition will 
also hold for any principal submatrix. Without the additional assumption that \( A \) is an 
inverse Stieltjes matrix, this will not be true. To see this, take e.g. 
   \[
 A = 
 \begin{pmatrix}
1& 0.81 & 0.51& 0.4 \\
0.81 & 1& 0.3&  0.5 \\
  0.51 &0.3&  1& 0.5 \\
  0.4& 0.5& 0.5& 1\\
  \end{pmatrix}.
  \]
One can verify that \( A \) is a positive definite  matrix for which the Savage condition holds, but that the Savage condition does not hold for the principal submatrix corresponding to the first three rows and columns.
\end{remark}

\begin{remark}
Lemma~\ref{lemma: subset of DGFF is DGFF} essentially proves that if \( X \) is a DGFF, then for any \( S \subseteq [n] \),  \( X_S \coloneqq  (X_i)_{i \in S} \) is also a DGFF.
\end{remark}

\begin{proof}[Proof of Lemma~\ref{lemma: subset of DGFF is DGFF}]
By induction, it suffices to show that the conclusion of the lemma holds for \( S \) of the form  \( [n] \backslash \{ k \} \) for some \( k \in [n] \).
To this end, fix \( k \in [n] \). Clearly, \( A_{[n]\backslash \{ k \}} \) is a positive and positive definite matrix. By a lemma on page 328 in~\cite{m1972},  \( A_{[n]\backslash \{ k \}} \) is also an inverse Stieltjes matrix. 
Next,  let \( (b_{ij}) \coloneqq A^{-1} \). Since \( A \) is positive definite, so is \( A^{-1} \), and hence  \( b_{kk} = e_k^T A^{-1} e_k > 0 \).
Next, since \( b_{kk} >0  \),  for \( i,j\in [n] \backslash \{ k \} \), it is well known that
\[
A_{[n] \backslash \{ k \}}^{-1} (i,j) = b_{ij} - \frac{b_{ik}b_{jk}}{b_{kk}}
\]
and hence for $j\neq k$
\begin{equation}\label{eq: Savage equalities}
\begin{split}
\mathbf{1}^T A^{-1}_{[n] \backslash \{ k \}}(j)&= \sum_{i \in [n] \backslash \{ k \}} \left(  b_{ij}  - \frac{b_{ik}b_{jk}}{b_{kk}} \right)
= \sum_{i \in [n]  } \left(  b_{ij}  - \frac{b_{ik}b_{jk}}{b_{kk}} \right)
\\&=
\frac{ \left( \sum_{i \in [n]  } b_{ij}\right) b_{kk}- \left( \sum_{i \in [n]  } b_{ik}\right) b_{jk} }{b_{kk}}
\\&=
\frac{\mathbf{1}^T A^{-1}(j)  b_{kk}-  \mathbf{1}^TA^{-1}(k)b_{jk} }{b_{kk}}
\\&=
\mathbf{1}^T A^{-1}(j) - 
\frac{\mathbf{1}^TA^{-1}(k)b_{jk} }{b_{kk}}.
\end{split}
\end{equation}
Since \( b_{jk} \leq 0 \),  \( b_{kk}>0 \) and \( \mathbf{1}^T A^{-1}(k) \geq 0\) , we obtain the inequality
\begin{align*}
\mathbf{1}^T A^{-1}_{[n] \backslash \{ k \}}(j) \geq
\mathbf{1}^T A^{-1}(j) .
\end{align*}
Since this holds for all \( j \neq k \), the desired conclusion follows.
\end{proof}

The following lemma follows from special cases of Theorems 2.1 and 2.2 in~\cite{dm2001} and Theorem~3.1 in~\cite{h2005}. This will be needed here and also in the
proofs of some lemmas which will be used in the proof of Theorem~\ref{theorem: Gaussian critical 3d}.

\begin{lemma}\label{lemma: Gaussian cube tails II}
Let \( X \) be a fully supported \(n \)-dimensional standard Gaussian vector with positive definite covariance matrix 
\( A  = (a_{ij})\).   If the vector \( \alpha \coloneqq  \mathbf{1}^T A^{-1} \) has no zero component,
then as \( h \to \infty \) one has that
\[
\nu_{(I(\alpha(i)>0))_i}(h) \sim \frac{1}{(2\pi)^{n/2} \sqrt{\det A} \cdot (\prod_{i=1}^n |\alpha(i)|) \cdot h^n} \cdot \exp\left( -\frac{h^2}{2} \cdot \mathbf{1}^T A^{-1} \mathbf{1}\right).
\]
Furthermore if \( \mathbf{1}^T A^{-1} (1) = 0\) , then 
\[\lim_{h \to \infty} \frac{\nu_{1^{n}}(h)}{\nu_{\cdot 1^{n-1}}(h)} = \frac{1}{2}.
\]
\end{lemma}

We note that if \( n = 3 \), then assuming $\alpha(1) \le \alpha(2) \le \alpha(3)$, 
then it is immediate to check that $\alpha(2)$ and $\alpha(3)$ are strictly positive, while
$\alpha(1)$ can be negative, zero or positive. 
 
\begin{lemma}\label{lemma: conditions hold for DGFF}
Let \( X \coloneqq (X_1,X_2, \ldots, X_n) \) be a standard Gaussian vector with strictly positive, positive 
definite covariance matrix \( A \) which is an inverse Stieltjes matrix and satisfies 
\( \mathbf{1}^T A^{-1} \geq \mathbf{0} \). 
Then for any \( S \subseteq [n] \) with \( |S|\ge 2\) 
and \( k \in S \), 
as   \( h \to \infty \), we have \begin{equation}\label{eq: three cases for DGFF}
p  \nu_h(1^{S \backslash \{ k \}}) \ll \nu_h(1^S) \asymp \nu_h(1^S 0^{[n] \backslash S}) .
\end{equation}
\end{lemma}

\begin{proof} 
Let \( S \subseteq [n] \) and define  \( X_S \coloneqq (X_i)_{i \in S } \). Let \( A_S \) be the covariance 
matrix of \( X_S \). By Lemma~\ref{lemma: subset of DGFF is DGFF}, the  matrix \( A_S \) is a strictly 
positive, positive definite inverse Stieltjes matrix which satisfies 
\( \mathbf{1}^T A^{-1}_S \geq \mathbf{0} \).
To simplify notation, let   \( (a_{ij}^{(S)}) \coloneqq A_S\) and \( (b_{ij}^{(S)}) \coloneqq A_S^{-1} \) .
The rest of the proof of this lemma will be divided into several steps
~\paragraph{Step 1.}
Fix \( S \subseteq [n] \) with \( |S| \ge 2 \) and \( k \in S \). In this step, we will prove the inequality 
\begin{equation}\label{eq: goal in first step}
b_{kk}^{(S)}  >  \Biggl( \sum_{i \in S }  b_{ki}^{(S)}   \Biggr)^2 
\end{equation}
or equivalently
\begin{equation}\label{eq: goal in first stepequiv}
 \sum_{i \in S } b_{ki}^{(S)}   >  \Biggl( \sum_{i \in S }  b_{ki}^{(S)}   \Biggr)^2  + 
\sum_{i \in S \backslash \{ k \}} b_{ki}^{(S)} .
\end{equation}
To this end, note first that since \(   (b_{ij}^{(S)}) \) is the inverse of \(  (a_{ij}^{(S)} ) \), 
we have that
\[
1 = \sum_{i \in S} a_{ik}^{(S)}  b_{ki}^{(S)} .
\]
Since \( X \) is a standard Gaussian vector, we have that  \( a_{kk}^{(S)}  = 1 \) and  that 
\( a_{ki}^{(S)}  <1 \) if \( i \in S \backslash \{ k \}\). Moreover, since \( A_S \) is a positive 
definite inverse Stieltjes matrix by Lemma~\ref{lemma: subset of DGFF is DGFF}, we have that
 \( b_{kk}^{(S)} >0 \) and that  \( b_{ji}^{(S)}  \leq 0 \) for \( i \not = j\). In addition, since 
\( a_{ij}^{(S)} > 0 \) for all \( i,j \in S \), we also obtain that 
\begin{equation}\label{eq: neg.sum}
 \sum_{i \in S \backslash \{ k \} } b_{ki}^{(S)} <0 . 
\end{equation}
Combining these observations, we have
\[
1 = \sum_{i \in S} a_{ik}^{(S)}  b_{ki}^{(S)}  = b_{kk}^{(S)}  + \sum_{i \in S \backslash \{ k \} } a_{ik}^{(S)}  b_{ki}^{(S)}  > b_{kk}^{(S)}  +  \sum_{i \in S \backslash \{ k \}} b_{ki}^{(S)}  \cdot 1 = \sum_{i \in 
S} b_{ki}^{(S)} .
\]
Since \( \sum_{i \in S} b_{ki}^{(S)}  = \mathbf{1}^T A_S^{-1} (k)\geq 0 \), it follows that
\[
1 > \sum_{i \in S} b_{ki}^{(S)}  \geq 0 .
\]
This implies in particular that
\[
 \sum_{i \in S} b^{(S)} _{ki}  \geq \Biggl(  \sum_{i \in S} b_{ki}^{(S)}   \Biggr)^2
\]
with equality if and only if \(  \mathbf{1}^T A_S^{-1} (k) =  0 \). This last equation, together with~\eqref{eq: neg.sum} implies~\eqref{eq: goal in first step}, as desired.

\paragraph{Step 2.} In this step, we will prove that for  all \( S \subseteq [n] \)
with \( |S| \ge 2 \) and \( k \in S \), we have
\begin{equation}\label{eq: pos def fors inequality}
\mathbf{1}^T A^{-1}_{S \backslash \{ k \}} \mathbf{1} \leq   \mathbf{1}^T A^{-1}_{S } \mathbf{1} < 1+  \mathbf{1}^T A^{-1}_{S \backslash \{ k \}} \mathbf{1}
\end{equation}
with the first inequality being strict if and only if \( \mathbf{1}^T A^{-1}_{S}(k) > 0 \).
To this end, note first that since \( A \) is positive definite, so is \( A_S \) and 
\( A_{S \backslash \{ k \}} \). So, as before,
\( b_{kk}^{(S)} = e_k^T A_S^{-1} e_k > 0 \) and if \( i,j \in S \backslash \{ k \} \) then 
\[
A_{S \backslash \{ k \}}^{-1}(i,j) = b_{ij}^{(S)} - \frac{b_{ik}^{(S)}b_{jk}^{(S)}}{b_{kk}^{(S)}}.
\]
Using this, we obtain
\[
\mathbf{1}^T A^{-1}_{S \backslash \{ k \}} \mathbf{1} = \sum_{i,j \in S \backslash \{ k \}} \Biggl(  b_{ij}^{(S)} - \frac{b_{ik}^{(S)}b_{jk}^{(S)}}{b_{kk}^{(S)}} \Biggr)
\]
and
\begin{equation}\label{eq: total sum}
\begin{split}
\mathbf{1}^T A_S^{-1} \mathbf{1} &= \sum_{i,j \in S} b_{ij}^{(S)} 
\\&=  \mathbf{1}^T A^{-1}_{S \backslash \{ k \}} \mathbf{1}  + \Biggl(  \sum_{i,j \in S \backslash \{ k \}}    \frac{b_{ik}^{(S)}b_{jk}^{(S)}}{b_{kk}^{(S)}} \Biggr)  + 2 \Biggl( \sum_{i \in S \backslash \{ k \} }b_{ik}^{(S)}\Biggr) +b_{kk}^{(S)} 
\\&=  \mathbf{1}^T A^{-1}_{S \backslash \{ k \}} \mathbf{1}  +\frac{ \left( \sum_{i \in S}  b_{ik}^{(S)} \right)^2    }{b_{kk}^{(S)}}.
\end{split}
\end{equation}
Recalling that \( b_{kk}^{(S)}>0 \)  and using the conclusion of Step 1,~\eqref{eq: pos def fors inequality} follows, which concludes Step~2. 

\paragraph{Step 3.}
For \( S \subseteq [n] \), define 
\( J_S \coloneqq \{ j \in S \colon \mathbf{1}^T A_S^{-1}(j) = 0 \} \).  Note that since  \( A_S \) is 
positive definite, we have that \( \mathbf{1}^T A^{-1}_S \mathbf{1} > 0 \)  and hence \( J_S \not = S \).
In this step, we show that the following hold for any sets \( S' \subseteq S \subseteq [n] \).
\begin{enumerate}[(i)]
\item If \( i \in J_S \), then \( J_{ S \backslash \{ i \}} = J_S \backslash \{ i \}. \)
\item \( |S\backslash J_{S}| \ge 2.\)
\item  \( J_{S'} \subseteq J_S. \)
\item   \( \mathbf{1}^T A^{-1}_{(S \backslash J_{S})\backslash S'} > \mathbf{0}.\)
\item The set \( \{ T \subseteq [n] \backslash S \colon T \subseteq J_{S \cup T} \} \)  is a power set of some set.
\end{enumerate}
To see that (i) holds, note first that by~\eqref{eq: Savage equalities}, for any set \( S \subseteq [n] \) 
and any distinct \( i,j \in S \) we have that
\begin{equation}\label{eq: Savage equality II}
\mathbf{1}^T A^{-1}_{S \backslash \{ i \}}(j)=
\mathbf{1}^T A_{S}^{-1}(j) -
\mathbf{1}^TA_{S}^{-1}(i) \cdot \frac{b^{(S)}_{ji} }{b^{(S)}_{ii}}.
\end{equation}
From this (i) immediately follows.

For (ii), one first checks that if $S$ has 2 elements, then \(J_S=\emptyset\). For larger $S$, we argue
by induction. Take \(i\in J_S\). By induction, 
\( (S\backslash \{ i \}) \backslash (J_{ S \backslash \{ i \}})\ge 2\) which by (i) implies
\( (S\backslash \{ i \}) \backslash (J_S \backslash \{ i \})\ge 2\), which yields the result for $S$.

Next, by Lemma~\ref{lemma: subset of DGFF is DGFF}, \( A_{S} \) is an inverse Stieltjes matrix which satisfies 
\( \mathbf{1}^T A^{-1}_S \geq \mathbf{0} \). In particular, this implies that \( b^{(S)}_{ji} \leq 0 \) 
and \( \mathbf{1}^T A_S^{-1}(i) \geq 0 \), and hence it follows from~\eqref{eq: Savage equality II} that
\begin{equation}\label{eq: Savage grows}
\mathbf{1}^T A^{-1}_{S \backslash \{ i \}}(j) \geq \mathbf{1}^T A_{S}^{-1}(j) \geq 0 ;
\end{equation}
(iii) follows.

Next, (iv) follows easily from (iii).

We will now show that (v) holds. To simplify notation, let  
\[ \mathcal{Z}_S \coloneqq \{ T \subseteq [n] \backslash S \colon T \subseteq J_{S \cup T} \} .\] 
It suffices to show that if \( T_1, T_2 \in \mathcal{Z}_S\) and \( i \in T_1 \), then 
\begin{enumerate}[(a)]
\item \( T_1 \backslash \{ i \}  \in \mathcal{Z}_S  \)
\item \( T_2 \cup \{ i \} \in\mathcal{Z}_S  \). 
\end{enumerate}
To see that (a) holds, fix \( T_1 \in \mathcal{Z}_S\) and \( i \in T_1 \). By the definition of \( \mathcal{Z}_S \), this implies that \( T_1 \subseteq J_{S \cup T_1} \), and hence \( T_1 \backslash \{ i \} \subseteq J_{S \cup T_1} \backslash \{ i \} \). Since \( i \in  T_1 \subseteq J_{S \cup T_1} \) we have \( i \in J_{S \cup T_1}  \), and hence by (i) we have \( J_{S \cup T_1} \backslash \{ i \} = J_{S \cup T_1 \backslash \{ i \}}   \). Combining these observations, we obtain \( T_1 \backslash \{ i \} \subseteq J_{S \cup T_1 \backslash \{ i \}} \), and hence \( T_1 \backslash \{ i \} \in \mathcal{Z}_S \) as desired. This concludes the proof of (a).

To see that (b) holds, fix \( T_1, T_2 \in \mathcal{Z}_2 \) and \( i \in T_1   \). By the definition of \( \mathcal{Z}_S \), we have \( T_1 \subseteq J_{S \cup T_1} \) and \( T_2 \subseteq J_{S \cup T_2} \). Since \( T_1 \subseteq J_{S\cup T_1} \), by applying (i) several times, we obtain
\[
J_{S \cup \{ i \}}= J_{(S \cup T_1) \backslash  (T_1 \backslash \{ i \} )} = J_{S \cup T_1} \backslash \bigl(T_1 \backslash \{ i \}\bigr).
\]
Since \( i \in T_1 \subseteq J_{S \cup T_1} \), this implies in particular that \( i \in J_{S \cup \{ i \}}\).
By (iii), we have that \( J_{S \cup \{ i \}}  \cup J_{S \cup T_2 } \subseteq J_{S \cup T_2 \cup \{ i \}} \). 
Since \( i \in  J_{S \cup \{ i \}} \) and \( T_2 \subseteq J_{S \cup T_2} \), it follows \( T_2 \cup \{ i \} \subseteq J_{S \cup T_2 \cup \{ i \}} \), and hence \( T_2 \cup \{ i \} \in \mathcal{Z}_S \) as desired.

\paragraph{Step 4.}
In this step, we will now show that for any \( S \subseteq [n] \) with \( |S| \ge 2 \)
and \( k \in S \), as \( h \to \infty \), we have that
\[ p \nu_h(1^{S\backslash  \{ k \}})    \ll \nu_h(1^{S})  .\]
To this end, fix  \( S \subseteq [n] \) and let \( J_S \) be  as in Step 3. By Step 2, for any \( k \in S \backslash J_S \), we have that
\begin{equation}\label{eq: step 4 goal}
   \mathbf{1}^T A^{-1}_{S \backslash J_S } \mathbf{1} < 1+  \mathbf{1}^T A^{-1}_{S \backslash (J_S \cup \{ k \})} \mathbf{1}.
\end{equation}
Since this trivially holds for \( k \in J_S \), it follows that these inequalities in fact hold for all \( k \in S\). 
Now fix \( k \in S \). By Step 3 (iv) we have that \( \mathbf{1}^T A_{S \backslash J_S}^{-1} > \mathbf{0} \) and  
\( \mathbf{1}^T A_{S \backslash (J_S\cup \{ k \})}^{-1} > \mathbf{0} \), and hence by applying the first part of
Lemma~\ref{lemma: Gaussian cube tails II} and using~\eqref{eq: step 4 goal}, it follows that as 
\( h \to \infty\), we have
\[
p \nu_h(1^{S\backslash (J_S \cup \{ k \})}) \ll \nu_h(1^{S\backslash J_S }). 
\]
Applying the second part of Lemma~\ref{lemma: Gaussian cube tails II} several times together with
Step 3 (iii), we see that 
\begin{equation}\label{eq: needed asymp}
\nu_h(1^{S\backslash J_S })  \sim 2^{|J_S|}\nu_h(1^{S}) 
\end{equation}
Using this, it follows that as \( h \to \infty\),
\[
p \nu_h(1^{S\backslash \{ k \}})  \le p \nu_h(1^{S\backslash (J_S \cup \{ k \})}) \ll \nu_h(1^{S\backslash J_S })
\asymp \nu_h(1^{S}) 
\]
and hence the desired conclusion holds.

\paragraph{Step 5.}
In this step, we show that for each \( S \subseteq [n] \) with \( |S| \ge 2 \), as \( h \to \infty \), we have that 
\begin{equation}\label{eq: step 5 goal}
 \nu_h(1^{S}) \asymp \nu_h(1^S0^{ [n] \backslash S}) .
\end{equation}
To this end, fix \( S \subseteq [n] \). By an inclusion-exclusion argument, we see that
\[
\nu_h(1^S 0^{ [n] \backslash S}) = \sum_{T \subseteq [n] \backslash S} \nu_h(1^{S \cup T}) (-1)^{|T|}.
\]
For each \( T \subseteq [n] \backslash S \), let \( J_{S \cup T} \) be as in Step 3. 
By~\eqref{eq: needed asymp} applied to \( S \cup T \), it follows that  
\[
 \nu_h(1^{S \cup T})  \sim 2^{-|J_{S \cup T}|}  \nu_h(1^{(S \cup T)\backslash J_{S \cup T}}).
\]
Now note that by~\eqref{eq: total sum} and Step 3 (iii), we have that
\[
\mathbf{1}^T A^{-1}_{(S \cup T)\backslash J_{S \cup T}} \mathbf{1} = \mathbf{1}^T A^{-1}_{S \cup T} \mathbf{1}.
\]
~\eqref{eq: pos def fors inequality} and induction now implies that
\[
\mathbf{1}^T A^{-1}_{(S \cup T)\backslash J_{S \cup T}} \mathbf{1} = \mathbf{1}^T A^{-1}_{S \cup T} \mathbf{1} \geq \mathbf{1}^T A^{-1}_{S} \mathbf{1} = \mathbf{1}^T A^{-1}_{S \backslash J_{S }} \mathbf{1}
\]
with equality if and only if \( T \subseteq J_{S \cup T} \). Since by Step 3 (iv) we have that 
\( \mathbf{1}^T A^{-1}_{(S \cup T)\backslash J_{S \cup T}} > \mathbf{0} \),  if we combine these observations 
and apply Lemma~\ref{lemma: Gaussian cube tails II}, it follows that
\[
\nu_h(1^S 0^{S^c}) \sim \sum_{T \subseteq S^c \colon T \subseteq J_{S \cup T}} \nu_h(1^{S \cup T}) (-1)^{|T|} 
\sim \nu_h(1^{S  }) \sum_{T \subseteq S^c \colon T \subseteq J_{S \cup T}} 2^{-|T|}  (-1)^{|T|}.
\]
By Step 3 (v),   the  set \( \{ T \subseteq [n] \backslash S \colon T \subseteq J_{S \cup T} \} \) is a power set 
of some set \( S_0 \). Using this, it follows that
\[
\sum_{T \subseteq S^c \colon T \subseteq J_{S \cup T}} 2^{-|T|}  (-1)^{|T|} = \sum_{T \subseteq S_0} 2^{-|T|}  (-1)^{|T|} = (1-2^{-1})^{|S_0|} = 2^{-|S_0|}
\]
and hence~\eqref{eq: step 5 goal} holds.

Since Step 4 and Step 5 together give the conclusions of the lemma, this concludes the proof.
\end{proof}

\begin{remark}
If we assumed Savage instead of weak Savage, the proof could be somewhat shortened.
\end{remark}

We are now ready to give the proof of  Theorem~\ref{theorem: strict dgff and large h}.
\begin{proof}[Proof of Theorem~\ref{theorem: strict dgff and large h}]
The covariance matrix for a discrete Gaussian free field is a block matrix with each block satisfying
the assumptions of Lemma~\ref{lemma: conditions hold for DGFF}. Hence, restricting to a block, we have that for all
\( S \) within this block with \( |S| \ge 2 \) and for \( k \in S \), we have that
\[
p \nu_h(1^{S \backslash \{ k \}}) \ll \nu_h(1^S) \asymp \nu_h(1^S 0^{S^c}).
\]
The second condition in Lemma~\ref{lemma: solution sim lemma} trivially holds and hence
applying this lemma, we obtain conclude that for large $h$, the threshold Gaussian corresponding to this fixed
block is a color process. Since the full process is independent over the different blocks, we easily obtain the desired
result for the full process.
\end{proof}

\section{General results for small and large thresholds for \texorpdfstring{$n=3$}{n=3} in the Gaussian
case} 
\label{section:smalllarge}

 When \( Y \) is a \( \{ 0,1\} \)-valued 3-dimensional random vector, and \( \nu \) is the corresponding probability measure, we know from Theorem  2.1(C) in~\cite{st2017} (see also Theorem 1.4 in~\cite{fs2019b}) that   \( Y \) has a unique signed color representation \( (q_\sigma)_{\sigma \in \mathcal{B}_3} \). It is easy to verify that this representation is given by
\begin{equation}\label{eq: gen sol when n is 3 and h is nonzero,h}
 \begin{cases}
q_{1,2,3} &= \frac{\nu_{100}-\nu_{011}}{(1-p)p(1-2p)}\cr
q_{12,3} &= \frac{(1-p)\nu_{110}-p \nu_{001}}{(1-p)p(1-2p)}\cr
q_{13,2} &= \frac{(1-p)\nu_{101}-p\nu_{010}}{(1-p)p(1-2p)}\cr
q_{1,23} &= \frac{(1-p)\nu_{011}-p\nu_{100}}{(1-p)p(1-2p)}\cr
q_{123} &= 1-\frac{\nu_1 \nu_{000} - \nu_0\nu_{111}}{(1-p)p(1-2p)}.
\end{cases}
\end{equation}
This implies in particular that \( Y \) has a color representation if and only if   \( (q_\sigma)_{\sigma \in \mathcal{B}_3} \) is non-negative.

\subsection{\texorpdfstring{$h$}{h} small}

Our next result describes the behavior of \( (q_\sigma)_{\sigma \in \mathcal{B}_3} \) when \( Y = X^h \) for a Gaussian vector \( X \), and \( h > 0 \) is small.

\begin{theorem}\label{theorem: small h}
%
Let \( X \) be a three-dimensional standard Gaussian vector with covariance matrix \( A = (a_{ij}) \) and \( \theta_{ij} \coloneqq \arccos a_{ij} \). Further, let \( (\nu_\rho(h))_{\rho \in \{ 0,1 \}^3} \) be the probability measure corresponding to \( X^h \) and  let \( (q_\sigma)_{\sigma \in \mathcal{B}_n } \) be given by~\eqref{eq: gen sol when n is 3 and h is nonzero,h}. Then 
\begin{equation}\label{eq: theorem goal}
\begin{cases}
\lim_{h \to 0} q_{1,2,3}(h) =2  -\frac{2\arccos \left( \frac{  \det A }{\prod_{i<j}(1+a_{ij})}-1\right)}{\pi}  \cr
\lim_{h \to 0} q_{12,3}(h)  =  \frac{\theta_{13} + \theta_{23} - \theta_{12}}{\pi} -1 +\frac{\arccos \left( \frac{  \det A }{\prod_{i<j}(1+a_{ij})}-1\right)}{\pi} \cr
\lim_{h \to 0} q_{13,2}(h)  =  \frac{\theta_{12} + \theta_{23} - \theta_{13}}{\pi}   -1 +\frac{\arccos \left( \frac{  \det A }{\prod_{i<j}(1+a_{ij})}-1\right)}{\pi} \cr
\lim_{h \to 0} q_{1,23}(h)  =   \frac{\theta_{12} + \theta_{13} - \theta_{23}}{\pi}  -1 +\frac{\arccos \left( \frac{  \det A }{\prod_{i<j}(1+a_{ij})}-1\right)}{\pi} \cr
\lim_{h \to 0} q_{123}(h)  = 2  - \frac{\theta_{12} + \theta_{13} +\theta_{23}}{\pi}-\frac{\arccos \left( \frac{  \det A }{\prod_{i<j}(1+a_{ij})}-1\right)}{\pi}.
\end{cases}
\end{equation}

\end{theorem}

\begin{proof} This proof will be divided into two steps.
\paragraph{Step 1.}
In this step, we will prove that
\begin{equation}\label{eq: first goal}
\begin{cases}
\lim_{h \to 0} q_{1,2,3}(h) =4-\frac{4   \nu_{000}'(0) }{\nu_0'(0)}  \cr
\lim_{h \to 0} q_{12,3}(h)  = 4\nu_{001}(0)  -2+\frac{2   \nu_{000}'(0) }{\nu_0'(0)}\cr
\lim_{h \to 0} q_{13,2}(h)  = 4\nu_{010}(0)  -  2+\frac{2   \nu_{000}'(0) }{\nu_0'(0)}\cr
\lim_{h \to 0} q_{1,23}(h)  = 4\nu_{100}(0)  -  2+\frac{2   \nu_{000}'(0) }{\nu_0'(0)}\cr
\lim_{h \to 0} q_{123}(h)  = 4\nu_{000}(0)+1-\frac{2   \nu_{000}'(0) }{\nu_0'(0)}  .
\end{cases}
\end{equation}
To this end, note first that by~\eqref{eq: gen sol when n is 3 and h is nonzero,h},
\[
q_{1,2,3}(h)= \frac{\nu_{100}(h)-\nu_{011}(h)}{\nu_0(h)\nu_1(h)(\nu_0(h)-\nu_1(h))}.
\]
Since \( \nu_\rho \) is differentiable at zero, it follows that
\begin{align*}
\lim_{h \to 0}  q_{1,2,3}(h) &= \lim_{h \to 0} \frac{\nu_{100}(h)-\nu_{011}(h)}{\nu_0(h)\nu_1(h)(\nu_0(h)-\nu_1(h))}
\\&= 4 \lim_{h \to 0} \frac{\nu_{100}(h)-\nu_{100}(-h)}{2h} \cdot \frac{2h}{\nu_0(h) - \nu_0(-h)}
= \frac{4   \nu_{100}'(0) }{\nu_0'(0)} .
\end{align*}
Similarly, again using~\eqref{eq: gen sol when n is 3 and h is nonzero,h}, one has that
\begin{align*}
\lim_{h \to 0}  q_{12,3}(h) &= \lim_{h \to 0} \frac{\nu_0(h)\nu_{110}(h)-\nu_1(h)\nu_{001}(h)}{\nu_0(h)\nu_1(h)(\nu_0(h)-\nu_1(h))}
\\&= 4 \lim_{h \to 0} \left( \frac{\nu_0(h)\nu_{110}(h)-\nu_0(-h)\nu_{110}(-h)}{2h} \right) \cdot \frac{2h}{\nu_0(h) - \nu_0(-h)}
\\&= 4 \cdot \frac{\nu_0'(0)\nu_{110}(0)+\nu_0(0)\nu'_{110}(0)}{\nu_0'(0)} 
=4\nu_{110}(0)   +    \frac{2\nu'_{110}(0)}{\nu_0'(0)} 
\\&=4\nu_{001}(0)   -    \frac{2\nu'_{001}(0)}{\nu_0'(0)} .
\end{align*}
 If we can show that 
 \begin{equation}\label{eq: p00 symmetry}
 \nu_{\cdot 00}'(0)= \nu_{ 0\cdot0}'(0)=\nu_{ 00\cdot}'(0)  = \nu_0'(0) 
 \end{equation}
 then~\eqref{eq: first goal} will follow using symmetry and the fact that \( \sum q_{\sigma} = 1 \). 
 To see that~\eqref{eq: p00 symmetry} holds, let \( f \) be the probability density function of \( (X_1,X_2) \) and note that \( \nu_0'(x) \) is the marginal density of both \( X_1 \) and \( X_2 \).   Then for any \( h_1,h_2 \in \mathbb{R} \)  we have that
   \begin{align*}
&\frac{d}{dh_2} P(X_1 \leq h_1, X_2 \leq h_2)
= 
 \frac{d}{dh_2} \int_{-\infty}^{h_1} \int_{-\infty}^{h_2} f(x_1, x_2) \,dx_2\, dx_{1}
 \\ &\qquad = 
\int_{-\infty}^{h_1} f(x_1,h_2) \,dx_1
=
P(X_1 \leq h_1 \mid X_2 = h_2) \cdot \nu_0'(h_2).
 \end{align*}
 Differentiating with respect to \( h_1 \) in the same way and then setting \( h_1 = h_2 = 0 \), it follows that
\[
\nu'_{00\cdot}(0) = \nu_0'(0) \left( P(X_1\leq 0 \mid X_2 = 0) + P(X_2\leq 0 \mid X_1 = 0) \right).
\]
By symmetry, the two summands are each equal to \( 1/2 \), and hence \( \nu_{00\cdot}'(0) = \nu_0'(0) \) 
as desired. The other equalities follow by an analogous argument.

\paragraph{Step 2.}
To obtain~\eqref{eq: theorem goal} from~\eqref{eq: first goal}, note first that by an analogous argument as above, one obtains in general 
that 
\[
\frac{\nu'_{000}(0) }{ \nu_0'(0)}=   P(X_2,X_3\leq 0 \mid X_1 = 0) +P(X_1,X_3\leq 0 \mid X_2 = 0) +P(X_1,X_2\leq 0 \mid X_3 = 0) .
\]
Using basic facts about Gaussian vectors, one has that \( (X_2,X_3) \mid X_1 = 0 \) is a Gaussian vector with 
correlation 
\[
\alpha =   \frac{a_{23}-a_{12}a_{13}}{\sqrt{(1-a_{12}^2)(1-a_{13}^2)}}.
\]
Using~\eqref{eq: p00}, it follows that
\[
P(X_2\leq 0,   X_3 \leq 0 \mid X_1 = 0) = \frac{1}{2}-\frac{\arccos \left( \frac{a_{23}-a_{12}a_{13}}{\sqrt{(1-a_{12}^2)(1-a_{13}^2)}}\right)}{2\pi}
\]
and hence, by symmetry, we obtain
\[
\frac{\nu'_{000}(0)}{\nu_0'(0)} = \frac{3}{2} -
 \frac{ \arccos \left( \frac{a_{23}-a_{12}a_{13}}{\sqrt{(1-a_{12}^2)(1-a_{13}^2)}} \right) +\arccos \left( \frac{a_{13}-a_{12}a_{23}}{\sqrt{(1-a_{12}^2)(1-a_{23}^2)}} \right) +\arccos \left( \frac{a_{12}-a_{13}a_{23}}{\sqrt{(1-a_{13}^2)(1-a_{23}^2)}}\right)}{2\pi}.
\]
Now recall that for any \( \alpha, \beta \in [-1,1] \) we have that
\[
\arccos \alpha + \arccos \beta = 
\begin{cases}
\arccos (\alpha \beta - \sqrt{(1-\alpha^2)(1-\beta^2)}) &\text{if } \alpha + \beta \geq 0 \cr
2\pi - \arccos (\alpha \beta - \sqrt{(1-\alpha^2)(1-\beta^2)}) &\text{if } \alpha + \beta \leq 0 .
\end{cases}
\]
and hence if \( \alpha, \beta \in [-1,1] \) satisfies \( \alpha + \beta \geq 0 \) and \(  \alpha \beta - \sqrt{1-\alpha^2} \sqrt{1-\beta^2} + \gamma  \leq 0 \), then
\begin{align*}
&\arccos \alpha + \arccos \beta + \arccos \gamma 
\\&\qquad = 2\pi - \arccos \left(\alpha \beta \gamma - \alpha \sqrt{(1-\beta^2)(1-\gamma^2)} - \beta \sqrt{(1-\alpha^2)(1-\gamma^2)} - \gamma \sqrt{(1-\alpha^2)(1-\beta^2)}\right) .
\end{align*}
Now let
\[
\begin{cases}
\alpha =  \frac{a_{23}-a_{12}a_{13}}{\sqrt{(1-a_{12}^2)(1-a_{13}^2)}} \cr
\beta =  \frac{a_{13}-a_{12}a_{23}}{\sqrt{(1-a_{12}^2)(1-a_{23}^2)}} \cr
\gamma =  \frac{a_{12}-a_{13}a_{23}}{\sqrt{(1-a_{13}^2)(1-a_{23}^2)}}.
\end{cases}
\]
Using that as \( A \) is positive definite, then \( a_{12} \leq a_{13}a_{23} + \sqrt{(1-a_{13}^2)(1-a_{23}^2)}\), it follows that we indeed have that \( \alpha + \beta \geq 0 \). Moreover, with some work, one verifies that 
\[
 \alpha \beta - \sqrt{1-\alpha^2} \sqrt{1-\beta^2} + \gamma 
 =
 \frac{-\det A}{(1+a_{12})\sqrt{1-a_{13}^2}\sqrt{1-a_{23}^2}} \leq 0
\]
and that
\begin{align*}
&\alpha \beta \gamma - \alpha \sqrt{(1-\beta^2)(1-\gamma^2)} - \beta \sqrt{(1-\alpha^2)(1-\gamma^2)} - \gamma \sqrt{(1-\alpha^2)(1-\beta^2)}
\\&\qquad = \frac{\det A}{\prod_{i<j} (1+a_{ij})} - 1.
\end{align*}
This implies in particular that
\begin{align*}
\frac{\nu'_{000}(0)}{\nu_0'(0)} 
& = \frac{3}{2} -  \frac{\arccos \alpha + \arccos \beta + \arccos \gamma }{2\pi}
\\& = \frac{3}{2} -  \frac{2\pi - \arccos \left( \frac{\det A}{\prod_{i<j} (1+a_{ij})} - 1 \right)  }{2\pi}
\\& = \frac{1}{2} + \frac{   \arccos \left( \frac{\det A}{\prod_{i<j} (1+a_{ij})} - 1 \right)  }{2\pi}.
\end{align*}

Combining this with~\eqref{eq: p000 Gaussian} and~\eqref{eq: first goal}, the desired conclusion  follows.

\end{proof}

\begin{remark}
	For the first part of the proof, one can also apply Theorem 1.7 in~\cite{fs2019b}, but since this does not significantly shorten the proof, we find the current proof more clear.
\end{remark}

We now apply   Theorem~\ref{theorem: small h}   to a few examples.
 \begin{corollary}\label{corollary: small h fully symmetric}
Let \(a \in (0,1)  \) and let \( X \coloneqq (X_1,X_2,X_3) \) be a standard Gaussian vector with  \( \Cov(X_1,X_2) = \Cov(X_1,X_3) = \Cov(X_2,X_3)  = a \). Then \( X^{h} \) is a color process for all sufficiently small \( h \).
 \end{corollary}

\begin{proof}
 Note first that by using Theorem~\ref{theorem: small h}, after a computation, we obtain
 \[
\begin{cases}
\lim_{h \to 0} q_{1,2,3}(h) =2  -\frac{2\arccos \left( \frac{ a(a^2-6a-3) }{(1+a)^3}\right)}{\pi}  \cr
\lim_{h \to 0} q_{12,3}(h)  =  \frac{\arccos a}{\pi} -1 +\frac{\arccos  \left( \frac{ a(a^2-6a-3) }{(1+a)^3}\right)}{\pi} \cr
\lim_{h \to 0} q_{123}(h)  = 2  - \frac{3\arccos a}{\pi}-\frac{\arccos  \left( \frac{ a(a^2-6a-3) }{(1+a)^3}\right)}{\pi}.
\end{cases}
\]
It suffices to show that the above limits are positive.
Since \(\arccos x \in (0,\pi ) \) for all \( x \in (-1,1) \) and \( \arccos x \) is strictly decreasing in \( x \), it follows that the first of these is strictly positive whenever
\[
 \frac{ a(a^2-6a-3) }{(1+a)^3} > -1.
\]
By rearranging, one easily sees this to be true whenever \(a \in(0,1) \).
Next, since \( \pi-\arccos x= \arccos(-x) \) for all \( x \in (0,1) \) it follows that the second limit is strictly positive whenever
\[
  a +  \frac{ a(a^2-6a-3) }{(1+a)^3}=\frac{-a(1-a)(2+5a+a^2)}{(1+a)^3}<0.
 \]
which clearly holds for all \( a\in (0,1) \). To see that \( X^h \) has a color representation for all sufficiently small \( h>0 \), it thus only remains to show that \( \lim_{h \to 0 }q_{123}(h) > 0 \). To this end, first note that this is equivalent to that
\[
3\arccos a + \arccos  \left( \frac{ a(a^2-6a-3)}{(1+a)^3}\right)  <2\pi.
\]
It is easy to verify that we get equality when \( a= 0 \), and hence it would be enough to show that the left hand side is strictly decreasing in \( a \). If we differentiate the left hand side one we obtain, after a detailed computation, that
\[
\frac{3}{(1+a)\sqrt{1 + 2a}} - \frac{3}{\sqrt{1-a^2}}
\]
which is clearly negative for all \( a \in (0,1) \). From this the desired conclusion follows.
\end{proof}

 \begin{corollary}\label{corollary: MC and small h}
Let \(a \in (0,1)  \) and let \( X \coloneqq (X_1,X_2,X_3) \) be a standard Gaussian vector with \( \Cov(X_1,X_2) = \Cov(X_2,X_3)  =a \) and \(\Cov(X_1,X_3) = a^2 \). Then \( X^{h} \)  is a color process  for all sufficiently small \( h \).
 \end{corollary}
 
 \begin{remark}
 With \( X = (X_1,X_2,X_3) \) defined as a above, \( X \) is a Markov chain.
 \end{remark}

 \begin{proof}[Proof of Corollary~\ref{corollary: MC and small h}]
 Note first that by using Theorem~\ref{theorem: small h}, after a computation, we obtain
 \[
\begin{cases}
\lim_{h \to 0} q_{1,2,3}(h) =2  -\frac{2\arccos \left( \frac{  -2a}{1+a^2}\right)}{\pi}  \cr
\lim_{h \to 0} q_{12,3}(h)  =  \frac{\arccos a^2}{\pi} -1 +\frac{\arccos \left( \frac{  -2a}{1+a^2}\right)}{\pi}  \cr
\lim_{h \to 0} q_{13,2}(h)  =  \frac{2\arccos a - \arccos a^2}{\pi}   -1 +\frac{\arccos \left( \frac{  -2a}{1+a^2}\right)}{\pi}  \cr
\lim_{h \to 0} q_{1,23}(h)  =   \frac{\arccos a^2}{\pi}  -1 +\frac{\arccos \left( \frac{  -2a}{1+a^2}\right)}{\pi}  \cr
\lim_{h \to 0} q_{123}(h)  = 2  - \frac{2 \arccos a + \arccos a^2}{\pi}-\frac{\arccos \left( \frac{  -2a}{1+a^2}\right)}{\pi}  .
\end{cases}
\]
It suffices to show that the above limits are positive.
By using the fact that \( \pi - \arccos x = \arccos(-x) \) for all \( x \in (-1,1) \) and the fact that arccosine is a strictly decreasing function, one easily verifies that the first, second and fourth of these are strictly positive for all \( a \in (0,1) \). To see that the third limit is strictly positive for \( a \in (0,1) \), we differentiate this limit  with respect to \( a \)  to obtain
\[
(a\sqrt{1+a^2} + \sqrt{1-a^2} - (1+a^2)) \cdot \frac{2}{(1+a^2)\sqrt{1-a^2}}.
\]
This expression can be equal to zero if and only if
\[
a\sqrt{1+a^2} + \sqrt{1-a^2}= 1+a^2.
\]
Squaring both sides and simplifying, we see that this is equivalent to that
\[
\sqrt{1-a^4} =  a
\]
which in turn is equivalent to that
\[
 1-a^2-a^4 =  0.
\]
This equation clearly has  exactly one solution in \( (0,1) \). Hence in particular, there can be only one maxima or minima in \( (0,1) \). Since \( \lim_{h \to 0} q_{13,2}(h)(a) \) is continuous in \( a \) for all \( a \in [0,1] \), \( \lim_{h \to 0} q_{13,2}(h)(0) =\lim_{h \to 0} q_{13,2}(h)(1)  = 0 \) and one easily verifies that \( \lim_{h \to 0} q_{13,2}(h)(0.5) >0 \) it follows that \( \lim_{h \to 0} q_{13,2}(h)(a) > 0  \) for all \( a \in (0,1) \).

Finally, one easily verifies that the derivative of \( \lim_{h \to 0}q_{123}(h)(a) \)  with respect to \( a \) is given by
\[
(a\sqrt{1+a^2} + (1+a^2)-\sqrt{1-a^2}) \cdot \frac{2}{(1+a^2)\sqrt{1-a^2}}
\]
which has no zeros in \((0,1) \). Since \(  \lim_{h \to 0} q_{123}(h) (0) = 0 \), \(  \lim_{h \to 0} q_{123}(h)(1)=1 \) and  \( \lim_{h \to 0}q_{123}(h)(a) \) is continuous in \( a \), it must be strictly increasing in \( a \) in \( (0,1) \), and hence it follows that \( \lim_{h \to 0 } q_{123}(h)(a)> 0 \) for all \( a \in (0,1) \).
 \end{proof}

\subsection{\texorpdfstring{$h$}{h} large}

Before proving Theorem~\ref{theorem: Gaussian critical 3d}, we start off by giving some interesting applications of it.
\begin{corollary}\label{corollary:4examples}
For each case below, there is at least one   Gaussian vector \( X \) with non-negative correlations 
which satisfies it.
\begin{enumerate}[(i)]
\item \( X^h \) has a color representation for all sufficiently large \( h \) and for all sufficiently small  \( h>0 \). 
\item  \( X^h \) has no color representation for any  sufficiently large \( h \) nor for any   sufficiently small  \( h>0 \). 
\item  \( X^h \) has a color representation for all sufficiently large \( h \) but not for any sufficiently small  \( h>0 \). 
\item  \( X^h \) has a color representation for all sufficiently small   \( h \) but not for any sufficiently large \( h \).
\end{enumerate}
In particular, the property of \( X^h \) being a color process for a fixed \( X \) is not monotone in \( h \) (in either direction) for \( h>0 \).
\end{corollary}

\begin{proof}\textcolor{white}{.}

\begin{enumerate}[(i)]
\item Of course one can take an i.i.d.\ process here. A more interesting example is as follows. 
Let \( X \) be a three-dimensional standard Gaussian vector with 
\(\Cov(X_1,X_2) =\Cov(X_1,X_3) =\Cov(X_2,X_3) = a \in (0,1)  \). By combining 
Corollary~\ref{corollary: small h fully symmetric} and Theorem~\ref{theorem: Gaussian critical 3d}(i), 
it follows that \( X^h \) has a color representation for both sufficiently small and sufficiently large \( h >0 \).
\item Let \( X \) be a three-dimensional Gaussian vector with  
\( \Cov(X_1,X_2) = 0.05 \), \(\Cov(X_1,X_3) = \Cov(X_2,X_3)  = 0.6825  \). One can verify that this 
corresponds to a positive definite covariance matrix. Using Theorem~\ref{theorem: small h}, one
verifies that \( \lim_{h \to 0 } q_{12,3}(h) \approx -0.05 \) and hence \( X^h \) does not have a 
color representation for any sufficiently small \( h \). Using 
Theorem~\ref{theorem: Gaussian critical 3d}, it follows that \( X^h \) does not 
either have a color representation for large \( h \). %
\item Let \( X \) be  a three-dimensional standard Gaussian vector with \(\Cov(X_1,X_2) = 0.1 \), 
\( \Cov(X_1,X_3) = \Cov(X_2,X_3)   = 0.5  \). One can verify that this corresponds to a positive 
definite covariance matrix. Now by Theorem~\ref{theorem: small h}, the limit 
\( \lim_{h \to 0 } q_{12,3}(h) \approx -0.016 \) and hence \( X^h \) does not have a color 
representation for any sufficiently small \( h > 0 \). Next, since the Savage 
condition~\eqref{eq: Savage condition term again} holds, we have that \( X^h \) has a color representation for 
all sufficiently large \( h \) by Theorem~\ref{theorem: Gaussian critical 3d}.
\item  This follows immediately from Theorem~\ref{theorem: 4pointsoncircle}.  
\end{enumerate}
\end{proof}

\begin{example}\label{example: the ab example}
It is illuminating to look at the subset of the set of three-dimensional standard Gaussians for
which at least two of the covariances are equal. So, we let
\( X_{a,b} = (X_1,X_2,X_3) \) be a standard Gaussian vector with covariance matrix 
\[
A = \begin{pmatrix}
1 & a & a \\
a & 1 & b \\
a & b & 1
\end{pmatrix}
\]
for some \( a,b \in (0,1) \). One can verify that \( A \) is positive definite exactly 
when \(  2a^2<1+b \). Applying  Theorem~\ref{theorem: Gaussian critical 3d}, one can check that 
\( X_{a,b} ^h \) is a color process for all sufficiently large \( h \) if and only if either 
\( 2a-1 \leq b \) or \( (2a-1)^2<b \) (note both of these inequalities imply that \(  2a^2<1+b \)).
Cases (i) and (ii) correspond to the first inequality holding and Case (iii) corresponds to the 
first inequality failing and the second inequality holding. For a fixed $h$, the set of parameters 
which yield a color process for threshold $h$ is a closed set. However the set of parameters which 
yield a color process for sufficiently large $h$ is not a closed set; for example, $a=.1$ and 
$b=\epsilon$ belongs to this set for every $\epsilon>0$ but not for $\epsilon=0$.

In Figure~\ref{fig: ab example}, we first draw the regions corresponding to the various cases in 
Theorem~\ref{theorem: Gaussian critical 3d} and the region corresponding to having a positive definite
covariance matrix. In the second picture, we superimpose the region corresponding to all 
choices of \( a \) and \( b \) for which \( X^h_{a,b}\) has a color representation for all \( h \) 
which are sufficiently close to zero. Interestingly, this figure suggests that if \( X_{a,b}^h \) is 
a color process for \( h \) close to zero, then \( X^h_{a,b} \) is also a color process for \( h \)  
sufficiently large. Moreover, the region corresponding to the set of \( a \) and \( b \) for which 
\( X_{a,b}^h \) has a color representation for \( h \) close to zero 
intersects both the regions corresponding to Cases (i) and (iii).

\begin{figure}[ht]
\centering
\begin{minipage}[t][][b]{0.45\linewidth}
\centering
\includegraphics[width=\textwidth]{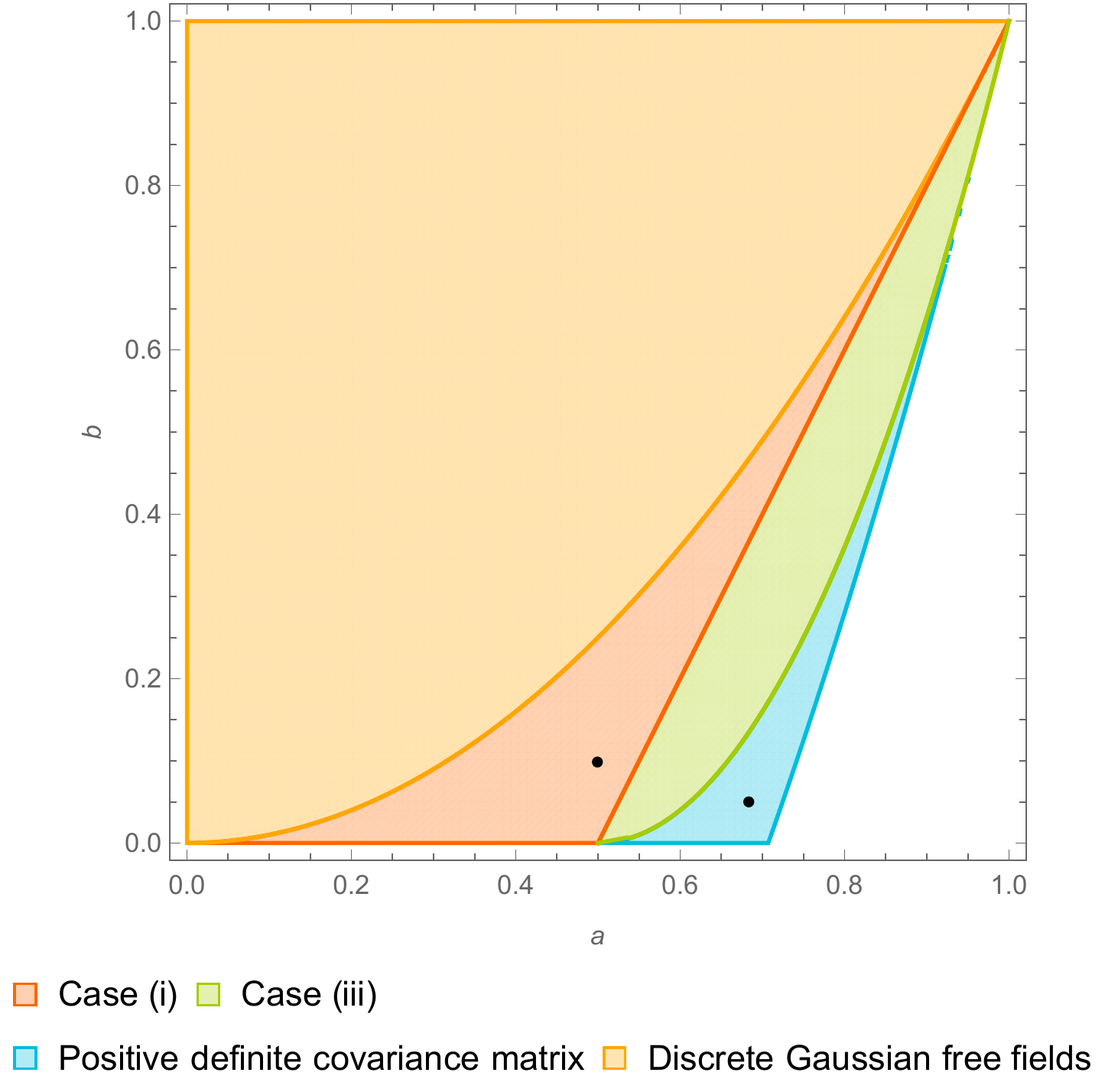}
\end{minipage}
\hspace{0.5cm}
\begin{minipage}[t][][b]{0.45\linewidth}
\centering
\includegraphics[width=\textwidth]{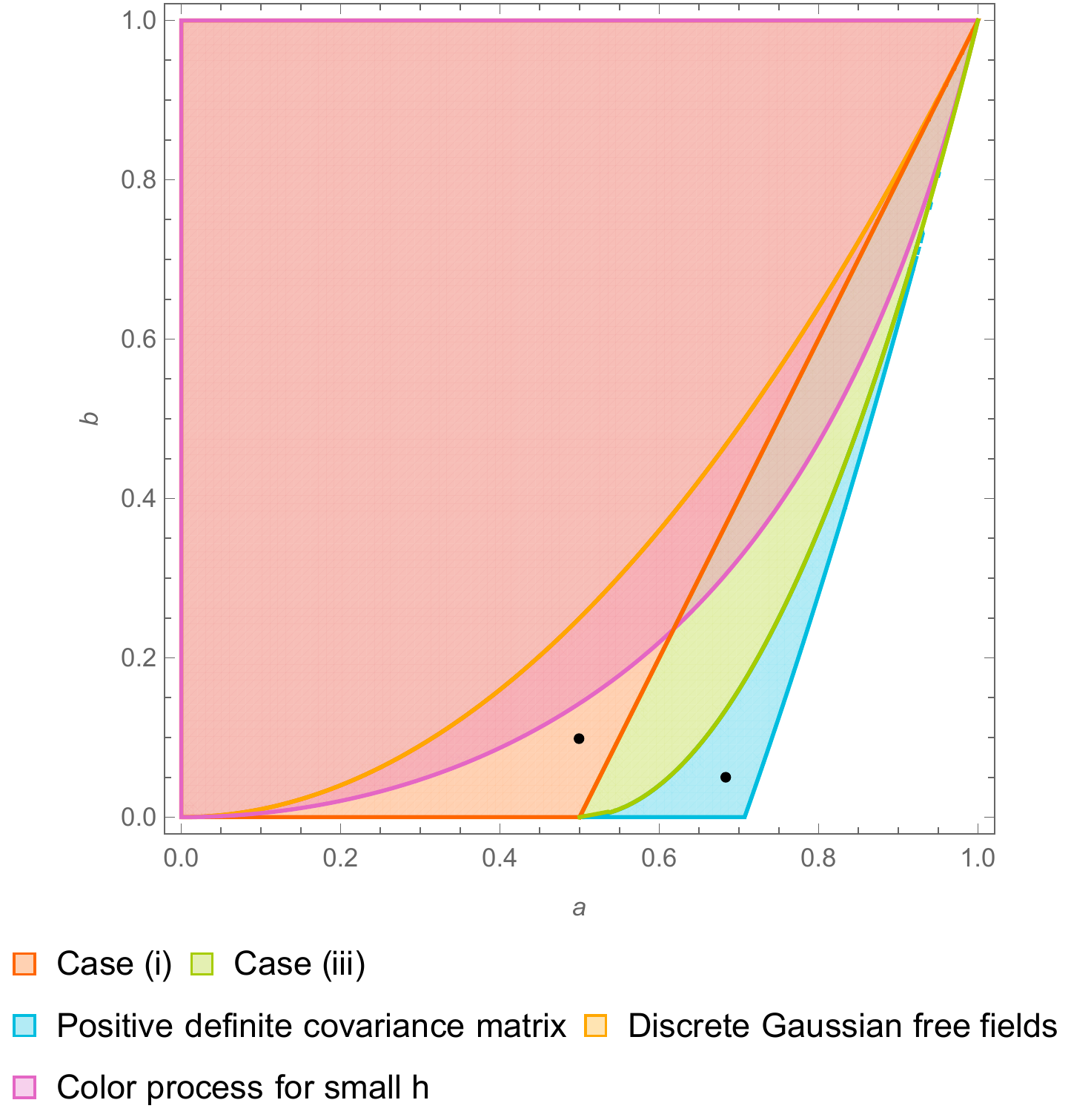}
\end{minipage}
\caption{The figure to the left shows, for Example~\ref{example: the ab example}, the 
different cases in Theorem~\ref{theorem: Gaussian critical 3d}.  \( A \) is positive definite
in the blue region and to its left, Case (iii) is the green region,
Case (i) is the red region and to its left and the set of DGFFs is the 
orange region. Case (ii) corresponds to the straight line \( b= 2a-1   \). The boundary of the orange 
region, which is the line \( b= a^2  \), corresponds to the family of standard Gaussian Markov chains.
The boundary between the green and blue regions is the right half of the parabola
\( b=(2a-1)^2 \). Finally the two black points correspond to the two examples given in the proof of 
(ii) and (iii) of Corollary~\ref{corollary:4examples}. The picture to the right is the same except
with the region where there is a color representation for \( h \) sufficiently close to zero 
being superimposed.}
\label{fig: ab example}
\end{figure}
\end{example}

We now proceed with the proof of Theorem~\ref{theorem: Gaussian critical 3d}. 

\begin{lemma}\label{lemma: pos cov implication} 
Let \( X \coloneqq (X_1,X_2)  \) be a fully supported standard Gaussian vector with covariance matrix 
\( A  = (a_{ij})\). Then \(\nu_{11} (h) \ll \nu_1(h) \) and if \( a_{12} > 0 \), then
\( \nu_1(h)^2 \ll \nu_{11} (h).\)
\end{lemma}

\begin{proof}
We have that
\(
\mathbf{1}^T A^{-1} = \left( (1+a_{12})^{-1}, (1+a_{12})^{-1}\right) > \mathbf{0}
\)
and hence Lemma~\ref{lemma: Gaussian cube tails II} implies that
\[
\nu_{11}(h) \asymp h^{-2}  \cdot \exp\left( -\frac{h^2}{2} \cdot \frac{2}{1+a_{12}} \right).
\]
Since \( 
p_1(h)  \asymp h^{-1}  \cdot \exp\left( - {h^2}/{2}   \right), 
\)
\(
p_1(h)^2 \asymp h^{-2}  \cdot \exp\left( -{h^2}\right)
\)
and \( a_{12} <1 \) by the fully supported assumption, the result easily follows.
\end{proof}

\begin{lemma}\label{lemma: pii comparison}
Let \( X \) be a fully supported \(3 \)-dimensional standard Gaussian vector with covariance matrix 
\( A  = (a_{ij})\).  If \( a_{ij} \in [0,1) \) for all \( i<j \), then
\begin{equation*} 
\begin{split}
&\nu_1(h) \max (\{ \nu_{11\cdot}(h), \nu_{1\cdot 1}(h), \nu_{\cdot 11}(h) \})  
\\&\qquad \ll \min(\{ \nu_{11\cdot}(h), \nu_{1\cdot 1}(h), \nu_{\cdot 11}(h) \}).
\end{split}
\end{equation*}
\end{lemma}

\begin{proof}
For \( i<j \), let \( A_{ij} \) be the covariance matrix of \( (X_i,X_j) \). Then
\[
\mathbf{1}^T A_{ij}^{-1} = \left( (1+a_{ij})^{-1}, (1+a_{ij})^{-1}\right) > \mathbf{0}
\]
and hence Lemma~\ref{lemma: Gaussian cube tails II} implies that
\begin{equation}\label{eq: star1}
\nu_{1^{\{i,j \}}}(h) \asymp h^{-2}  \cdot \exp\left( -\frac{h^2}{2} \cdot \frac{2}{1+a_{ij}} \right)
\end{equation}
and so
\begin{equation}\label{eq: star2}
\nu_1(h) \nu_{1^{\{i,j \}}}(h) \asymp h^{-3}  \cdot \exp\left( -\frac{h^2}{2} \cdot 
\left( 1+ \frac{2}{1+a_{ij}}\right) \right).
\end{equation}
In particular, this implies that the desired conclusion follows if we can show that
\[
 \max_{i<j} \frac{2}{1+a_{ij}} < 1 + \min_{i<j}  \frac{2}{1+a_{ij}}.
\]
However, since \( a_{ij} \in [0,1) \) for all \( i<j \) we have that
\[
 \max_{i<j} \frac{2}{1+a_{ij}} \leq 2 < 1 + \min_{i<j} \frac{2}{1+a_{ij}}.
\]
\end{proof}

\begin{lemma}\label{lemma: Savage implication}
Let \( X \) be a fully supported \(3 \)-dimensional standard Gaussian vector with covariance matrix 
\( A  = (a_{ij})\). 
If \(   \mathbf{1}^T A^{-1}> \mathbf{0} \) and at most one of the covariances \( a_{ij} \) is equal to zero, then  
\begin{equation}\label{eq: order inequalities if Savage holds}
\begin{split}
&\nu_1(h) \max (\{ \nu_{11\cdot}(h), \nu_{1\cdot 1}(h), \nu_{\cdot 11}(h) \}) \ll \nu_{111}(h) 
\\&\qquad \ll \min(\{ \nu_{11\cdot}(h), \nu_{1\cdot 1}(h), \nu_{\cdot 11}(h) \}).
\end{split}
\end{equation}
\end{lemma}
 
\begin{proof}
We first show that the second of the two inequalities holds.
First, since \( \mathbf{1}^T A^{-1} > \mathbf{0} \) by assumption, we have that
\begin{equation}\label{eq: star3}
\nu_{111}(h) \asymp h^{-3}   \cdot \exp\left( -\frac{h^2}{2} \cdot \mathbf{1}^T A^{-1} \mathbf{1}\right).
\end{equation}
Since
\[
\nu_{111}(h) \leq \min(\{ \nu_{11\cdot}(h), \nu_{1\cdot 1}(h), \nu_{\cdot 11}(h) \}),
\]
\eqref{eq: star1} implies
\[
\mathbf{1} A^{-1} \mathbf{1} \geq \frac{2}{1+a_{ij}}
\]
for all \( i<j \). However since \( h^{-3} \ll h^{-2} \), it follows that we then must have that
\[
\nu_{111}(h) \ll \min(\{ \nu_{11\cdot}(h), \nu_{1\cdot 1}(h), \nu_{\cdot 11}(h) \}).
\]
This shows that the second inequality in~\eqref{eq: order inequalities if Savage holds} holds.

Next, to show that the first of the two inequalities in~\eqref{eq: order inequalities if Savage holds} holds, 
we will show that
\begin{equation}\label{eq: goal equation Savage}
\mathbf{1}^T A^{-1} \mathbf{1} < 1 + \frac{2}{1+a_{ij}}
\end{equation}
for all \( i<j \), since if this holds, then~\eqref{eq: star2} and~\eqref{eq: star3}
immediately imply the desired conclusion.
To this end, using~\eqref{eq: Savage condition term},
one first verifies that\( \mathbf{1}^T A^{-1} > \mathbf{0} \) is equivalent to
\begin{equation}\label{eq: the linear Savage conditions}
1+2\min_{i<j} a_{ij}> \sum_{i<j} a_{ij}.
\end{equation}
Similarly,~\eqref{eq: goal equation Savage} can be shown to be equivalent to
\begin{equation}\label{eq: ellipse inequality}
(1+\max_{i<j}\{a_{ij}\}) \prod_{i<j} (1-a_{ij}) < 1 - \sum_{i<j} a_{ij}^2 + 2 \prod_{i<j} a_{ij}.
\end{equation}

If \( a_{ij} = 0 \) for exactly one of the covariances, then one easily verifies 
that~\eqref{eq: ellipse inequality} holds when~\eqref{eq: the linear Savage conditions} holds. 
Now instead assume that \( a_{ij}> 0 \) for all \( i>j \).
If we think of \( a_{12}>0 \) as being fixed, then~\eqref{eq: ellipse inequality} holds for all \( a_{13} \) and \( a_{23} \) in the interior of the ellipse \( E \) given by
\[
(1-a_{12}^2) (1-x)(1-y)= 1 - a_{12}^2-x^2-y^2 + 2a_{12}xy,\quad x,y \in \mathbb{R}.
\]
One verifies that the boundary of \( E \) passes through the origin and the points 
\( (0,1-a_{12}^2) \), \( (1-a_{12}^2,0) \), \( (a_{12},1) \) and \( (1,a_{12}) \).
Since we are assuming the Savage condition~\eqref{eq: Savage condition term again}, any possible
\(a_{13}) \) and \(a_{23}) \) under consideration necessarily lies in the region \( R \) given by
\[
1 + 2\min( \{a_{12},x,y \}) > a_{12}+ x+ y,\quad x,y>0.
\]
Hence we need only show that \( R\subseteq E \). (See Figure~\ref{figure: regioninsideellipse}.)
\begin{figure}[ht]
\centering
\includegraphics[width = 0.6\linewidth]{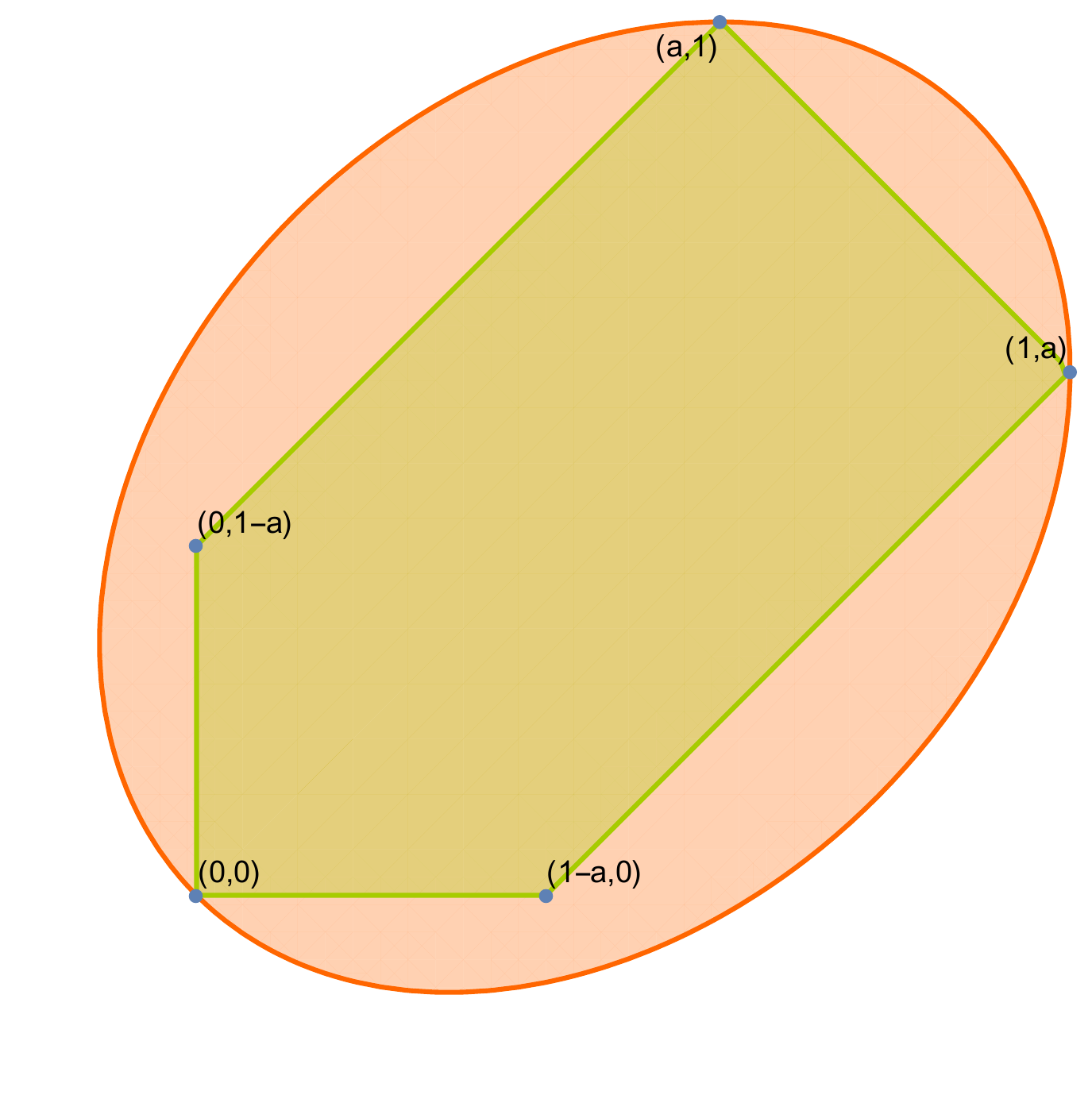}
\caption{The image above shows the situation in the proof of Lemma~\ref{lemma: Savage implication}, where we are 
interested in whether a region \( R \) is contained inside a given ellipse \( E \).}
\label{figure: regioninsideellipse}
\end{figure}
To see this containment, note that \( R \) is a polygon with vertices given by 
\( (0,0) \), \( (0,1-a_{12}) \), \( (1-a_{12},0) \), \( (1,a_{12}) \) and \( (a_{12},1) \). 
We already know that the first, fourth and fifth of these vertices lie on the boundary of $E$
while one easily checks that the other two lie inside $E$. Since 
\( E \) is convex, and \( R \) is a polygon, it follows that \( R\subseteq E \). 
\end{proof}

We are now ready to give the proof of Theorem~\ref{theorem: Gaussian critical 3d}. We remark that in the proof, Case 1 and Case 2 can alternatively be proven, using the lemmas in this section, by appealing to Lemma~\ref{lemma: solution sim lemma}.

\begin{proof}[Proof of Theorem~\ref{theorem: Gaussian critical 3d}]
For each \( h > 0 \), let \( (q_\sigma(h))_{ \sigma \in \mathcal{B}_3} \) be given by~\eqref{eq: gen sol when n is 3 and h is nonzero,h}. 
Using inclusion-exclusion, we see that for any \( h > 0 \) we have that
 \begin{equation*}\label{eq: q1,2,3 in limit thm}
 \begin{split}
q_{1,2,3}(h) &= \frac{ \nu_1(h) - (\nu_{\cdot 11}(h)+\nu_{1\cdot 1}(h)+\nu_{11\cdot}(h )) +2 \nu_{111}(h) }{\nu_0(h)\nu_1(h)(\nu_0(h)-\nu_1(h))},
\end{split}
\end{equation*}
 \begin{equation*}\label{eq: q12,3 in limit thm}
 \begin{split}
q_{12,3}(h) &= \frac{ (1-2\nu_1(h))\nu_{11\cdot}(h) +  \nu_1(h) ( \nu_{\cdot 11}(h)+\nu_{1\cdot 1}(h)+\nu_{11\cdot}(h )) - \nu_1(h)^2- \nu_{111}(h)   }{\nu_0(h)\nu_1(h)(\nu_0(h)-\nu_1(h))}  
\end{split}
\end{equation*}
and
 \begin{equation*}\label{eq: q123 in limit thm}
 \begin{split}
q_{123}(h) &=  \frac{  2\nu_1(h)^3   +\nu_{111}(h)   -   \nu_1(h)(\nu_{\cdot 11}(h) + \nu_{1\cdot 1}(h) + \nu_{11\cdot}(h)) }{\nu_0(h)\nu_1(h)(\nu_0(h)-\nu_1(h)) } .
\end{split}
\end{equation*}
This implies that there is a color representation for large \( h \) if and only if
for all large \( h \) we have that
\begin{equation}\label{eq: CR inequalities for large h III}
\nu_{\cdot 11}(h)+\nu_{1\cdot 1}(h)+\nu_{11\cdot}(h ) \leq \nu_1(h) +2 \nu_{111}(h),
\end{equation}
\begin{equation}\label{eq: CR inequalities for large h II}
 \begin{split}
\nu_{111}(h)   + \nu_1(h)^2 
&  \leq  \nu_1(h)(\nu_{\cdot 11}(h) + \nu_{1\cdot 1}(h)  + \nu_{11\cdot}(h))
    \\&\qquad   + (1-2\nu_1(h)) \min( \{ \nu_{11\cdot}(h) , \nu_{1\cdot 1}(h) , \nu_{\cdot 11}(h)  \})
\end{split}
\end{equation}
and 
\begin{equation}\label{eq: CR inequalities for large h I}
    \nu_1(h)(\nu_{\cdot 11}(h) + \nu_{1\cdot 1}(h) + \nu_{11\cdot}(h))    \leq \nu_{111}(h)   + 2\nu_1(h)^3 .
\end{equation}

We will check 
when~\eqref{eq: CR inequalities for large h III},~\eqref{eq: CR inequalities for large h II}~and~\eqref{eq: CR inequalities for large h I} hold for large \( h \) by comparing the decay rate of the various tails.

Before we do this, note that by~\eqref{eq: Savage condition term}, one has
that \( \mathbf{1}^T A^{-1}(1) \leq 0 \) exactly when \( 1+a_{23} \leq a_{12}+a_{13} \). If this holds, then clearly \( a_{23} = \min_{i<j} ( a_{ij}) \) and hence   \( \nu_{\cdot 11}(h) =  \min( \{ \nu_{11\cdot}(h) , \nu_{1\cdot 1}(h) , \nu_{\cdot 11}(h)  \}) \).

Without loss of generality, we assume that \( 0 \leq  a_{23} \leq a_{13} \leq a_{12} \) and that 
\( a_{12} > 0 \), since the case \( a_{12} = a_{13} = a_{23} = 0 \) is trivial. Note that this assumption 
implies by~\eqref{eq: Savage condition term} that
\[
\mathbf{1} A^{-1}(1) \leq \mathbf{1} A^{-1}(2) \leq \mathbf{1} A^{-1}(3)
\]
with the largest two terms being positive.

We now claim that~\eqref{eq: CR inequalities for large h III}   holds for all sufficiently large \( h \), 
without making any additional assumptions on \( A \). To see this, note that 
Lemma~\ref{lemma: pos cov implication} implies that
\begin{equation*}
\nu_{\cdot 11}(h)+\nu_{1\cdot 1}(h)+\nu_{11\cdot}(h ) \leq
   3\nu_{11\cdot}(h) \ll \nu_1(h) \leq  
 \nu_1(h) +2 \nu_{111}(h)
\end{equation*}
and hence~\eqref{eq: CR inequalities for large h III} holds for all  large \( h \).
 
We  now divide into four cases.

\paragraph{Case 1.} Assume that \( \mathbf{1}^T A^{-1}(1) > \mathbf{0} \) and \( a_{23}>0 \). We will show 
that both~\eqref{eq: CR inequalities for large h II}~and~\eqref{eq: CR inequalities for large h I} 
hold in this case without any further assumptions. To this end, note first that since \( a_{23} >0\), 
Lemma~\ref{lemma: pos cov implication} implies that \(  \nu_1(h)^2 \ll  \nu_{\cdot 11}(h) \). Moreover, since \( \mathbf{1}^T A^{-1}(1) = \min_{i \in [3]}\mathbf{1}^T A^{-1}(i)>0 \) implies that \( \mathbf{1}^T A^{-1} > \mathbf{0} \), Lemma~\ref{lemma: Savage implication} gives 
\begin{equation*} 
\begin{split}
&\nu_1(h) \max (\{ \nu_{11\cdot}(h), \nu_{1\cdot 1}(h), \nu_{\cdot 11}(h) \}) \ll \nu_{111}(h) 
\\&\qquad \ll \min(\{ \nu_{11\cdot}(h), \nu_{1\cdot 1}(h), \nu_{\cdot 11}(h) \}).
\end{split}
\end{equation*}
Combining these observations, we obtain
\begin{equation*}
\begin{split}
&\nu_{111}(h)  + \nu_1(h)^2 \ll \nu_{\cdot 11}(h) \sim  (1-2\nu_1(h))   \nu_{\cdot 11}(h)
\\&\qquad \leq
 \nu_1(h)(\nu_{\cdot 11}(h) + \nu_{1\cdot 1}(h)  + \nu_{11\cdot}(h))
    \\&\qquad\qquad  + (1-2\nu_1(h)) \min( \{ \nu_{11\cdot}(h) , \nu_{1\cdot 1}(h) , \nu_{\cdot 11}(h)  \})
\end{split}
\end{equation*}
and hence~\eqref{eq: CR inequalities for large h II} holds.
 Similarly, we obtain
\begin{equation*}
    \nu_1(h)(\nu_{\cdot 11}(h) + \nu_{1\cdot 1}(h) + \nu_{11\cdot}(h))   \ll \nu_{111}(h) \leq \nu_{111}(h)   + 2\nu_1(h)^3 .
\end{equation*}
establishing~\eqref{eq: CR inequalities for large h I}. This concludes the proof of (i).

\paragraph{Case 2.} Assume that \( \mathbf{1}^T A^{-1}(1) = 0 \) and \( a_{23} > 0 \). We will 
show that~\eqref{eq: CR inequalities for large h II}~and~\eqref{eq: CR inequalities for large h I} 
both hold in this case.  To this end, note first   that since \( \mathbf{1}^T A^{-1}(1) = 0 \), 
Lemma \ref{lemma: Gaussian cube tails II} implies that \( \nu_{111}(h) \sim \nu_{\cdot 11}(h)/2 \) and,
since \( a_{23} >0\), Lemma~\ref{lemma: pos cov implication} implies that \( \nu_{\cdot 11}(h) \gg \nu_1(h)^2 \). 
This implies in particular that
\begin{equation*}
 \begin{split}
& \nu_{111}(h)   + \nu_1(h)^2  \sim \nu_{\cdot 11}(h)/2 + \nu_1(h)^2 \sim \nu_{\cdot 11}(h)/2
\\&\qquad \sim \frac{1}{2} (1-2\nu_1(h)) \nu_{\cdot 11}(h) 
\\&\qquad = \frac{1}{2}(1-2\nu_1(h)) \min( \{ \nu_{11\cdot}(h) , \nu_{1\cdot 1}(h) , \nu_{\cdot 11}(h)  \})
\end{split}
\end{equation*}
and hence~\eqref{eq: CR inequalities for large h II} holds for all sufficiently large \( h \).
Next, using Lemma~\ref{lemma: pii comparison}, we obtain
\[
\nu_1(h)(\nu_{\cdot 11}(h) + \nu_{1\cdot 1}(h) + \nu_{11\cdot}(h)) 
    \ll \nu_{\cdot 11}(h)/2 \sim \nu_{111}(h)
\]
and hence~\eqref{eq: CR inequalities for large h I} holds for all sufficiently large \( h \) in this case. This finishes the proof of (ii).

\paragraph{Case 3.}
Assume that \( \mathbf{1}^T A^{-1}(1) < 0 \). By Lemma~\ref{lemma: Gaussian cube tails II}, we have that 
\( \nu_{111}(h) \sim \nu_{\cdot 11}(h) \). Using this, one easily checks 
that~\eqref{eq: CR inequalities for large h I} holds by the same argument as in Case 2, and hence it 
remains only to check when~\eqref{eq: CR inequalities for large h II} holds. To this end, note first that 
if we use the assumption that \( a_{23} \leq a_{13} \leq a_{12} \), 
then~\eqref{eq: CR inequalities for large h II} is equivalent to
\begin{equation}\label{eq: goal in step 3}
 \begin{split}
&   \nu_1(h)^2 + \nu_1(h)  (\nu_{\cdot 11}(h)  -   \nu_{1\cdot 1}(h)  - \nu_{11\cdot}(h)) \leq    \nu_{011}(h).
\end{split}
\end{equation}
Since \( a_{ij} < 1 \) for all \( i<j \),  Lemma~\ref{lemma: pos cov implication} implies that 
\[
\nu_1(h)^2 +\nu_1(h)  (\nu_{\cdot 11}(h)  -   \nu_{1\cdot 1}(h)  - \nu_{11\cdot}(h)) \sim \nu_1(h)^2.
\]

Therefore, by Lemma~\ref{lemma: Gaussian cube tails II}, we see that if 
\( \mathbf{1}^T A^{-1} \mathbf{1} <2 \) holds, then \(\nu_{011}(h) \gg  \nu_1(h)^2\)
yielding~\eqref{eq: goal in step 3}. On the other hand,
if \( \mathbf{1}^T A^{-1} \mathbf{1} \geq 2 \) holds, then
\( \nu_{011}(h) \ll  \nu_1(h)^2 \) in which case~\eqref{eq: goal in step 3} fails.

\paragraph{Case 4} Assume now that \( a_{23} = 0 \), i.e.\ that \( X_2 \) and \( X_3 \) are independent.  
Note that if \( a_{13}=a_{23} = 0 \), then there is a color representation by 
Proposition~\ref{theorem: h to infinity}, and hence we can assume that \( a_{13} > 0 \).
Now note that since \( X_2 \) and \( X_3 \) are independent by assumption, if \( X^h \) has a color 
representation  \( (q_\sigma(h)) \) for  some  \( h \), it must satisfy \( q_{1,23}(h) = q_{123}(h) = 0 \).  
Using the general formula for these expressions, we obtain that
\begin{equation*}
 \begin{split}
\nu_{111}(h)   + \nu_1(h)^2 
  &=  \nu_1(h)(\nu_{\cdot 11}(h) + \nu_{1\cdot 1}(h)  + \nu_{11\cdot}(h))
    \\&\qquad  + (1-2\nu_1(h))   \nu_{\cdot 11}(h)  
\end{split}
\end{equation*}
and 
\begin{equation*}
\nu_1(h)(\nu_{\cdot 11}(h) + \nu_{1\cdot 1}(h) + \nu_{11\cdot}(h))    = \nu_{111}(h)   + 2\nu_1(h)^3 .
\end{equation*}
Using that \( \nu_{\cdot 11}(h) = \nu_1(h)^2 \) by assumption, we see that these equations are both 
equivalent to that
\begin{equation}\label{eq: equality equation equivalent to both}
 \nu_{111}(h)    +   \nu_1(h)^3
 =  \nu_1(h)( \nu_{1\cdot 1}(h)  + \nu_{11\cdot}(h)).
\end{equation}

We will show that~\eqref{eq: equality equation equivalent to both} does not hold for any large \( h \).
To this end, note first that if \( \mathbf{1}^T A^{-1} > \mathbf{0 } \) and \( a_{12},a_{13}> 0 \), 
then by Lemma~\ref{lemma: Savage implication} we have that
\[
 \nu_{111}(h)    +   \nu_1(h)^3
 \sim  \nu_{111}(h)  \gg \nu_1(h)( \nu_{1\cdot 1}(h)  + \nu_{11\cdot}(h))
\]
and hence~\eqref{eq: equality equation equivalent to both} cannot hold, implying that there can be no color representation for any large \( h \) in this case. 

Next, if \( \mathbf{1}^T A^{-1} (1)=0 \), then using Lemma~\ref{lemma: Gaussian cube tails II} we get that
\begin{align*}
&\nu_{111}(h)    +   \nu_1(h)^3 \sim \nu_{\cdot 11}(h)/2    +   \nu_1(h)^3 =  \nu_1(h)^2/2  +  \nu_1(h)^3  
\sim \nu_1(h)^2/2.
\end{align*}
Using Lemma~\ref{lemma: pos cov implication}  and the assumption that \( a_{12},a_{13} < 1 \), it follows that
\begin{align*}
&\nu_{1}(h)^2   \gg \nu_1(h)( \nu_{1\cdot 1}(h)  + \nu_{11\cdot}(h)) 
\end{align*}
and hence~\eqref{eq: equality equation equivalent to both} cannot hold, implying that there can be no color representation for any large \( h \) in this case.

Finally, if \( \mathbf{1}^TA^{-1} (1) < 0 \) then we can use Case 3. Observing that if \( a_{23} = 0 \), then \( \det A > 0 \) implies that \( a_{12}^2 + a_{13}^2 <1 \), we have that
\begin{align*}
&\mathbf{1}^T A^{-1} \mathbf{1} < 2
\Leftrightarrow 1 + \frac{2(1-a_{12})(1-a_{13})}{1-a_{12}^2-a_{13}^2}  < 2
\\&\qquad\Leftrightarrow (1-a_{12}-a_{13})^2 < 0
\end{align*}
implying in particular that there can be no color representation.
\end{proof}

\section{Large threshold results for stable vectors with emphasis on the \texorpdfstring{$n=3$}{n=3} 
case} \label{section:nonzerostable}

\subsection{Two-dimensional stable vectors and nonnegative correlations}

In this subsection, we give a proof of Proposition~\ref{2dpos.corr}. 

\begin{proof}[Proof of Proposition~\ref{2dpos.corr}]
We may stick to $h\ge 0$ throughout. Since for $n=2$, being a color process is trivially equivalent to having
nonnegative correlations, we can immediately replace (ii) by $X^0$ has nonnegative correlations
and (iii) by $X^h$ has nonnegative correlations for all $h$.
It is elementary to check that \( X^h_1 \) and \( X^h_2 \) have nonnegative correlation if and only if
\begin{equation}\label{eq: example inequality}
P((1-a^\alpha)^{1/\alpha} S_2 \geq a  |S_1| + h) \geq P(S_1 \geq h)^2.
\end{equation}
When \( h = 0 \) and \( a = 2^{-1/\alpha} \), we have that
$$
P((1-a^\alpha)^{1/\alpha} S_2 \geq a  |S_1| + h)   = P(S_2 \geq   |S_1| )  = 1/4
$$
and
\[
 P(S_1 \geq h)^2 = 1/4.
 \]
Hence we get equality in~\eqref{eq: example inequality} in this case. Now note that the left hand side 
of~\eqref{eq: example inequality} is strictly decreasing in \( a \). This  implies that when 
\( h = 0 \), we get nonnegative correlations if and only if \( a \leq 2^{-1/\alpha} \), establishing the 
equivalence of (i) and (ii).

We now show that (ii) implies (iii).
To see this, note first that since the left hand side of~\eqref{eq: example inequality} is 
strictly decreasing in  \( a \), it suffices to show that~\eqref{eq: example inequality} holds for all 
\( h \geq 0 \) when \( a = 2^{-1/\alpha} \). To this end, note first that in this case, we have that
\begin{align*}
P((1-a^\alpha)^{1/\alpha} S_2 \geq a  |S_1| + h) &= P( S_2 \geq  |S_1| + h2^{1/\alpha}) .
\end{align*}
Now observe that
\begin{align*}
&2P( S_2 \geq  S_1+ h2^{1/\alpha},\, S_2 \geq h2^{1/\alpha})
\\&\qquad= 2P( S_2 \geq  S_1+ h2^{1/\alpha},\, S_2 \geq h2^{1/\alpha},\, S_1 < 0)
\\&\qquad\qquad + 2P( S_2 \geq  S_1+ h2^{1/\alpha},\, S_2 \geq h2^{1/\alpha},\, S_1 \geq 0 )
\\&\qquad= 2P(  S_2 \geq h2^{1/\alpha},\, S_1 < 0)
\\&\qquad\qquad + 2P( S_2 \geq  |S_1|+ h2^{1/\alpha}, \, S_1 \geq 0 )
\\&\qquad= 2P(  S_2 \geq h2^{1/\alpha})P(S_1 < 0)
\\&\qquad\qquad +2 P( S_2 \geq  |S_1|+ h2^{1/\alpha})P(S_1 \geq 0 )
\\&\qquad=  P(  S_2 \geq h2^{1/\alpha}) + P( S_2 \geq  |S_1|+ h2^{1/\alpha}) 
\end{align*}
and that
\begin{align*}
P( S_1\geq  h )
&=
P( S_2 + S_1\geq  h2^{1/\alpha})
\\&=
P( S_2 + S_1\geq  h2^{1/\alpha},\, S_1 \geq h2^{1/\alpha})
\\&\qquad +
P( S_2 + S_1\geq  h2^{1/\alpha},\, S_2 \geq h2^{1/\alpha})
\\&\qquad -
P( S_2 + S_1\geq  h2^{1/\alpha},\, S_1 \geq h2^{1/\alpha},\, S_2 \geq h2^{1/\alpha})
\\&=
2P( S_2 + S_1\geq  h2^{1/\alpha},\, S_2 \geq h2^{1/\alpha}) -
P(S_1 \geq h2^{1/\alpha})^2.
\end{align*}
Putting these observations together, we obtain
\begin{align*}
& P( S_2 \geq  |S_1| + h2^{1/\alpha})
\\&\qquad= 
2P( S_2 \geq  S_1+ h2^{1/\alpha},\, S_2 \geq h2^{1/\alpha}) - P(  S_2 \geq h2^{1/\alpha}) 
\\&\qquad= 
P( S_1\geq  h ) + P(S_1 \geq h2^{1/\alpha})^2- P(  S_2 \geq h2^{1/\alpha}).
\end{align*}
In particular, we get nonnegative correlations if and only if
\[
P( S_1\geq  h ) + P(S_1 \geq h2^{1/\alpha})^2- P(  S_2 \geq h2^{1/\alpha}) \geq P(S_1 \geq h)^2.
\]
Rearranging, we see that this is equivalent to  
\[
P( S_1\geq  h ) -P(S_1 \geq h)^2  \geq  P(  S_2 \geq h2^{1/\alpha}) - P(S_1 \geq h2^{1/\alpha})^2
\]
which will hold for all \( h \geq 0 \) since \( P( S_1\geq  h ) - P(S_1 \geq h)^2 \) is decreasing 
in \( h \) for all \( h \geq 0 \). This establishes (iii).
\end{proof}

\subsection{\texorpdfstring{\( h\)}{h} large and a phase transition in the stability exponent}

In this subsection we will look at what happens  when \( X \) is a symmetric  multivariate stable 
random variable with index  \( \alpha < 2 \) and marginals \( S_\alpha(1,0,0) \), and the 
threshold \( h> 0 \) is large. The fact that stable distributions have fat tails for 
\( \alpha < 2 \) will  result in  behavior that is radically different from  the Gaussian case.
We will obtain various results, perhaps the most interesting being a phase transition in $\alpha$
at $\alpha=1/2$; this is Theorem~\ref{theorem:ptalpha12}.

\begin{proof}[Proof sketch of Theorem~\ref{theorem:stablegoodsupport}]
We show that the assumptions of Lemma~\ref{lemma: solution sim lemma} hold.
First Theorem~1.1 in~\cite{fs2019a}, with \( k = 1 \), implies that~\eqref{eq: the main assumption} holds. Next, a computation using the same theorem shows that the last condition in Lemma~\ref{lemma: solution sim lemma} holds if~\eqref{eq: stablegood} holds.
\end{proof}

We will now apply Theorem~1.1 in~\cite{fs2019a} to a stable version of a Markov chain.

\begin{corollary}\label{corollary: stable MC}
Let \( \alpha \in (0,2) \) and let \( S_1 \), \( S_2 \) and \( S_3 \) be i.i.d.\ with \( S_1 \sim S_\alpha(1,0,0) \). Furthermore, let \( a \in (0,1) \) and define \( X_1 \coloneqq S_1 \) and \( X_2 \) and \( X_3 \) by
\[
 X_{i } \coloneqq aX_{i-1} + (1-a^\alpha)^{1/\alpha} S_i , \quad i=2,3.
\]
Then  \( X^h \) is a color process for all sufficiently large \( h \).
\end{corollary}

\begin{remark}
The random vector \( X \) defined by this corollary is a stable Markov chain. We have already
seen a Gaussian analogoue of this result.
\end{remark}

\begin{proof}[Proof of Corollary~\ref{corollary: stable MC}]
Clearly \( (X_1,X_2,X_3) \) is a three-dimensional symmetric \( \alpha \)-stable random vector 
whose marginals are \( S_\alpha(1,0,0) \). If we let \( A \) be given by
\[
\begin{pmatrix}
1 & 0 & 0  \\
a & (1-a^\alpha)^{1/\alpha} & 0 \\
a^2  & a(1-a^\alpha)^{1/\alpha}&  (1 - a^{\alpha})^{1/\alpha} \\
\end{pmatrix}
\]
then
\[
\begin{pmatrix} X_1 \\ X_2 \\ X_3\end{pmatrix}
=
A 
\cdot
\begin{pmatrix} S_1 \\ S_2 \\ S_3\end{pmatrix}.
\]
It follows that for each  \( \mathbf{x} \in \support(\Lambda) \),  exactly one of \( \pm (2\Lambda(\mathbf{x}))^{1/\alpha} \mathbf{x} \) is a column in \( A \). Moreover, each column of \( A \) corresponds to a pair of points in the 
support of \( \Lambda \) in this way. To simplify notation, for \( \mathbf{x} \in \support(\Lambda ) \)  
we write \(  \hat{\mathbf{x}} \coloneqq  (2\Lambda(\mathbf{x}))^{1/\alpha} \mathbf{x} \).
Using Theorem 1.1 in~\cite{fs2019a} with $n=3$ and $k=1$, one easily 
verifies that this implies that
\begin{align*} 
\lim_{h \to \infty} \frac{\nu_{111}(h)}{\nu_1(h)} 
&=\sum_{\mathbf{x}_1 \in \support ( \Lambda)}
\int_0^\infty
I \left(     s_1 \hat {\mathbf{x}}_1 > \mathbf{1}  \right) 
\cdot    \alpha  s_1^{-(1+\alpha)}\, ds_1
\\&
 = 
\int_{a^{-2}}^\infty  \alpha  s_1^{-(1+\alpha)}\, ds_1 = a^{2\alpha}
\end{align*}
%
and similarly that 
\begin{numcases}{}
\lim_{h \to \infty} \frac{\nu_{110}(h)}{\nu_1(h)} = a^{\alpha}(1-  a^{\alpha})\nonumber \\
\lim_{h \to \infty} \frac{\nu_{100}(h)}{\nu_1(h)} = 1-a^{\alpha}\nonumber \\
\lim_{h \to \infty} \frac{\nu_{011}(h)}{\nu_1(h)} = a^\alpha (1-a^\alpha)\nonumber \\
\lim_{h \to \infty} \frac{\nu_{010}(h)}{\nu_1(h)} = (1-a^\alpha)^2 \label{eq: p010/p1 limit}\\
\lim_{h \to \infty} \frac{\nu_{001}(h)}{\nu_1(h)} = 1-a^\alpha \nonumber\\
\lim_{h \to \infty} \frac{\nu_{101}(h)}{\nu_1(h)} = 0. \nonumber
\end{numcases}

Combining this with~\eqref{eq: gen sol when n is 3 and h is nonzero,h} we obtain
\[
\begin{cases}
\lim_{h \to \infty}q_{1,2,3}(h) &= (1-a^\alpha )^2\cr
\lim_{h \to \infty}q_{12,3}(h) &= a^\alpha (1-a^\alpha) \cr
\lim_{h \to \infty}q_{13,2}(h) &= 0\cr
\lim_{h \to \infty}q_{1,23}(h) &= a^\alpha (1-a^\alpha) \cr
\lim_{h \to \infty}q_{123}(h) &= a^{2\alpha}.
\end{cases}
\]
 From this it follows that \( X^h \) has a  color representation for all sufficiently large \( h \) if   \( q_{13,2}(h) \) is non-negative for large \( h \).
By~\eqref{eq: gen sol when n is 3 and h is nonzero,h}, \( q_{13,2}(h) \) is given by
 \[
 q_{13,2}(h) = \frac{\nu_0(h)\nu_{101}(h) - \nu_1(h)\nu_{010}(h)}{\nu_1(h) \nu_0(h)  (\nu_0(h)-\nu_1(h))}.
 \]
 Here the denominator is strictly positive for all \( h > 0 \), and   we know 
from~\eqref{eq: p010/p1 limit} that \( \nu_{010}(h) = (1-a^\alpha)^2\nu_1(h) + o(\nu_1(h)) \). Hence it is sufficient to show that
 \[
 \lim_{h \to \infty} \frac{\nu_{101}(h)}{\nu_1(h)^2} >  (1-a^\alpha)^2.
 \]
To see this, we again apply Theorem~1.1 in~\cite{fs2019a} to obtain
\begin{align*}
 &\lim_{h \to \infty} \frac{\nu_{101}(h)}{\nu_1(h)^2}  
\\&\qquad =
 \frac{1  }{2}
 \sum_{\mathbf{x}_1,   \mathbf{x}_2 \in \support ( \Lambda)}
\int_0^\infty
\int_{0}^\infty
I
 \bigl( 
    s_1\hat{\mathbf{x}}_1(1) + s_2\hat{\mathbf{x}}_2 (1)>1,
   \\[-2ex]&\qquad\qquad\qquad\qquad\qquad\qquad\qquad\qquad
   s_1\hat{\mathbf{x}}_1(2) + s_2\hat{\mathbf{x}}_2 (2)\leq1,
   \\[0ex]&\qquad\qquad\qquad\qquad\qquad\qquad\qquad\qquad
   s_1\hat{\mathbf{x}}_1(3) + s_2\hat{\mathbf{x}}_2 (3)>1
     \bigr) 
\,   \alpha^2  s_1^{-(1+\alpha)}s_2^{-(1+\alpha)} \, ds_2 \, ds_1
\\&\qquad =
\int_0^\infty
\int_{0}^\infty
I
 \left(a^{-1} >  s_1 > 1  ,\,    s_2 > \frac{1 - a^2s_1}{(1-a^\alpha)^{1/\alpha}}
     \right) 
\,   \alpha^2  s_1^{-(1+\alpha)}s_2^{-(1+\alpha)} \, ds_2 \, ds_1
\\&\qquad =
(1-a^\alpha)\int_1^{a^{-1}}
\left(1 - a^2s_1 \right)^{-\alpha}
  \alpha  s_1^{-(1+\alpha)} \, ds_1
  \\&\qquad >
  (1-a^\alpha)    \int_{1}^{a^{-1} }  \alpha  s_1^{-(1+\alpha)} \, ds_1
  =
  (1-a^\alpha)^2
\end{align*}
which is the desired conclusion.
\end{proof}

We can now prove Theorem~\ref{theorem:ptalpha12} which is a stable version of the example in 
the proof of (i) of Corollary~\ref{corollary:4examples}.

\begin{proof}[Proof of Theorem~\ref{theorem:ptalpha12}]
We start a little more generally.
Let \( \alpha \in (0,2) \) and let \( S_0 \), \( S_1 \), \ldots,  \( S_n \) be i.i.d.\ with 
\( S_0 \sim S_\alpha(1,0,0) \).  Furthermore, let \( a \in (0,1) \) and for \( i = 1, 2,\ldots, n \), define
\[
X_i = aS_{0} + (1-a^\alpha)^{1/\alpha} S_i.
\] 
Note first that for any  \(n \geq 1 \),  \( (X_1,X_2 , \ldots, X_n) \) is clearly an
\(n\)-dimensional symmetric \( \alpha \)-stable random vector whose marginals 
have distribution \( S_\alpha(1,0,0) \). Moreover, for any  \( n \geq 2 \), if we let \( A \) be 
the \( n \times (n+1)\) matrix defined by 
\[
A(i,j) = 
\begin{cases}
a &\text{if } j = 1 \cr
 (1-a^\alpha)^{1/\alpha}  &\text{if } j = i+1\cr
 0 &\text{otherwise}
\end{cases}
\]
then
\[
\begin{pmatrix} X_1   , \ldots,  X_n \end{pmatrix}^T
=
A 
\cdot
\begin{pmatrix} S_0 , S_1, \ldots, S_n\end{pmatrix}^T.
\]
It follows that for each  \( \mathbf{x} \in \support(\Lambda) \),  exactly one of \( \pm (2\Lambda(\mathbf{x}))^{1/\alpha} \mathbf{x} \) is a column in \( A \). Moreover, each column of \( A \) corresponds to a pair of
points in the support of \( \Lambda \) in this way. To simplify notation, for \( \mathbf{x} \in \support(\Lambda ) \)  we write \(  \hat{\mathbf{x}} \coloneqq  (2\Lambda(\mathbf{x}))^{1/\alpha} \mathbf{x} \).
Using Theorem~1.1 in~\cite{fs2019a},  one easily verifies that it follows that
\begin{equation}\label{eq: p111/p1 limit for large h}
\begin{split}
\lim_{h \to \infty} \frac{\nu_{1^n}(h)}{\nu_1(h)} 
&=
\sum_{\mathbf{x}_1 \in \support ( \Lambda)}
\int_0^\infty
I \left(     s_1 \hat {\mathbf{x}}_1 > \mathbf{1}  \right) 
\cdot    \alpha  s_1^{-(1+\alpha)}\, ds_1
\\&= 
\int_0^\infty
I \left(     as_1  >1   \right) 
\cdot    \alpha  s_1^{-(1+\alpha)}\, ds_1
= a^\alpha.
\end{split}
\end{equation}
Returning to the case $n=3$, let, for $h>0$, 
\( (q_{123}(h), q_{12,3}(h), q_{13,2}(h), q_{1,23}(h), q_{1,2,3}(h) ) \) 
be given by~\eqref{eq: gen sol when n is 3 and h is nonzero,h}. Then, symmetry and inclusion-exclusion, we have that
\[
q_{1,2,3}(h) = \frac{\nu_{100}(h)-\nu_{011}(h)}{\nu_0(h)\nu_1(h)(\nu_0(h)-\nu_1(h))} = \frac{\nu_{1 }(h)-3\nu_{11 \cdot}(h) +2\nu_{11 1}(h) }{\nu_0(h)\nu_1(h)(\nu_0(h)-\nu_1(h))}  
\]
and hence \( \lim_{h \to \infty} q_{1,2,3}(h) = 1-a^\alpha. \) Similarly, one sees that
 \[ \lim_{h \to \infty} q_{12,3}(h)  =  \lim_{h \to \infty} q_{13,2}(h)  =  \lim_{h \to \infty} q_{1,23}(h)  = 0 \] and hence \( \lim_{h \to \infty} q_{123}(h) = a^\alpha \). Since the solution is permutation invariant, it  follows that we have a color representation for all sufficiently large \( h \) if and only if \( q_{12,3}(h) \geq 0 \) for all sufficiently large \( h \). To see when this happens, note first that by symmetry, \( \nu_{101} + \nu_{010} = \nu_{011} + \nu_{010} = \nu_{01 \cdot} \) and hence, 
using~\eqref{eq: gen sol when n is 3 and h is nonzero,h}, it follows that
\begin{align*}
q_{12,3}(h) &= \frac{\nu_0(h)\nu_{110}(h)-\nu_1(h)\nu_{001}(h)}{\nu_0(h)\nu_1(h)(\nu_0(h)-\nu_1(h))} 
=
\frac{ \nu_{110}(h)-\nu_1(h)\nu_{01 \cdot }(h)}{\nu_0(h)\nu_1(h)(\nu_0(h)-\nu_1(h))} .
\end{align*}
The denominator is strictly positive for all large \( h \) and by~\eqref{eq: p111/p1 limit for large h} we have that
\begin{align*}
\nu_{01\cdot}(h) &= \nu_1(h) - \nu_{11\cdot }(h)  = \nu_1(h)(1-a^\alpha)+ o(\nu_1(h)) .
\end{align*}
The question is now how \( \lim_{h \to \infty} \nu_{110}(h)/\nu_1(h)^2 \) compares with \( 1 - a^\alpha \).
Using Proposition 4.9 in~\cite{fs2019a}, it follows that
\begin{equation*}
\lim_{h \to \infty} \frac{\nu_{101}(h)}{\nu_1(h)^2}
=  \begin{cases}
(1-a^\alpha)^2+ a^\alpha (1-a^\alpha) \, \frac{\alpha \Gamma(2\alpha) \Gamma(1-\alpha)}{\Gamma(1+\alpha)}  &\textnormal{if } \alpha \in (0,1) \cr
\infty &\textnormal{else.}
  \end{cases}
\end{equation*}
From this it immediately follows that \(X^h \) has a color representation for all sufficiently large \( h \) if \( \alpha \in [1,2) \). When \( \alpha \in (0,1) \), then 
 \( X^h \) has a color representation for all sufficiently large \( h \) if  
\[
 \frac{\alpha \Gamma(2\alpha) \Gamma(1-\alpha)}{\Gamma(1+\alpha)}  > 1
\]
and has no color representation for any large \( h \) if
\[
 \frac{\alpha \Gamma(2\alpha) \Gamma(1-\alpha)}{\Gamma(1+\alpha)}   <1.
\]
This expression is strictly positive for all \( \alpha \in (0,1) \) and equal to 1 if \( \alpha = 
1/2 \). Furthermore, it is equal to 
\[
\frac{\alpha \Gamma(2\alpha) \Gamma(1-\alpha)}{\Gamma(1+\alpha)}
=
\frac{ \Gamma(2\alpha) \Gamma(1-\alpha)}{\Gamma(\alpha)}
=
2^{2\alpha-1} \Gamma \left( \alpha + \frac{1}{2} \right) \Gamma \left( 1 - \alpha \right) \cdot \frac{1}{\sqrt{\pi}}
\]
where the last equality follows by using the Legendre Duplication Formula 
(see \cite{AS}, 6.1.18, p.\ 256).
We claim that this expression is strictly increasing in \( \alpha \). If we can show this, the conclusion of 
the theorem will follow since we get equality at \( \alpha = 1/2 .\)
To see this, recall first that \( \Gamma'(\alpha) = \Gamma(\alpha) \psi (\alpha) \), where \( \psi\) is the so-called digamma function. It follows that the derivative of the expression above is equal to
\[
2^{2\alpha-1} \Gamma \left( \alpha + \frac{1}{2} \right) \Gamma \left( 1 - \alpha \right) \cdot \frac{1}{\sqrt{\pi}}\cdot \left( 2\log 2 + \psi \left( \alpha + \frac{1}{2} \right)- \psi(1-\alpha)  \right).
\]
Since the first term is equal to our original integral, it is clearly strictly larger than zero. Moreover, 
an integral representation of \( \psi \) given in \cite{AS} (see 6.3.21, p.\ 259)
implies that \( \psi(x) \) is strictly increasing in \( x \) for \( x > 0 \).
It follows that the second term is strictly larger than
\[
2\log 2 + \psi \left({    {1}/{2}} \right) - \psi \left(1 \right).
\]
Using the values of the digamma function at $1/2$ and 1 
(see \cite{AS}, 6.3.2 and 6.3.3, p.\ 258), this last expression is 0. This finishes the proof.
\end{proof}

We next give the proof of Theorem~\ref{theorem: alternative symmetric example}.

\begin{proof}[Proof of Theorem~\ref{theorem: alternative symmetric example}]
Clearly \( (X_1,X_2,X_3) \) is a three-dimensional symmetric \( \alpha \)-stable random vector 
whose marginals are \( S_\alpha(1,0,0) \).

If we define \( c = c(\alpha,a,b) \coloneqq  (1-2a^\alpha-2b^\alpha)^{1/\alpha} \) and 
let \( A \) be given by
\[
\begin{pmatrix}
a & b & 0 & b & a & 0& c\\
 0 & a & b & 0 & b& a & c\\
 b & 0 & a & a & 0 & b & c\\
\end{pmatrix}
\]
 then
\[
\begin{pmatrix} X_1 \\ X_2 \\ X_3\end{pmatrix}
=
A 
\cdot
(S_1,\,S_2,\,S_3,\,S_4,\,S_5,\,S_6,\,S_7)^T
\]
It follows that for each  \( \mathbf{x} \in \support(\Lambda) \),  exactly one of \( \pm (2\Lambda(\mathbf{x}))^{1/\alpha} \mathbf{x} \) is a column in \( A \). Moreover, each column of \( A \) corresponds to a pair of points in the support of \( \Lambda \) in this way. To simplify notation, for \( \mathbf{x} \in \support(\Lambda ) \)  we write \(  \hat{\mathbf{x}} \coloneqq  (2\Lambda(\mathbf{x}))^{1/\alpha} \mathbf{x} \).
%
Using Theorem~1.1 in~\cite{fs2019a},  we get that
\begin{equation} 
\begin{split}
\lim_{h \to \infty} \frac{\nu_{111}(h)}{\nu_1(h)} 
&=
\sum_{\mathbf{x}_1 \in \support ( \Lambda)}
\int_0^\infty
I \left(     s_1 \hat {\mathbf{x}}_1 > \mathbf{1}  \right) 
\cdot    \alpha  s_1^{-(1+\alpha)}\, ds_1
\\&=   \int_0^\infty I(c s_1  >  {1})\cdot   \alpha  s_1^{-(1+\alpha)}\, ds_1 
= c^\alpha = 1-2a^\alpha - 2b^\alpha.
\end{split}
\end{equation}
Similarly, we obtain
\begin{equation} 
\begin{split}
\lim_{h \to \infty} \frac{\nu_{110}(h)}{\nu_1(h)} 
&=   2  \int_0^\infty I(s_1 \cdot  \min(\{ a,b \})  > 1)\cdot   \alpha  s_1^{-(1+\alpha)}\, ds_1  
\\&= 2 \min(\{ a,b \})^\alpha.
\end{split}
\end{equation}
Using~\eqref{eq: gen sol when n is 3 and h is nonzero,h}, it follows that 
\begin{equation}\label{eq: cr limits}
\begin{cases}
\lim_{h \to \infty} q_{123} (h) = 1-2a^\alpha - 2b^\alpha \cr
\lim_{h \to \infty} q_{12,3} (h) = 2\min(\{ a,b \})^\alpha \cr
\lim_{h \to \infty} q_{13,2} (h) =2\min(\{ a,b \})^\alpha\cr
\lim_{h \to \infty} q_{1,23} (h) = 2\min(\{ a,b \})^\alpha 
\end{cases}
\end{equation}
and as \( q_{1,2,3}(h) = 1-q_{123}(h)-q_{12,3}(h) - q_{13,2}(h) - q_{1,23}(h) \) for \( h \in \mathbb{R} \), we also obtain
\[ 
\lim_{h \to \infty} q_{1,2,3}(h)  = 1 - (1-2a^\alpha-2b^\alpha) - 6\min(\{a,b\})^\alpha = 2 \left( \max(\{ a,b \})^\alpha  - 2\min(\{ a,b \})^\alpha \right).
\]
Since \( a,b \in (0,1) \) and \(  2a^\alpha + 2b^\alpha < 1 \) (as $\alpha > c_1$),
it follows that all of the limits in~\eqref{eq: cr limits} lie in \( (0,1) \) for 
any \( \alpha \in (0,1) \).  

Let $g(\alpha)=\max(\{ a,b \})^\alpha  - 2\min(\{ a,b \})^\alpha$ for $\alpha\in (0,\infty)$.
If $a=b$, then $c_2=\infty$ and $g(\alpha)=\max(\{ a,b \})^\alpha  - 2\min(\{ a,b \})^\alpha$ is 
negative for all $\alpha$ and the claim holds. If $a\neq b$, then it is easy to check that $c_2$ 
is the unique zero of $g(\alpha)$ on $(0,\infty)$ and that $g$ is negative (positive) to the left (right)
of $c_2$. This immediately leads to all of the claims.
\end{proof}

\section*{Acknowledgements}

We thank Enkelejd Hashorva for providing some references. We also thank the referees for useful comments both on an earlier version and on the present version of this paper. We are in particular grateful to one anonymous referee for providing a simpler proof of Step 3 (v) in the proof of Lemma~\ref{lemma: conditions hold for DGFF}.
The first author acknowledges support from the European Research Council, 
grant no.\ 682537.
The second author acknowledges the support of the Swedish Research 
Council, grant no.\ 2016-03835
and the Knut and Alice Wallenberg Foundation, grant no.\ 2012.0067.


\begin{thebibliography}{99}


\bibitem{AS} Abramowitz, M. and Irene A.\ Stegun, I. A.:
{Handbook of mathematical functions with formulas, graphs, and mathematical tables}, 
National Bureau of Standards Applied Mathematics Series, 55 (1970). 
    
    
    \bibitem{bp1994}
    Benjamini, I. and Peres, Y.: Markov chains indexed by trees, \emph{Ann. Probab.}, \textbf{22}  (1994), no.\ 1,  219 -- 243.  
    
\bibitem{BMMU}
	Bj\"ornberg, J. E., Mailler, C., M\"orters, P. and Ueltschi, D.:
	{Characterising random partitions by random colouring}, \emph{Electron. Commun. Probab.} \textbf{25}, (2020), paper no.\ 4.  

\bibitem{dlp}
Ding, J., Lee, J. R., and Peres, Y.: 
Cover times, blanket times, and majorizing measures, \emph{Ann. of Math.} \textbf{175} (2), (2012), no.\ 3, 1409--1471.  

\bibitem{dm2001}
    Dai, M. and Mukherjea, A.: 
    {Identification of the parameters of a multivariate 
normal vector by the distribution of the maximum}, \emph{J. Teoret. Probab.}, \textbf{14}, (2001), no.\ 1 
 267--298.  

\bibitem{fs2019a}
   Forsstr\"om, M. P. and Steif, J. E.:
   {A formula for hidden regular variation behavior for symmetric stable distributions}.
   
\bibitem{fs2019b}
   Forsstr\"om, M. P. and Steif, J. E.:
   {An analysis of the induced linear operators associated to divide and color models},
   \emph{J. Theor. Probab.}, (2020).   
   
\bibitem{fs2019c}
   Forsstr\"om, M. P. and Steif, J. E.:
   {A few surprising integrals}, \emph{Statist. Probab. Lett.}, \textbf{157}, (2020), no.\ 108635. 

\bibitem{h2005}
   Hashorva, E.:
   {Asymptotics and bounds for multivariate Gaussian tails},
   \emph{J. Theoret. Probab.}, \textbf{18}, (2005), no.\ 1, 79--97. 


\bibitem{tl2016}
  Lupu, T.:
  {From loop clusters and random interlacement to the free field},
 \emph{Ann. Probab.}, \textbf{44}, (2016), no.\ 3., 2117--2146. 

\bibitem{lw2015}                
	Lupu, T. and Werner, W.:
{A note on Ising random currents, Ising-FK, loop-soups and the Gaussian free field},
\emph{Electron. Commun. Probab.}, \textbf{21}, (2016), paper no. 13. 

\bibitem{m1972}
Markham, T. L.:
{Nonnegative matrices whose inverses are \(M\)-matrices},
\emph{Proc. Amer. Math. Soc.}, \textbf{36},  (1972), no.\ 2, 326--330. 

\bibitem{st1994}
   Samorodnitsky, G. and Murad S. Taqqu, M. S.:
   {Stable non-Gaussian random processes, Stochastic models with infinite variance}, \emph{Chapman \& Hall, New York}, (1994). 

\bibitem{s1899}
    Sheppard, W.: On the application of the theory of error to cases of normal distribution and 
normal correlation, 
\emph{Philosophical Transactions of the Royal Society of London}, Series A, Vol. 192,
 (1899), 101 -- 567.
 


 \bibitem{st2017}
  Steif, J. E. and Johan Tykesson, J.:
  {Generalized divide and color models},  
\emph{ALEA Lat. Am. J. Probab. Math. Stat.},
\textbf{16} (2), (2019), 899--955.  

 

\end{thebibliography}
\end{document}